\documentclass[11pt]{article}

\usepackage{amsthm,amsmath,amssymb}
\usepackage{amsthm,amsmath}
\usepackage[matrix,arrow]{xy}
\usepackage{graphicx}
\usepackage{enumerate,amssymb,setspace}
\usepackage{amsfonts,amsthm,amscd}
\date{}
\newtheorem{theorem}{Theorem}[section]
\newtheorem{lemma}[theorem]{Lemma}
\newtheorem{prop}[theorem]{Proposition}
\newtheorem{corollary}[theorem]{Corollary}
\newtheorem{conjecture}[theorem]{{Conjecture}}
\newtheorem{example}[theorem]{{Example}}

\newtheorem{definition}[theorem]{{Definition}}

\theoremstyle{remark}
\newtheorem{remark}[theorem]{Remark}

\makeatletter\@addtoreset{equation}{section} \makeatother

\font\sml=cmr6

\newcommand{\ocal}{\mathcal{O}}

 \def\calO{\ocal}

\def\a{\alpha} \def\be{\beta}
\def\o{\omega} 
\def\th{\theta} 
\def\vp{\varphi} \def\eps{\epsilon}

\newcommand{\dbar}{\bar\partial}
\newcommand{\ddbar}{\partial\dbar}

\def\aut{{\operatorname{aut}}}
\def\Aut{{\operatorname{Aut}}}
\def\Im{{\operatorname{Im}}}
\def\Re{{\operatorname{Re}}}

\def\max{{\operatorname{max}}}

\def\K{K\"ahler } 
\def\KE{K\"ahler--Einstein }

\def\on{\omega^n}

\def\Ric{\hbox{\rm Ric}\,}

\def\h#1{\hbox{#1}}

\def\strutdepth{\dp\strutbox}
\def\specialstar{\vtop to \strutdepth{
    \baselineskip\strutdepth
    \vss\llap{$\star$\ \ \ \ \ \ \ \ \  }\null}}
\def\marginalstar{\strut\vadjust{\kern-\strutdepth\specialstar}}
\def\marginal#1{\strut\vadjust{\kern-\strutdepth
    {\vtop to \strutdepth{
    \baselineskip\strutdepth
    \vss\llap{{ \small #1 }}\null}
    }}
    }

\def\text{\textstyle}

\def\q{\quad}

\def\b#1{\bar{#1}}

\def\PSH{\mathrm{PSH}}

\def\ra{\rightarrow}

\def\i{\sqrt{-1}}

\def\Aut{{\operatorname{Aut}}}
\def\id{{\operatorname{id}}}

\let\s=\sigma

\def\Pic{\operatorname{Pic}}
\def\sm{\setminus}

\def\exc{\operatorname{exc}}
\def\lct{\operatorname{lct}}
\def\D{\Delta}
\def\del{\partial}
\newcommand{\PP}{{\mathbb P}} \newcommand{\RR}{\mathbb{R}}
\newcommand{\QQ}{\mathbb{Q}} \newcommand{\CC}{{\mathbb C}}
 \newcommand{\NN}{{\mathbb N}}
\newcommand{\FF}{{\mathbb F}}

\def\la{\lambda}
\def\beq{\begin{equation}}
\def\eeq{\end{equation}}
\def\bpf{\begin{proof}}
\def\epf{\end{proof}}
\def\Rico{\Ric\,\!\o}

\def\eaeq{\end{aligned}}
\def\baeq{\begin{aligned}}
\def\mult{\operatorname{mult}}
\def\saldp{strongly asymptotically log del Pezzo }
\def\Saldp{Strongly asymptotically log del Pezzo }

\title{Asymptotically log Fano varieties}

\author{Ivan A. Cheltsov and Yanir A. Rubinstein}

\begin{document}
\maketitle

\begin{abstract}

Motivated by the study of Fano type varieties
we define a new class of log
pairs that we call {asymptotically log Fano varieties} and
{strongly asymptotically log Fano varieties}.
We study their
properties in dimension two under an additional assumption of log
smoothness, and give a complete classification of two dimensional
strongly asymptotically log smooth log Fano varieties.
Based on this classification
we formulate an asymptotic
logarithmic version of Calabi's conjecture for del Pezzo surfaces
for the existence of K\"ahler--Einstein edge metrics in
this regime.
We make some initial progress towards its proof by demonstrating
some existence and non-existence results, among them a generalization
of Matsushima's result on the reductivity of the
automorphism group of the pair, and results on log canonical thresholds
of pairs. 
One by-product of this study is a new conjectural
picture for the small angle regime and limit which
reveals a rich structure in the asymptotic regime, of which
a folklore conjecture 
concerning
the case of a Fano manifold with an anticanonical divisor is
a special case.

\end{abstract}

\tableofcontents

\section{Introduction}

This article draws its motivation from classification theory
of Fano type varieties in algebraic geometry on the one hand,
and the uniformization problem of \K edge manifolds in
complex differential geometry on the other hand.
Our results here contribute to both of these problems, and
also draw some connections between the two. In addition,
we relate both of these to the theory of non-compact
Calabi--Yau fibrations
and the Minimal Model Program.

\subsection{Asymptotically log Fano varieties}

A projective variety $X$ is said to be \emph{of Fano type} if
there exists an effective $\mathbb{Q}$-divisor
$$
\Delta=\sum_{i=1}^{r}a_i\Delta_i
$$
on $X$ such that the divisor
$-K_X-\Delta$ is ample and the pair $(X,\Delta)$ has at most
Kawamata log terminal singularities \cite[Lemma--Definition
2.6]{PrSh}. Fano type varieties possess very nice properties: they
are rationally connected \cite{Zhang}, they are Mori dream spaces
\cite{BCHM}, and their Cox rings have mild singularities
\cite{Brown}, \cite{GOST}. Moreover, Fano type varieties play an
important role in birational geometry: they are building blocks of
rationally connected varieties \cite{KoMo98}, they appears as exceptional
divisors of extremal contractions
%\cite[Lemma 2.8 (iii)]{PrSh},
and they behave well
under contractions \cite[Lemma 2.8]{PrSh}.

Can we classify Fano type varieties? Probably not. This problem
seems to be beyond current reach even in dimension two. One can expect that the
problem is much easier if we restrict ourself to the \emph{log
smooth} case, i.e., when $X$ is smooth and the support of
$\Delta$
is a simple normal crossing divisor.
However, this does not seem to be the
case, and the later problem seems equally hard
and is also very far from
being solved even in dimension two.

One of the early attempts at classifying pairs with such properties
is  Maeda's work. Maeda coined the term
``log Fano varieties" for log smooth pairs $(X,D)$ such
that $-K_X-D$ is ample and gave a complete classification
in dimensions two and three \cite{Maeda}.
A special family of two-dimensional Fano type varieties whose boundaries have
\emph{standard} coefficients, i.e., all $a_i$
are of the form
$\frac{m-1}{m}$ for $m\in\NN$, appeared naturally in
the work of Koll\'ar who used
them to construct
$5$-dimensional real manifolds
that carry an Einstein metric with positive constant
\cite{Kollar2005,Kollar2007}.
In a different setting, work of Tsuji, Tian, and Donaldson,
suggests to consider pairs $(X,D)$ where $X$ is itself Fano,
and $D$ is an anticanonical divisor
whose boundaries have real coefficients close to
$1$ \cite{Tsuji,Tian1994,D}.
Then the numbers $2\pi(1-a_i)$ have a natural concrete
geometrical interpretation
by considering K\"ahler metrics with positive
curvature that have edge singularities along $\Delta$, in other
words metrics modelled on a one-dimensional cone of angle $2\pi(1-a_i)$
along each `complex edge' $\Delta_i$.
Such metrics were introduced by Tian as a natural generalization of conical Riemann surfaces.
A general existence theorem
for K\"ahler--Einstein edge (KEE) metrics with a smooth divisor has been obtained by Jeffres--Mazzeo--Rubinstein \cite{JMR} and we
come back to this circle of ideas in
\S\ref{CalabiSubsec}--\ref{ExistenceSubsec}.

The present work draws its motivation from all three of these geometric
settings: the asymptotic classes we introduce next contain as special cases
these previously studied geometries.

\begin{definition}
\label{definition:log-Fano} We say that a pair $(X,D)$ consisting
of a
projective variety $X$ with $-K_X$ $\QQ$-Cartier and a
divisor $D=\sum_{i=1}^rD_i$ (where the $D_i$ are distinct
$\QQ$-Cartier
prime Weyl divisors)
on $X$ is {\it (strongly) asymptotically log Fano} if the log pair
$(X,(1-\beta_i)D_i)$ has Kawamata log terminal singularities, and
the divisor $-K_X-\sum_{i=1}^r(1-\beta_i)D_i$ is ample for (all)
sufficiently small $(\beta_1,\ldots,\beta_r)\in(0,1]^r$.
\end{definition}

In the two dimensional case, we also refer to such
pairs as {\it (strongly) asymptotically log del Pezzo.}
Note that both definitions (asymptotically log Fano and strongly
asymptotically log Fano) coincide if $D$ consists of a single
component. This is not the case when $D$ is reducible.

For the rest of this article we restrict without further mention
to the (already challenging) log smooth case, i.e., when
$X$ is smooth and $D$ has simple normal crossings.

\subsection{Classification results in dimension two}

In this article we classify all strongly asymptotically log del
Pezzo surfaces, i.e., we explicitly describe all pairs $(S,C)$
consisting of a smooth surface $S$ and a simple normal crossing
curve $C$ on $S$ such that $(S,C)$ is strongly asymptotically log
del Pezzo. We believe many of the results and techniques
presented should also
be useful for classifying all asymptotically log del Pezzo surfaces
in the future.
Our main classification result is as follows.

\begin{theorem}
\label{theorem:main} Let $S$ be a smooth complex surface. Let $C=C_1+\ldots C_r$ be a simple normal
crossing divisor on $S$, with each of the $C_i$ smooth. Then $(S,C)$ is
a strongly asymptotically log del Pezzo surface if and only if it is one
of the pairs listed in Theorem \ref{theorem:main-1} (when $r=$1)
or Theorem \ref{theorem:main-2-3-4} ($r\ge$ 2).
\end{theorem}

This generalizes the classical result of
Castelnuovo, Enriques and del Pezzo for the case
with no boundary \cite{delPezzo,Hitchin},
as well as its generalization to the logarithmic
setting by Maeda \cite{Maeda} who classified
all pairs $(S,C)$ with $-K_S-C$ ample.

The classification part (`only if') of the proof occupies
Sections 2 ($r=1$) and 3 ($r\ge2$). The first several steps
are to obtain useful topological and cohomological
restrictions on the boundary curve. For instance, $C$
has genus at most one, and when it is elliptic it
must be anticanonical, $r$ must be 1, and $S$ must be del Pezzo
(Lemmas \ref{lemma:log-del-Pezzo-rational-or-elliptic-1}
and \ref{lemma:log-del-Pezzo-rational-or-elliptic}).
Thus, we may assume that $C\not\sim-K_S$
and that the $C_i$ are rational. Then $C^2\le -2$,
i.e., $C$ `traps' the negative curvature portion of $-K_S$
(Lemma \ref{lemma:log-del-Pezzo-rational-curves-not-del-Pezzo-1}).
In the same token, no other rational curve may have self-intersection
less than $-1$ (Lemma \ref{lemma:log-del-Pezzo-rational-curves-1}),
reflecting the fact that the
curvature should morally be positive outside of $C$.
But $-1$-curves are indeed allowed away from $C$ and
an important task is to understand their geometry
and configuration relative to $C$.
Lemma \ref{lemma:log-del-Pezzo-blow-down-1} shows
that such curves come in two types: disjoint from
$C$ or intersecting it transversally at exactly one point.
Motivated by this observation we say a pair is {\it minimal }
if it contains no $-1$-curves of the second type.
Lemmas \ref{lemma:log-del-Pezzo-minimal-good-small-Picard-r-1}
and \ref{lemma:minimalsmallPic}
show that minimality implies the Picard group
is `small', namely, of rank at most 2.
The case $r\ge2$ relies on some general results
(proved in \S\ref{subsection:ALDP})
on the combinatorial and cohomological structure of
the boundary that hold also in the asymptotic
(and not necessarilly strongly asymptotic)
regime.
Thus,
we perform an induction on $\h{rk}(\Pic(S))$
by successively contracting $-1$-curves;
the observation that makes this possible
is that when $-1$-curves of the first
type are contracted the resulting pair is still
log smooth and strongly asymptotically del Pezzo
(Lemmas \ref{lemma:log-del-Pezzo-blow-down-minimally-1},
\ref{lemma:log-del-Pezzo-blow-down-minimally},
and
\ref{lemma:log-del-Pezzo-blow-down-minimally-2}).
An additional complication
in the case $r\ge2$ is that the blown-down
$-1$-curve could be a component of the boundary.
According to Lemma \ref{lemma:strongly-non-strongly}
such a curve must be at the `tail': it cannot
intersect two boundary components.
Then Lemma \ref{lemma:minimalsmallPic} guarantees
the inductive step can still be carried out.
Once this induction has been carried out
all that remains
is to classify all pairs with $\h{rk}(\Pic(S))\le 2$
(Lemmas \ref{lemma:log-del-Pezzo-Picard-rank-two-1}
and \ref{lemma:log-del-Pezzo-Picard-rank-two}).

The second part of the proof of Theorem \ref{theorem:main}
consists of the verification
that each pair appearing in the lists of Theorems
\ref{theorem:main-1} and \ref{theorem:main-2-3-4}
is strongly asymptotically log del Pezzo
(\S\ref{VerifSubsec}).
Instead of checking each case separately,
we approach this straightforward task slightly more systematically
by first reformulating those
two lists in a unified list (Theorem \ref{theorem:4-cases})
according to the positivity of
the logarithmic anticanonical bundle $-K_S-C$---this is
discussed in detail in the next paragraph.
When this bundle is trivial or ample the verification
is then immediate. In the remaining two cases (big but not ample,
and nef but not big) we verify case by case.

The classification theorem has a number of corollaries, but we
state here only the most obvious one.

\begin{corollary}
Let $(S,C)$ be a log smooth \saldp pair. Then
$C$ contains at most four components.
\end{corollary}

It would be interesting to
find a similar bound in all dimensions.
In the simpler log Fano setting of Maeda, a pair $(X,D)$ induces by restriction
a log Fano pair of one dimension lower, and so
by induction the number of components is bounded
by $\dim X$ \cite[Lemma 2.4]{Maeda}.

The classification of \saldp surfaces according
to the positivity of the logarithmic anticanonical bundle
just mentioned plays a crucial role also in other parts
of this article and so we now state it precisely.
We distinguish between four mutually exclusive classes.
Class $\mathrm{(\aleph)}$: $S$ is del Pezzo and $C\sim-K_S$;
class $\mathrm{(\beth)}$: $C\not\sim-K_S$ and $(K_S+C)^2=0$;
class $\mathrm{(\gimel)}$: $-K_S-C$ is big but not ample;
class $\mathrm{(\daleth)}$: $-K_S-C$ is ample.

\begin{theorem}
\label{theorem:4-cases}
\Saldp pairs, whose list appears in
Theorems \ref{theorem:main-1} and \ref{theorem:main-2-3-4},
are classified according to the positivity properties
$\mathrm{(\aleph)}$, $\mathrm{(\beth)}$, $\mathrm{(\gimel)}$,
and $\mathrm{(\daleth)}$ as follows:
\begin{itemize}
\item[$\mathrm{(\aleph)}$] $(S,\sum_{i=1}^{r}C_{i})$ is one of
$\mathrm{(I.1A})$, $\mathrm{(I.4A})$, $\mathrm{(I.5.m})$,
$\mathrm{(II.1A})$, $\mathrm{(II.4})$, $\mathrm{(II.5A.m})$,
$\mathrm{(II.8.m})$, $\mathrm{(III.1})$, $\mathrm{(III.2})$, $\mathrm{(III.4.m})$ or
$\mathrm{(IV})$,

\item[$\mathrm{(\beth)}$] $(S,\sum_{i=1}^{r}C_{i})$ is one of
$\mathrm{(I.3A})$, $\mathrm{(I.4B})$, $\mathrm{(I.9B.m})$,
$\mathrm{(II.2A.n})$, $\mathrm{(II.2B.n})$, $\mathrm{(II.3})$,
$\mathrm{(II.6A.n.m})$, $\mathrm{(II.6B.n.m})$,
$\mathrm{(II.7.m})$, $\mathrm{(III.3.n})$ or
$\mathrm{(III.5.n.m})$,

\item[$\mathrm{(\gimel)}$] $(S,\sum_{i=1}^{r}C_{i})$ is one of
$\mathrm{(I.6B.m})$, $\mathrm{(I.6C.m})$, $\mathrm{(I.7.n.m})$,
$\mathrm{(I.8B.m})$,
$\mathrm{(I.9C.m})$, $\mathrm{(II.5B.m})$ or
$\mathrm{(II.6C.n.m})$,

\item[$\mathrm{(\daleth)}$] $(S,\sum_{i=1}^{r}C_{i})$ is one of
$\mathrm{(I.1B})$, $\mathrm{(I.1C})$, $\mathrm{(I.3B})$,
$\mathrm{(I.2.n})$, $\mathrm{(I.4C})$, $\mathrm{(II.1B})$ or
$\mathrm{(II.2C.n})$.
\end{itemize}
\end{theorem}

The verification of this list is an elementary corollary
of Theorems \ref{theorem:main-1} and \ref{theorem:main-2-3-4}
and appears in \S\ref{PositivitySubsec}.
It can be seen as a generalization of two
previously known classes. Class $(\daleth)$ is Maeda's classical
classification of what he coined as `log del Pezzo surfaces'
\cite{Maeda}.
On the other hand, the class $\mathrm{(\aleph)}$
is simply the classical class of del Pezzo surfaces
together with the information of a simple normal crossing
anticanonical curve but its explicit (and very elementary)
classification seems to appear here for the first time.
The classes $\mathrm{(\beth)}$ and $\mathrm{(\gimel)}$
are new.

\subsection{An asymptotic logarithmic version of Calabi's conjecture}
\label{CalabiSubsec}

In 1990, in what became known as the resolution
of Calabi's conjecture for del Pezzo surfaces,
Tian gave a complete classification
of those complex surfaces that admit a smooth KE metric of positive
curvature \cite{Ti90}.
In light of Theorem \ref{theorem:main} it
is therefore very natural and tempting to hope for
a counterpart for strongly asymptotically log
del Pezzo surfaces. One of the main goals of this article
is to formulate such a conjecture as well as prove key
parts of it. As might be expected, the situation in the singular
setting is quite a bit more complex and we intend to pursue
other aspects of this conjecture in future work.

As we now explain in detail, a surprisingly accurate guide to this uniformization problem is the positivity
classification of Theorem \ref{theorem:4-cases}.

Pairs of class $\mathrm{(\aleph)}$ are the best understood,
since according to a result of Berman the Tian invariant
of the pair is then bigger than $\frac n{n+1}$, which
subsequently implies by the work of Jeffres--Mazzeo--Rubinstein
(Theorem \ref{theorem:Mazzeo-Rubinstein-Jeffres} below, cf. \cite[Corollary 1.5]{MR})
that the pair admits KEE metrics for all small angles.
We generalize Berman's result in several ways by
obtaining a general bound on the global log canonical
threshold in a possibly singular and/or degenerate setting
(Proposition \ref{theorem:handy-proposition}). This gives
an algebraic proof of the aforementioned estimate
due to Berman
for the class $\mathrm{(\aleph)}$
with explicit (but far from optimal) lower bounds on the largest angle possible
in dimensions two and
three (Proposition \ref{theorem:log-del-Pezzo-alpha-1}).

The uniformization problem is thus reduced to understanding
the existence problem for pairs of classes $(\beth),
(\gimel),$ and $(\daleth)$.

As a first guide, we investigate the asymptotic behavior
in the small-angle limit
of Tian's invariant $\alpha(X,(1-\beta)D)$, also refereed
to as the global log canonical
threshold of the pair $(X,D)$ (see \S\ref{lctSubsec} for definitions).

\begin{theorem}
\label{theorem:log-del-Pezzo-alpha}
Assume $(S,C)$ is asymptotically log del Pezzo with $C$ smooth and irreducible.
Then
$$
\lim_{\beta\to 0^+}\alpha(S,(1-\beta)C)
=
\left\{\aligned%
&1\qquad\h{class $(\aleph)$},\\
&1/2\quad\h{class $(\beth)$},\\
&0\qquad\h{class $(\gimel)$ or $(\daleth)$}
\endaligned
\right.
$$
\end{theorem}

This gives an indication that the existence theory might,
in fact, depend on the positivity classification.
In fact, we conjecture that the positivity classification
completely determines the existence problem.

\begin{conjecture}
\label{UniformizationConj}
Suppose that $(S,C)$ is strongly asymptotically log
del Pezzo with $C$ smooth and irreducible.
Then $S$ admits K\"ahler--Einstein edge metrics with
angle $\beta$ along $C$ for
all sufficiently small $\beta$ if and
only if $(K_{S}+C)^2=0$, i.e., $(S,C)$ is of class
$(\aleph)$ or $(\beth)$.
\end{conjecture}

To put this conjecture in appropriate context and give
perhaps more striking motivation for its validity
we begin by noting that $0,1/2$ and $1$ are the Tian
invariants of $\PP^n, n\ra\infty,
\PP^1$, and $\PP^0$, respectively. It is then tempting to think
of $1/2$ as the Tian invariant of certain generic rational fiber.
Motivated by this we prove the following
structure theorem for surfaces of class~$\mathrm{(\beth)}$.

\begin{prop}
\label{proposition:conic-bundle} If
$(K_{S}+\sum_{i=1}^{r}C_{i})^2=0$, then the linear system
$|-(K_{S}+\sum_{i=1}^{r}C_{i})|$ is free from base points and
gives a morphism $S\to\mathbb{P}^1$ whose general fiber
is $\mathbb{P}^1$, and every reducible fiber consists of exactly
two components, each a $\PP^1$.
\end{prop}

Thus, surfaces of class $(\beth)$ are conic bundles,
and the boundary $C$ intersects each generic fiber at
two points. This gives strong motivation for
the `if' part of
Conjecture \ref{UniformizationConj} because it suggests
what the small-angle limit of the purported KEE metrics
on pairs of class $(\beth)$ could be:

\begin{conjecture}
\label{LimitConj}
Let $(S,C,\o_\be)$ be KEE pairs of class $(\aleph)$
or $(\beth)$.
Then $(S,C,\o_\be)$ converges in an appropriate sense to a
a generalized KE metric $\o_\infty$ on $S\setminus C$ as $\beta$ tends to zero.
In particular, $\o_\infty$ is a Calabi--Yau metric
in case $(\aleph)$, and a cylinder along each generic fiber
in case $(\beth)$.
\end{conjecture}

The generalized KE metrics alluded to in the conjecture
are related to  
metrics studied by Song--Tian on elliptic fibrations \cite{SongTian},
however there are some important differences. 
We postpone an in-depth discussion of this to a sequel.

This conjecture can be generalized to any dimension, and is
perhaps better understood in such a more general context. To that
end we first note that Proposition \ref{proposition:conic-bundle}
is a very special and explicit case of a much more general result
that is a direct corollary of deep results of Kawamata and
Shokurov.

\begin{theorem}
\label{theorem:KS} Suppose $(X,D)$ is asymptotically log Fano and
$-K_X-D$ is not big. Then $|-n(K_{X}+D)|$ is base point free for $n\gg
1$ and gives a morphism $\phi\colon X\to Z$ whose general fiber
$F$ is a Fano type variety. Moreover, $D\vert_{F}\sim_{\mathbb{Q}}
-K_{F}$ and $(F,D\vert_{F})$ is asymptotically log Fano.
Furthermore, if $(X,D)$ is strongly asymptotically log Fano, then
$F$ is a Fano variety with Kawamata log terminal singularities,
and $(F,D\vert_{F})$ is strongly asymptotically log Fano.
\end{theorem}

\begin{proof}
By Kawamata--Shokurov's Basepoint-free Theorem
\cite[Theorem~3.3]{KoMo98} the linear system $|-n(K_{X}+D)|$ is
base point free for $n\gg 1$. Let $\phi\colon X\to Z$ be a
morphism given by it, and let $F$ be its general fiber. If
$(X,\sum_{i=1}^{r}(1-\beta_{i})D_i)$ is Kawamata log terminal and
$-K_{X}-\sum_{i=1}^r(1-\beta_i)D_i$ is ample, then
$(F,\sum_{i=1}^{r}(1-\beta_{i})D_i\vert_{F})$ has at most Kawamata
log terminal singularities and
$$
-\Big(K_{F}+\sum_{i=1}^r(1-\beta_i)D_i\vert_{F}\Big)
\sim_{\mathbb{R}}
-\Big(K_{X}+\sum_{i=1}^r(1-\beta_i)D\Big)\Big\vert_{F}
$$
is ample. Thus, $(F,D\vert_{F})$ is asymptotically log Fano. Note
that by using adjunction
$D\vert_{F}\sim_{\mathbb{Q}} -K_{F}$, because $F$ is a fiber
of $\phi$ and $\phi$ is given by $|-n(K_{X}+D)|$.

If $(X,D)$ is strongly asymptotically log Fano the same
argument shows that so is $(F,D\vert_{F})$.
Moreover, then
$$
-K_{F}\sim -K_{X}\vert_{F}\sim_{\mathbb{Q}} D\vert_{F}
\sim_{\mathbb{R}}\frac{1}{\beta}\beta
D\vert_{F}\sim_{\mathbb{R}}\frac{1}{\beta}(-K_{X}-D-\beta
D)\vert_{F}=-(K_{X}+\sum_{i=1}^r(1-\beta)D)\vert_{F}
$$
for small $\beta\in(0,1]$, which implies that $-K_{F}$ is ample,
i.e., $F$ is a Fano variety.
\end{proof}

\begin{corollary}
\label{corollary:KS} Let $X$ be a smooth variety, let
$D$ be smooth and irreducible Weil divisor on
$X$. Suppose $(X,D)$ is asymptotically log Fano and
$-K_X-D$ is not big. Then $|-n(K_{X}+D)|$ is base point free for $n\gg
1$ and gives a morphism $\phi\colon X\to Z$ whose general fiber
$F$ is a smooth Fano variety with $D\vert_{F}\in|-K_{F}|$.
\end{corollary}

Therefore, we conjecture:

\begin{conjecture}
\label{GeneralConj}
Suppose that $(X,D)$ is strongly asymptotically log
Fano manifold with $D$ smooth and irreducible.
Let
$
\kappa:=\inf\{\NN\ni k\le\dim X \,:\, (K_X+D)^k=0\}.
$
\hfill\break
(i) There exist no KEE metric with small $\be$ if
$\kappa=\infty$.
\hfill\break
(ii) Suppose that $(K_X+D)^{\dim X}=0$.
Then there exist KEE metrics $\o_\be, \be\in(0,\eps)$ on $(X,D)$
for some $\eps>0$.
\hfill\break
(iii) As $\be$ tends to zero $(X,D,\o_\be)$
converges in an appropriate sense to a
generalized KE metric $\o_\infty$ on $X\setminus D$
that is Calabi--Yau along its generic $(\dim X+1-\kappa)$-dimensional
fibers.
\hfill\break
(iv) Furthermore,
\beq
\label{AlphaConjEq}
\lim_{\beta\to
0^+}\alpha(X,(1-\beta)D)=
\left\{\aligned%
&1\qquad \mathrm{if}\ K_{X}+D\sim 0,\\
&\mathrm{min}\{1,\alpha(X,[-K_{X}-D]),\alpha(D)\}
\qquad \!\mathrm{if}\ 0\not\sim-K_{X}-D\ \mathrm{is\ not\ big},\\
&0\qquad  \mathrm{if}\ -K_{X}-D\ \mathrm{is\ big}.\\
\endaligned
\right.
\eeq
\end{conjecture}

Conjecture \ref{GeneralConj} (iii) is itself a generalization
of a folklore conjecture in K\"ahler geometry
for the case $\kappa=1$ mentioned, e.g., by Donaldson \cite[p. 76]{D},
saying that $X\setminus D$ equipped with the Tian--Yau
metric \cite{TianYau} should be a limit of KEE metrics
on $(X,D)$ when $X$ is Fano and $D\in|-K_X|$.
As mentioned earlier, Conjecture \ref{GeneralConj} (ii)
holds when $\kappa=1$. In Propostion \ref{theorem:log-del-Pezzo-alpha-1}
we further give explicit bounds on $\eps$ when $\dim X\in\{2,3\}$
and $\kappa=1$.

Since the work of Hitchin, Kobayashi, and many others,
a standard condition for the existence of canonical metrics
that can be described as zeros of an infinite-dimensional
moment map is some sort of `stability' condition.
How, then, does Conjecture \ref{GeneralConj} fit into this scheme?
The condition $(K_X+D)^{\dim X}=0$ hardly looks at first
like a stability condition. Perhaps one way to motivate it is
to conceive the non-compact Calabi--Yau fibration of
Conjecture \ref{GeneralConj} (ii)
as a KEE metric itself, only with $\be=0$. In
case (i) such a smooth (and hence non-compact)
limit does not exist since too much `positivity'
is still remaining, and
so the small angle regime, which would otherwise
be a metric `perturbation' of that limit, should not
exist either. Thus, the existence of the Calabi--Yau
degeneration provides the necessary `stability' in this situation,
at least conjecturally. An obvious advantage of
the existence criterion of Conjecture \ref{GeneralConj}
is that it is very explicit as opposed
to logarithmic K-stability which in general seems
hard to check.

We prove Conjecture \ref{GeneralConj} (iv) except for the middle case
which we only prove in dimension two
(Propositions \ref{theorem:log-del-Pezzo-alpha-2},
\ref{theorem:log-del-Pezzo-alpha-3},
and \ref{theorem:log-del-Pezzo-alpha-1}).
We refer to \S\ref{LimitAlphaSubsec} for one technical issue relevant
to the definition of the invariants appearing in \eqref{AlphaConjEq}.

Finally, we make some progress towards
Conjecture \ref{GeneralConj} (i) and (ii) in dimension two,
i.e., Conjecture \ref{UniformizationConj}, that we describe next.

\subsection{Existence and non-existence results in the asymptotic regime}
\label{ExistenceSubsec}

Matsushima's theorem \cite{Matsushima} implies that the \KE metric
is the most aesthetically pleasing one since it exhibits
the maximal symmetry possible: every one-parameter subgroup
of automorphisms of the complex structure $\Aut(X)$
can be realized as the complexification of a one-parameter subgroup of
isometries of the KE metric. This has a natural generalization
to the edge setting by considering the automorphism group
of the pair $\Aut(X,D)$, i.e., elements of $\Aut(X)$ that map $D$ to itself.

\begin{theorem}
\label{reductiveThm}
Let $(X,D,g)$ be a KEE manifold. Then
$\Aut_0(X,D)=\h{\rm Isom}_0(X,g)^\CC$. In particular,
$\Aut_0(X,D)$ is reductive.
\end{theorem}

Here $\h{\rm Isom}_0(X,g)^\CC$ denotes the complexification
of the identity component of the isometry group of $(X,g)$,
while $\Aut_0(X,D)$ denotes the identity component of $\Aut(X,D)$. 
Theorem \ref{reductiveThm} is proved in Section \ref{ReductiveSec}, 
using the asymptotic structure of solutions to linear elliptic
equations with edge degeneracies in the sense of Mazzeo
\cite{Mazzeo} as developed in the complex codimension one setting in \cite{JMR}.
After posting this article we were informed that
in the special case $X$ is itself Fano
a more general result than Theorem \ref{reductiveThm} 
has been obtained in
\cite[Theorem 4]{CDSun} and
\cite[Lemma 6.3]{Tian2013} that appeared several months prior
to our work. While the proofs in these articles quite likely can be extended to
the case $X$ is not Fano, these proofs are quite
different from our approach. 

In the smooth world, Matsushima's criterion is often considered
as a rather coarse obstruction to existence. Nevertheless,
in the asymptotic regime with its much richer variety
of cases, such a tool proves to be quite useful.

\begin{theorem}
\label{NoKEEThm} The following strongly asymptotically log del
Pezzo pairs listed in Theorem \ref{theorem:main-1}
do not admit KEE metrics for
sufficiently small $\be$: $\mathrm{(I.1C})$, $\mathrm{(I.2.n})$
with any $n\ge 0$, $\mathrm{(I.6C.m})$ with any $m\ge 1$,
$\mathrm{(I.7.n.m})$ with any $n\ge 0$ and $m\ge 1$,
$\mathrm{(I.6B.1})$, $\mathrm{(I.8B.1})$ and $\mathrm{(I.9C.1})$.
\end{theorem}

This proves part of the `only if' direction
of Conjecture \ref{UniformizationConj}.
It is proven in Section \ref{AutomorphismSec}
by 
computing the automorphisms groups
of pairs of classes $(\gimel)$ and $(\daleth)$.
We also supply
further evidence for the converse direction of the conjecture
by showing that all pairs of class $(\beth)$ have reductive automorphism groups
(Theorem \ref{BethReductiveGroupsThm}).

Next, we turn to the existence %sufficient
part of Conjecture \ref{UniformizationConj}.
Our main tool here is the following existence theorem
that is a special case of \cite[Theorem 2, Lemma 6.13]{JMR}. The invariant
$\a_G(S,(1-\be)C)$ is the $G$-invariant Tian invariant of the
pair $(S,(1-\be)C)$ with respect to
the \K class $[-K_S-(1-\be)C]$
(see Definition \ref{definition:global-threshold-G}).

\begin{theorem}
\label{theorem:Mazzeo-Rubinstein-Jeffres} Let $(S,C)$ be a
strongly asymptotically log del Pezzo surface with $C$
smooth and irreducible. Suppose that
$G\subset \Aut(S)$ is a finite group and that
$\a_G(S,(1-\be)C)>2/3$. Then
there exists a \KE edge metric
with positive Ricci curvature and with angle $2\pi\be$
along $C$.
\end{theorem}

We apply this to prove the following existence theorem for pairs
of class $(\beth)$ giving the first construction of
KEE metrics of positive curvature and of small angle
outside of the classical class $(\aleph)$.

\begin{theorem}
\label{KEEBethThm}
There exist strongly asymptotically log del
Pezzo pairs of type (I.3A), (I.4B), and (I.9B.5)
(listed in Theorem \ref{theorem:main-1})
that admit KEE metrics for all sufficiently small $\be$.
\end{theorem}

This result is proven using computations of the Tian
invariant of these pairs
(Subsection \ref{BethExistenceSubsec}).
In the cases (I.3A), (I.4B)
the pair possesses certain discrete symmetry that
allows using representation theoretic arguments
coupled with Shokurov's connectedness principle
for log canonical loci to conclude that in fact Tian's invariant
equals 1 for all $\be\in(0,1]$. The case (I.9B.5)
is somewhat more delicate since then $S$ varies in a moduli space.
We choose the Clebsch cubic surface in
that space and again are able to show that
the Tian invariant equals 1 for all $\be\in(0,1]$.
We also compute the Tian invariant of more general
cubic surfaces with an Eckardt point and
show that without symmetry one cannot apply
the existence result of Theorem \ref{theorem:Mazzeo-Rubinstein-Jeffres}.
This last computation (Proposition \ref{lemma:cubic-surface})
generalizes a result from the smooth setting \cite[Theorem 1.7]{Ch07b}.

Using log slope stability
Li--Sun proved that the pairs (I.1B) and (I.3B)
admit no KEE metrics for small $\be$ \cite[\S3]{LiSun}.
It is possible to apply arguments similar to theirs to prove non-existence
results for other pairs of class $(\gimel)$ and $(\daleth)$ but for the sake
of brevity we postpone this discussion,
along with further existence results for class~$(\beth)$,
to a separate article.

\subsection{Conventions}

Let us describe notation and basic results
that will be used throughout the article.

By a curve in an algebraic variety $X$ we mean an irreducible
reduced subvariety of dimension one. Occasionally, we allow curves
to be reducible (but we always assume that they are reduced). For
a curve $C$ on a smooth surface $S$, we define its arithmetic
genus $p_a(C)$ by
\beq \label{ArithGeomGenus} p_a=h^1(O_{C}). \eeq
Then $2p_a(C)-2=K_{S}\cdot C+C^2$ by
\cite{GH}.
When $C$ is smooth $p_a(C)$ equals the genus of $C$, $g(C)$.
If $C$ is an
irreducible curve on a smooth surface $S$, then
by applying adjunction one verifies that
\beq
\label{SuchRationalIsSmooth} \h{$C\cong\PP^1$\ \ if and only if\ \ $p_a(C)=0$}
\eeq
This can be quite handy.

By $\sim$ we assume rational equivalence of Weil divisors or
Cartier divisors (or their classes in $\mathrm{Cl}(X)$ and
$\mathrm{Pic}(X)$, respectively)
except in Section \ref{ReductiveSec} where $\sim$
stands for equality in the sense of complete asymptotic expansions
as in \cite{JMR}. By $\sim_{\mathbb{Q}}$ we assume
$\mathbb{Q}$-rational equivalence of $\mathbb{Q}$-divisors, i.e.,
$D_1\sim_{\mathbb{Q}} D_2$ if and only if $nD_1\sim nD_2$ for some non-zero
integer $n$ such that $nD_1$ and $nD_2$ are integral divisors. By
$\mathbb{Q}$-Cartier and $\mathbb{R}$-Cartier divisors we mean
elements in $\mathrm{Pic}(X)\otimes\mathbb{Q}$ and
$\mathrm{Pic}(X)\otimes\mathbb{R}$, respectively. By
$\sim_{\mathbb{R}}$ we assume $\mathbb{Q}$-rational equivalence of
$\mathbb{R}$-divisors, i.e., $D_1-D_2$ is a sum with real
coefficients of $\mathbb{Q}$-Cartier divisors that are
$\mathbb{Q}$-rationally equivalent to zero.

For two divisors $D_1$ and $D_2$, we write $D_1\equiv D_2$ (and
say that $D_1$ and $D_2$ are numerically equivalent) iff $D_1-D_2$
is $\mathbb{R}$-Cartier divisor such that $(D_1-D_2). C=0$ for
every curve $C\subset X$. Vice versa, we say that two curves $C_1$
and $C_2$ on $X$ are numerically equivalent iff $D. C_1=D.
C_2$ for every $\mathbb{R}$-Cartier divisor $D$ on $X$. Similarly,
we define numerical equivalence of $1$-cycles (with real or
rational coefficients) on $X$. We denote the real vector space of
$1$-cycles modulo numerical equivalence by $\mathrm{N}_1(X)$. By
the cone of curves or the Mori cone of $X$ we assume the cone in
$\mathrm{N}_1(X)$ generated by curves in $X$. We denote the Mori
cone of $X$ by $\mathrm{NE}(X)$. By $\overline{\mathrm{NE}(X)}$ we
denote its closure.

Recall that a $\mathbb{Q}$-Cartier $\mathbb{Q}$-divisor $D$ is
called ample if there exists positive integer $n$ such that $nD$
is a very ample Cartier divisor. By Kleiman's criterion, $D$ is
ample if and only if $D$ is positive on
$\overline{\mathrm{NE}(X)}$ (and this in turn is equivalent to the
differential geometric notion of positivity of a class).
The latter can be used as a
definition of ampleness for $\mathbb{R}$--Cartier
$\mathbb{R}$-divisors. Note that in the case of surfaces, the
ampleness of a $\mathbb{Q}$-Cartier $\mathbb{Q}$-divisor $D$ is
equivalent to
\beq
\label{NMEq} \h{$D^2>0$ and $D. C>0$}
\eeq
for every curve $C\subset X$. So we can use the latter condition
as another definition of ampleness for $\mathbb{R}$--Cartier
$\mathbb{R}$-divisors on surfaces. This criterion-definition is
very handy for surfaces: if $D$ is an ample $\mathbb{R}$-Cartier
divisor on a smooth surface $S$, then
\beq
\label{pushforward}
\h{$\pi_{\star}(D)$ is an ample $\mathbb{R}$-Cartier divisor}
\eeq
for every birational morphism $\pi\colon S\to s$ such that $s$ is
a smooth surface.

Recall that a $\mathbb{Q}$-divisor $D$ is called big if
$h^{0}(\mathcal{O}_{X}(nD))$ grows as $O(n^{\mathrm{dim}(X)})$ for
$n\gg 1$ such that $nD$ is an integral divisor. One can show that
$D$ is big if and only if it is a sum of an effective divisor and
an ample divisor. For $\mathbb{R}$-divisors this can be used as a
definition of bigness.

Recall that a divisor $D$ is effective if $D$ is a finite linear
combination of prime Weil divisors with non-negative coefficients,
and that $h^{0}(\mathcal{O}_{X}(-D))=0$ for every non-zero
effective Weil divisor $D$. An $\mathbb{R}$-Cartier
$\mathbb{R}$-divisor $D$ is called nef (a shortcut for numerically
effective) if $D. C\ge 0$ for every curve $C\subset X$.
Thus,
\beq
\label{EffMinusNeffZero} \h{ if $D$ is effective and $-D$ is
nef, then $D$ is a zero divisor.}
\eeq

For each $n\ge0$,
denote by
\beq
\label{FnZn}
\FF_n
\eeq
the unique rational ruled surface
whose Picard group has rank two and
contains a unique (if $n>0$) smooth rational curve of self-intersection $-n$.
We denote this curve by $Z_n$, and by $F$ we denote
an irreducible smooth rational curve such that $F^2=0$ and $F.Z_n=1$.
If $n=0$ when we refer to $Z_0$ and $F$ we
intend that each is a fiber of a different projection to $\PP^1$.
Such a surface can be constructed, e.g., as a toric variety
or as a ruled surface and \cite[Chapter 5, \S2]{Har77}
and applying adjunction yields
\beq
\label{KFnEq}
-K_{S}\sim 2Z_n+(n+2)F
\eeq
Recall that every smooth irreducible curve in $|Z_n+nF|$
(a `zero section') intersects each fiber transversally at a single point
and does not intersect the `infinity section' $Z_n$.
Any curve $C$ on $\FF_n$ satisfies $C\sim aZ_n+bF$
with $a,b\in\NN\cup\{0\}$. This, combined with
\eqref{NMEq}, implies
\beq\label{ampleFn}
\h{$C$ is ample if and
only if $a>0$ and $b>na$,}
\eeq
and furthermore,
\beq\label{irreducibleFn}
\h{$C$ is an
irreducible curve
if and only if $C=Z_n$ or $b\ge na\ge0$.}
\eeq

The classification of rational surfaces \cite[p. 520]{GH}
implies that
\beq\label{MinusOneCurveExistence}
\h{every rational surface with $\h{rk}(\h{Pic})>2$ contains a $-1$-curve},
\eeq
and that
\beq
\label{RationalClassifPicAtMostTwo}
\h{a rational surface with $\h{rk}(\h{Pic})\le 2$
is either $\PP^2$ or $\FF_n, \,n\ge0$}.
\eeq

We denote by $\mathbb{G}_a$ the additive
group $(\mathbb{C},\,+\,)$, by
$\mathbb{G}_m$ the multiplicative
group $(\mathbb{C}^{\star},\,\cdot\,)$, and by
$\mu_n$ the finite group of order $n$.

Finally, if $G$ is a graph with vertex set $V$ and
edges $E$, the dual graph of $G$ refers to the graph
whose vertex set is $E$ and whose edge set is $V$,
namely if $v\in V, e_1,e_2\in E$ and $v\in e_1\cap e_2$
then $e_1$ and $e_2$ are connected in the dual graph
by $v$. A graph is a cycle if for some $\NN\ni k\ge2$
$E=\{e_1,\ldots,e_k\}, V=\{v_1,\ldots,v_k\}$,
and $e_i\cap e_{i+1}=v_i$ with $e_{k+1}:=e_1$.
A tree is a graph that contains no cycles.
A chain is a connected tree with no three edges intersecting.

\subsection*{Acknowledgments}
The bulk of this work took place during visits to Stanford University
in Spring and Summer 2012, and to CIRM, Universit\`a di Trento, in
Summer 2013. We are grateful to CIRM for its financial support
through the Research in Pairs program.
The NSF supported this research through grant DMS-1206284.
IAC was also supported by AG Laboratory SU-HSE,
RF government grant ag.11.G34.31.0023.
YAR is also supported by a Sloan Research Fellowship.

\section{Asymptotically log del Pezzo surfaces with smooth connected boundary}
\label{section:r-1}

The following theorem gives complete classification in the case of
a single boundary component.

\begin{theorem}
\label{theorem:main-1} Let $S$ be a smooth surface (the surface),
and let $C$ be an irreducible smooth curve on $S$ (the boundary
curve). Then $-K_S-(1-\beta)C$ is ample for all sufficiently
small $\beta>0$ if and only if $S$ and $C$ can be described as
follows:
\begin{itemize}
\item [$\mathrm{(I.1A})$] $S\cong\mathbb{P}^2$, and $C$ is a smooth cubic elliptic curve,%
\item [$\mathrm{(I.1B})$] $S\cong\mathbb{P}^2$, and $C$ is a smooth conic,%
\item [$\mathrm{(I.1C})$] $S\cong\mathbb{P}^2$, and $C$ is a line,%
\item [$\mathrm{(I.2.n})$] $S\cong\mathbb{F}_{n}$ for any $n\ge 0$, and $C=Z_n$,%
\item [$\mathrm{(I.3A})$] $S\cong\mathbb{F}_1$, and $C\in |2(Z_1+F)|$,%
\item [$\mathrm{(I.3B})$] $S\cong\mathbb{F}_1$, and $C\in |Z_1+F|$,%
\item [$\mathrm{(I.4A})$] $S\cong\mathbb{P}^1\times\mathbb{P}^1$, and $C$ is a smooth elliptic curve of bi-degree $(2,2)$,%
\item [$\mathrm{(I.4B})$] $S\cong\mathbb{P}^1\times\mathbb{P}^1$, and $C$ is a smooth  rational curve of bi-degree $(2,1)$,%
\item [$\mathrm{(I.4C})$] $S\cong\mathbb{P}^1\times\mathbb{P}^1$, and $C$ is a smooth  rational curve of bi-degree $(1,1)$,%
\item [$\mathrm{(I.5.m})$] $S$ is a blow-up of the surface in $\mathrm{(I.1A)}$ at $m\le 8$ distinct points on the boundary curve such that $-K_{S}$ is ample, i.e., $S$ is a del Pezzo surface, and $C$ is the proper transform of the boundary curve in $\mathrm{(I.1A)}$, i.e., $C\in|-K_{S}|$,%
\item [$\mathrm{(I.6B.m})$] $S$ is a blow-up of the surface in $\mathrm{(I.1B)}$ at $m\ge 1$ distinct points on the boundary curve, and $C$ is the proper transform of the boundary curve in $\mathrm{(I.1B)}$,%
\item [$\mathrm{(I.6C.m})$] $S$ is a blow-up of the surface in $\mathrm{(I.1C)}$ at $m\ge 1$ distinct points on the boundary curve, and $C$ is the proper transform of the boundary curve in $\mathrm{(I.1C)}$,%
\item [$\mathrm{(I.7.n.m})$] $S$ is a blow-up of the surface in $\mathrm{(I.2.n)}$ at $m\ge 1$ distinct points on the boundary curve, and $C$ is the proper transform of the boundary curve in $\mathrm{(I.2)}$,%
\item [$\mathrm{(I.8B.m})$] $S$ is a blow-up of the surface in
$\mathrm{(I.3B)}$ at $m\ge 1$ distinct points on the boundary curve, and $C$ is the proper transform of the boundary curve in $\mathrm{(I.3B)}$,%
\item [$\mathrm{(I.9B.m})$] $S$ is a blow-up of the surface in $\mathrm{(I.4B)}$ at $m\ge 1$ distinct points on the boundary curve with no two
of them on a single curve of bi-degree $(0,1)$, and $C$ is the proper transform of the boundary curve in $\mathrm{(I.4B)}$,%
\item [$\mathrm{(I.9C.m})$] $S$ is a blow-up of the surface in $\mathrm{(I.4C)}$ at $m\ge 1$ distinct points on the boundary curve, and $C$ is the proper transform of the boundary curve in $\mathrm{(I.4C)}$.%
\end{itemize}
\end{theorem}

The rest of the section is devoted to the proof of this theorem.

\subsection{Classification}
\label{ClassifSmoothSubsec}
Throughout this subsection we assume without further mention that
\beq\label{mainassumption}
-K_S-(1-\beta)C\ \mathrm{is\ ample\ for\ sufficiently\ small}\ \beta\in(0,1],%
\eeq i.e., $(S,C)$ is asymptotically log del Pezzo. Then
$-K_S-C$ is nef. Moreover, the surface $S$ is projective, since
$-K_S-(1-\beta)C$ is an ample $\mathbb{Q}$-divisor for
sufficiently small rational $\beta\in(0,1]$. Furthermore, the
divisor $-K_{S}$ is big, since it is a sum of an ample class and
an effective class, to wit,
$$
-K_{S}=-(K_S+(1-\beta)C)+(1-\beta)C.
$$
Since $-K_{S}$ is big, we have
$h^{0}(\mathcal{O}_{S}(K_{S}))=h^{0}(\mathcal{O}_{S}(2K_{S}))=0$.
Moreover, it follows from the Kawamata--Viehweg Vanishing Theorem
\cite[Vol. II, \S9.1.C]{Lazarsfeld} that
$h^{1}(\mathcal{O}_{S})=h^{2}(\mathcal{O}_{S})=0$. Thus, the
surface $S$ is rational by Castelnuovo's rationality criterion
\cite[p. 536]{GH}. We remark that all of these considerations
apply equally when $C$ has several components, but for the
rest of this subsection we implicitly assume $r=1$ unless
explicitly stated.

In the rest of this subsection we prove that $(S,C)$
is one of the pairs listed in Theorem~\ref{theorem:main-1}.
Our proof is divided into several steps, each contained
in a separate paragraph.

\subsubsection{Non-rational boundary}

Let $g(C)$ denote the genus of the (smooth) curve $C$.

\begin{lemma}
\label{lemma:log-del-Pezzo-rational-or-elliptic-1} Suppose that
$g(C)\ne 0$. Then $-K_{S}$ is ample, i.e., $S$ is a del Pezzo
surface, and $C$ is a smooth elliptic curve in $|-K_{S}|$.
\end{lemma}

\bpf Since $-K_S-C$ is nef, it follows from the adjunction theorem
that \beq \label{FirstKSCEq}
0\le 2g-2=(K_S+C). C\le 0,%
\eeq
thus $g=1$.

Next, by Kodaira--Serre duality,
$h^{2}(\mathcal{O}_{S}(K_{S}+C))=h^{0}(\mathcal{O}_{S}(-C))=0$.
Also, since $S$ is rational $\chi(\calO_S)=1$.
Recalling that $(K_{S}+C). C=0$ by \eqref{FirstKSCEq},
and using the Riemann–-Roch Theorem thus gives
\beq
\baeq
1
&=
\chi(\calO_S)+\frac{(K_S+C).(K_S+C-K_S)}2
=
\chi(\calO_S(K_S+C))
\cr
&=
h^{0}(\mathcal{O}_{S}(K_{S}+C))
-
h^{1}(\mathcal{O}_{S}(K_{S}+C))
+
h^{2}(\mathcal{O}_{S}(K_{S}+C))
\cr
&\le
h^{0}(\mathcal{O}_{S}(K_{S}+C)).
\eaeq
\eeq
Therefore,
there exists an
effective divisor $R$ such that $R\sim K_{S}+C$.
Thus by \eqref{EffMinusNeffZero} $R=0$, from
which $C\in|-K_S|$, and
\begin{equation}
\label{equation:log-canonical-class}
-\beta K_{S}\sim -K_S-(1-\beta)C>0,%
\end{equation}
for sufficiently small rational $\beta\in(0,1]$, i.e.,
$S$ is del Pezzo.
\epf

Recall that a smooth projective surface $S$
not equal to $\PP^1\times \PP^1$
is del Pezzo precisely when there is a smooth
anticanonical curve $C\subset S$ which is the proper
transform of a smooth cubic curve in $\PP^2$ blown-up
at $8$ points in general position on the cubic
\cite[Proposition 3.2]{Hitchin}.
Thus, we have:

\begin{corollary}
\label{corollary:log-del-Pezzo-rational-or-elliptic-1} Suppose $C$
is not rational, then $(S,C)$ is one of $\mathrm{(I.1A)}$,
$\mathrm{(I.4A)}$, or $\mathrm{(I.5.m)}$.
\end{corollary}

Thus, for the remainder of \S\ref{ClassifSmoothSubsec},
we assume
\beq\label{CrationalAssump}
\h{$C$ is a smooth rational curve.}
\eeq

\begin{remark}
\label{remark:del-Pezzo-case} In the notation and assumption of
Theorem~\ref{theorem:main}, suppose additionally that $S$ is a del
Pezzo surface. Then the divisor
$
-K_{S}-\sum_{i=1}^r(1-\beta_i)C_i
$
is ample for {\it any} $(\beta_1,\ldots,\beta_r)\in(0,1]^r$.
\end{remark}

\subsubsection{Rational boundary and curves of negative self-intersection}

The goal is now to show that
$(S,C)$ is one of the cases not covered by Corollary
\ref{corollary:log-del-Pezzo-rational-or-elliptic-1} .
In this paragraph we derive some basic
intersection properties of the boundary and other
curves of negative self-intersection.

\begin{lemma}
\label{lemma:log-del-Pezzo-rational-curves-1} Let $Z$ be an
irreducible curve on $S$ such that $Z\ne C$ and $Z^2<0$. Then $Z$
is a smooth rational curve and $Z^2=-1$.
\end{lemma}

\begin{proof}
Since $Z\ne C$,
\beq
\label{KSZEq}
-K_{S}. Z=-\Big(K_{S}+(1-\beta)C\Big). Z+(1-\beta)C. Z\ge -\Big(K_{S}+(1-\beta)C\Big). Z>0,%
\eeq
for sufficiently small $\beta\in(0,1]$.
Hence, by \eqref{ArithGeomGenus} it follows that
\begin{equation}
\label{equation:subadjunction-1}
0>Z^2>K_{S}. Z+Z^{2}=2h^{1}(\mathcal{O}_{Z})-2,%
\end{equation}
or $h^{1}(\mathcal{O}_{Z})=0$. Now it follows from
\eqref{SuchRationalIsSmooth}
that $Z$ is a smooth rational
curve. Going back to $(\ref{equation:subadjunction-1})$
then $Z^2=-1$.
\end{proof}

\begin{lemma}
\label{lemma:log-del-Pezzo-rational-curves-not-del-Pezzo-1}
Suppose that
$S$ is not a del Pezzo
surface. Then $C^2\le -2$.
\end{lemma}

\begin{proof}
Suppose that $C^2\ge -1$. Then $-K_{S}. C>0$ by the
adjunction formula, since by Lemma
\ref{lemma:log-del-Pezzo-rational-curves-1} $C$ is a smooth rational curve.
Also, by \eqref{KSZEq}
$
K_{S}. Z<0
$
for every irreducible curve $Z\ne C$ on the surface $S$.
Moreover, $K_{S}^2>0$, since $-K_{S}$ is big. Therefore, the
divisor $-K_{S}$ is ample by the Nakai--Moishezon criterion, which
contradicts our assumption.
\end{proof}

The next lemma is crucial for the proof of the main result.
It shows that any $-1$-curve intersects the boundary transversally
at most at one point.

\begin{lemma}
\label{lemma:log-del-Pezzo-blow-down-1} Suppose that there exists
a smooth irreducible rational curve $E$ on the surface $S$ such
that $E^{2}=-1$ and $E\ne C$. Then either $E\cap C=\emptyset$ or
$E. C=1$.
\end{lemma}

\begin{proof}
Choose
$\beta$ such that $\beta C. E<1$
(and, as always,
also satisfying \eqref{mainassumption}).
By adjunction, $-K_{S}. E=1$.
Then
$$
0<-(K_{S}+(1-\beta)C). E=1-C. E+\beta C.
E<2-C. E,
$$
thus $C. E<2$. Hence, either $C. E=0$ (and,
thus $E\cap C=\emptyset$ since $E\ne C$) or $C. E=1$.
\end{proof}

\subsubsection{Minimal pairs}
\label{MinimalPairsSubsubsec}

Suppose that there exists
a smooth irreducible rational curve $E$ on the surface $S$ such
that $E^{2}=-1$ and $E\ne C$. By Castelnuovo's
contractibility criterion
there exists a birational morphism $\pi\colon S\to s$
that contracts the curve $E$ to a smooth point of the surface
$s$ \cite[p. 476]{GH}. Since by Lemma~\ref
{lemma:log-del-Pezzo-blow-down-1} $E$ and $C$ intersect
transversally at most at one point then
$\pi(C)$ is a smooth curve. Moreover,
by \eqref{pushforward},
we see that the divisor $-(K_s+(1-\beta)\pi(C))$ is ample
provided that $-(K_{S}+(1-\beta)C)$ is ample. Thus, we
see that $(s,\pi(C))$ is asymptotically log del Pezzo as
well.

Thus, it seems possible to use
Lemma~\ref{lemma:log-del-Pezzo-blow-down-1} to give an inductive
proof
(in the rank of the Picard group, $\mathrm{Pic}(S)$) of
one direction of Theorem~\ref{theorem:main-1}. To do this in a consistent way we
make the following definition.

\begin{definition}
\label{definiton:log-del-Pezzo-minimally-good-1}
The pair $(S,C)$
is \emph{minimal} if there exist no smooth irreducible rational
curve $E$ on the surface $S$ such that $E^{2}=-1$, $E\ne C$ and
$E\cap C\ne\emptyset$.
\end{definition}

The base of our induction is given by the next lemma.
Recall that throughout we are assuming $C$ is a smooth
rational curve \eqref{CrationalAssump}.

\begin{lemma}
\label{lemma:log-del-Pezzo-Picard-rank-two-1} Suppose that
$\mathrm{rk}(\mathrm{Pic}(S))\le 2$
and $C\not\sim-K_{S}$. Then when $(S,C)$
is minimal it is one of
$\mathrm{(I.1B)}$, $\mathrm{(I.1C)}$, $\mathrm{(I.2.n)}$,
$\mathrm{(I.3A)}$, $\mathrm{(I.3B)}$, $\mathrm{(I.4B)}$, or
$\mathrm{(I.4C)}$, and otherwise it is
$\mathrm{(I.6B.1)}$ or $\mathrm{(I.6C.1)}$.
\end{lemma}

\begin{proof}
First note that
all the cases listed
in the statement are indeed asymptotically log
del Pezzo by \S\ref{VerifSubsec}.
By \eqref{RationalClassifPicAtMostTwo}
the assumption $\mathrm{rk}(\mathrm{Pic}(S))\le 2$ implies that
either
$S\cong\mathbb{P}^2$ or $S\cong\mathbb{F}_n, n\ge0$.
In the former case, $(S,C)$
is either $\mathrm{(I.1B)}$ or
$\mathrm{(I.1C)}$, as $C$ is rational.
Let us consider the latter cases.
If $n=0$, then one sees that $(S,C)$ is either
$\mathrm{(I.2_0)}$, $\mathrm{(I.4B)}$, or $\mathrm{(I.4C)}$
(again, as $C$ is rational).
Let $n>0$ and suppose $C\in|aZ_n+bF|$ with $a,b\in\NN\cup\{0\}$.
Then by \eqref{KFnEq}--\eqref{ampleFn},
$-K_S-(1-\be)C=(2-(1-\be)a)Z_n+(n+2-(1-\be)b)F$
is ample if and only if $a\in[0,2], b\in[0,na+\frac{2-n}{1-\be})$.

Suppose first that $b=0$. Then either $(a,b)=(1,0)$, i.e., $C=Z_n$ and we
are in the case $\mathrm{(I.2.n)}$, or else
$(a,b)=(2,0)$, i.e., $C\in|2Z_n|$, but since
$Z_n$ is unique for $n>0$ by \eqref{FnZn} this means $C$ is not reduced,
so this case is excluded.

Thus it remains to consider
the case $b>0$.  If $a=0$ then necessarily $b=n=1$.
This is excluded
by minimality since then $C.Z_1=1$
and $Z_1^2=-1$. If $a=1$ then $b\in[1,2]$.
The case $(a,b)=(1,1)$ implies $C.Z_n=1-n\ge 0$ since $C\not=Z_n$.
Thus $n=1$ and we obtain case $\mathrm{(I.3B)}$.
Similarly the case $(a,b)=(1,2)$ implies $n\le2$.
But $n=1$ is excluded by minimality since then
$C.Z_1=1, Z_1^2=-1$ and $C\not=Z_1$, while
$n=2$ is excluded by Lemma \ref{lemma:log-del-Pezzo-rational-curves-1}
as $C\not=Z_2$.
Finally, if $a=2$ then $b\in[1,n+2]$. Then $C.Z_n=-2n+b\ge0$ as
$C\not=Z_n$. Thus either $n=2$ and $b=4$, or else $n=1$ and
$b=2$ or $b=3$. The former is again
excluded by Lemma \ref{lemma:log-del-Pezzo-rational-curves-1}, while
the latter gives only the case $\mathrm{(I.3A)}$ since
if $(a,b)=(2,3)$ then $C\in|-K_S|$ is not rational.
\end{proof}

\subsubsection{The inductive step}

The next lemma provides the inductive step for our classification.
Note that, by definition, part (ii) refers to the case the $-1$-curve
$E$ is disjoint from the boundary.

\begin{lemma}
\label{lemma:log-del-Pezzo-blow-down-minimally-1}
(i)
Suppose that there exists a smooth irreducible
rational curve $E$ on the surface $S$ such that $E^{2}=-1$ and
$E\ne C$. Then there exists a birational morphism $\pi\colon
S\to s$ such that $s$ is a smooth surface, $\pi(E)$ is
a point, the morphism $\pi$ induces an isomorphism $S\setminus
E\cong s\setminus\pi(E)$, the curve $\pi(C)$ is a smooth
rational curve, and
 $(s, \pi(C))$ is asymptotically
log del Pezzo.
\hfill\break
(ii) Suppose in addition that $(S,C)$ is minimal. Then
 $(s, \pi(C))$ is minimal.
\end{lemma}

\begin{proof}
(i) By
the discussion at the beginning of \S\ref{MinimalPairsSubsubsec}
there exists a birational morphism $\pi\colon S\to s$ that
contracts $E$ to a smooth point of the surface $s$, the
curve $\pi(C)$ is a smooth rational curve, and the divisor
$-(K_{s}+(1-\beta)\pi(C))$ is ample for sufficiently small
$\beta\in(0,1]$, i.e., the pair $(s, \pi(C))$ asymptotically
del Pezzo.

(ii) It remains to show that $(s,\pi(C))$ is minimal. Suppose, on
the contrary, that there exists a smooth irreducible rational
curve $z$ on the surface $s$ such that $z^{2}=-1$, ${z}\ne
\pi(C)$, and ${z}\cap \pi(C)\not=\emptyset$. Let $Z$ be the proper
transform of the curve ${z}$ on the surface $S$. Then either
$\pi(E)\in{z}$ and $Z^{2}=-2$, contradicting
Lemma~\ref{lemma:log-del-Pezzo-rational-curves-1}, or else
$\pi(E)\not\in{z}$ and $Z^2=-1$, but then $Z\cap C\ne\emptyset$,
contradicting minimality of $(S,C)$.
\end{proof}

\subsubsection{Classification of minimal pairs}

The next lemma uses a geometric argument to apply the inductive
step to reduce the classification of minimal pairs to
Lemma \ref{lemma:log-del-Pezzo-Picard-rank-two-1}.

\begin{lemma}
\label{lemma:log-del-Pezzo-minimal-good-small-Picard-r-1} Suppose
that $(S,C)$ is minimal. Then
$\mathrm{rk}(\mathrm{Pic}(S))\le 2$.
\end{lemma}

\begin{proof}
If $C\sim -K_S$ then by Corollary
\ref{corollary:log-del-Pezzo-rational-or-elliptic-1}
the pair must be $\mathrm{(I.1A)}$ or $\mathrm{(I.4A)}$,
hence $\mathrm{rk}(\mathrm{Pic}(S))\le 2$. So we
assume that $C\not\sim -K_S$.

Let $(S,C)$ be a pair and
suppose that $\mathrm{rk}(\mathrm{Pic}(S))\ge 3$.
We would like to show that the pair is not minimal.

If $S$ a del Pezzo surface, there exists a smooth irreducible
rational $-1$-curve $E$ on the surface $S$ such that $E\ne C$.
Indeed, in this case it follows from the classification of smooth
del Pezzo surfaces referred to before Corollary
\ref{corollary:log-del-Pezzo-rational-or-elliptic-1} that there
are at least three distinct $-1$-curves in $S$ at most one of
which can be the boundary. On the other hand, if $S$ is not del
Pezzo, then the existence of such curve $E$ follows from
Lemma~\ref{lemma:log-del-Pezzo-rational-curves-not-del-Pezzo-1},
and \eqref{MinusOneCurveExistence}. To complete the proof it thus
suffices to show that $E\cap C\not=\emptyset$ (recall
Definition~\ref{definiton:log-del-Pezzo-minimally-good-1}).

Suppose, on the contrary, that every such $-1$-curve $E$
in $S$ satisfies $E\cap C=\emptyset$, i.e., that $(S,C)$
is minimal. By
Lemma~\ref{lemma:log-del-Pezzo-blow-down-minimally-1} there
exists a birational morphism $\pi\colon S\to s$ such that
$s$ is a smooth surface, $\pi(E)$ is a point, the morphism
$\pi$ induces an isomorphism $S\setminus E\cong
s\setminus\pi(E)$, the curve $\pi(C)$ is smooth, and the pair
$(s, \pi(C))$ is asymptotically log del Pezzo and minimal.
Since
$$
\mathrm{rk}(\mathrm{Pic}(s))=\mathrm{rk}(\mathrm{Pic}(S))-1\ge 2,%
$$
we may as well assume that $\mathrm{rk}(\mathrm{Pic}(S))=3$ and
$\mathrm{rk}(\mathrm{Pic}(s))=2$.

Since $\mathrm{rk}(\mathrm{Pic}(s))=2$, one has
$s\cong\mathbb{F}_{n}$ for some $n\ge 0$. Put
$${c}:=\pi(C).$$ Then $\pi(E)\not\in{c}$, since $E\cap
C=\emptyset$. Since $\mathrm{rk}(\mathrm{Pic}(s))=2$ and
$(s, {c})$ is minimal,
Lemma~\ref{lemma:log-del-Pezzo-Picard-rank-two-1} implies that
$(s,c)$ is one of $\mathrm{(I.2.n)}$, $\mathrm{(I.3A)}$, $\mathrm{(I.3B)}$,
$\mathrm{(I.4B)}$, or $\mathrm{(I.4C)}$.

Let $\xi\colon s\to\mathbb{P}^1$ be a natural projection
(unique when $n>0$, and one of two choices when $n=0$, see
\eqref{FnZn}), let $f$ be the fiber of the
morphism $\xi$ that passes through the point $\pi(E)$, and let $F$
be its proper transform on $S$. Here we are following the
conventions of \eqref{FnZn}. Then $F$ is a smooth irreducible
rational curve such that $F^{2}=-1$ (since $f^2=0$ downstairs).
Moreover, we have $F\cap C=\emptyset$, since $(S,C)$ is minimal.
Since $E\cap C=\emptyset$, we see that $f\cap{c}=\emptyset$, which
implies that ${c}$ is also a fiber of the morphism $\xi$, i.e.,
\beq
\label{cfiberEq}
c\in|f|,
\eeq
and in particular also $c^2=0$.
In the case $\mathrm{(I.2.n)}$ ${c}^2=-n$,
while in the cases
$\mathrm{(I.3A)}$, $\mathrm{(I.3B)}$, $\mathrm{(I.4B)}$,
$\mathrm{(I.4C)}$, we have ${c}^2\ne 0$.
Thus, $(s,{c})$ is neither $\mathrm{(I.2.n)}, n>0$,
nor one of $\mathrm{(I.3A)}$, $\mathrm{(I.3B)}$,
$\mathrm{(I.4B)}$, or $\mathrm{(I.4C)}$. The only remaining
possibility is that $(s,{c})$ is $\mathrm{(I.2.0)}$. Then
$s\cong\mathbb{P}^1\times\mathbb{P}^1$ and $\xi=p_1$
is a projection to one of the factors.
Let $p_2$ denote the projection onto the second factor,
and let $g$ denote the fiber of $p_2$ passing through $\pi(E)$.
Then $g$ intersects $c$ at one point, $g^2=0$, and hence
the proper transform of $g$, denoted $G$, satisfies
$G^2=-1$ and $G\cap C=1$, which once again contradicts
minimality of $(S,C)$. This completes the
proof of the lemma.
\end{proof}

\subsubsection{Dealing with non-uniqueness}

Now we are ready to finish the proof of the classification part of
Theorem~\ref{theorem:main-1}. If
$\mathrm{rk}(\mathrm{Pic}(S))\le 2$, then it follows from
Lemma~\ref{lemma:log-del-Pezzo-Picard-rank-two-1} that
$(S,C)$
is one of
$\mathrm{(I.1B)}$, $\mathrm{(I.1C)}$, $\mathrm{(I.2.n)}$,
$\mathrm{(I.3A)}$, $\mathrm{(I.3B)}$, $\mathrm{(I.4A)}$,
$\mathrm{(I.4B)}$, $\mathrm{(I.4C)}$, $\mathrm{(I.6B.1)}$, or
$\mathrm{(I.6C.1)}$. Thus, we may assume that
$\mathrm{rk}(\mathrm{Pic}(S))\ge 3$. In particular, the pair
$(S,C)$ is not minimal by
Lemma~\ref{lemma:log-del-Pezzo-minimal-good-small-Picard-r-1}. To
prove Theorem~\ref{theorem:main-1}, we must show that $(S,C)$
is
one of the cases: $\mathrm{(I.6B.m)},\mathrm{(I.6C.m)}$ for some $\NN\ni m\ge 2$,
$\mathrm{(I.7.n.m)}$ for some positive integers $n$ and $m$, or,
finally,
$\mathrm{(I.8B.m)}$,
$\mathrm{(I.9B.m)}$, or $\mathrm{(I.9C.m)}$
for some positive integer $m$.

Since $(S,C)$ is not minimal, there exists a curve $E$
and a
birational morphism $\pi\colon S\to s$
as in
Lemma~\ref{lemma:log-del-Pezzo-blow-down-minimally-1}.
The next lemma
follows directly from
Lemma~\ref{lemma:log-del-Pezzo-rational-curves-1}.

\begin{lemma}
\label{lemma:lemma:log-del-Pezzo-rational-curves-1-1} Let
$g$ be a smooth irreducible rational curve on the surface
$s$ such that $g\ne\pi(C)$ and $g^{2}=-1$. Then
$\pi(E)\not\in g$.
\end{lemma}

Now we may replace the pair $(S,C)$ by the pair
$(s,{c})$ and iterate this process. As a result, we
obtain a birational morphism, that by abuse, we still
denote by $\pi: S\to s$ such that
${s}$ is a smooth surface, $\pi$ is a a blow-up of $m$
\emph{distinct} points $P_1,P_2,\ldots,P_m$ on the smooth curve
$\hat{C}\subset\hat{S}$ such that $c:=\pi(C)$,
and, finally, $(s,c)$ is a minimal
asymptotically log del Pezzo pair. By
Lemma~\ref{lemma:log-del-Pezzo-minimal-good-small-Picard-r-1}, one
has $\mathrm{rk}(\mathrm{Pic}(s))\le 2$. By
Lemma~\ref{lemma:log-del-Pezzo-Picard-rank-two-1},
$(s,c)$ is
$\mathrm{(I.1B)}$, $\mathrm{(I.1C)}$, $\mathrm{(I.2.n)}$,
$\mathrm{(I.3A)}$, $\mathrm{(I.3B)}$, $\mathrm{(I.4B)}$, or
$\mathrm{(I.4C)}$.

\begin{corollary}
\label{corollary:almost-end-1} If $(s,c)$
is  $\mathrm{(I.1B)}$,
$\mathrm{(I.1C)}$, $\mathrm{(I.2.n)}$, or $\mathrm{(I.4C)}$, then
$(s,c)$ can also be obtained as described in one of the cases
$\mathrm{(I.6B.m)}$, $\mathrm{(I.6C.m)}$, $\mathrm{(I.7.n.m)}$,
$\mathrm{(I.8C.m)}$, respectively.
\end{corollary}

Thus, to complete the proof of Theorem~\ref{theorem:main-1}, we
must do the following two things:
\begin{itemize}
\item if $s\cong\mathbb{P}^1\times\mathbb{P}^1$ and
$c$ is a smooth rational curve of bi-degree $(2,1)$, we must
check that no two points among $P_1,P_2,\ldots,P_m$ lie on a
one curve in $s$ of bi-degree $(0,1)$,

\item if $(s,c)$ is
$\mathrm{(I.3A)}$ we must show the pair $(S,C)$
can also be described by a birational morphism
 that is listed in
Theorem~\ref{theorem:main-1}.
\end{itemize}

The first point is simple.
Suppose that there exist two points among $P_1,P_2,\ldots,P_m$
that lie on a one curve in $s$ of bi-degree $(0,1)$. Let us
denote this curve by $z$. Denote by $Z$ its proper transform
on the surface $S$. Then $Z^2\le -2$, contradicting
Lemma~\ref{lemma:log-del-Pezzo-rational-curves-1}, because
$z\ne c$.

The second point is dealt with using the next lemma.

\begin{lemma}
\label{lemma:the-end-1} Suppose that
$(s,c)$ is $\mathrm{(I.3A)}$.
Then $(S,C)$ can also be described
as $\mathrm{(I.9B.m)}$.
\end{lemma}

\begin{proof}
We have $s\cong\mathbb{F}_1$. Let $\xi\colon
S\to\mathbb{P}^1$ be the natural projection, let $z$ be the
section of $\xi$ such that ${z}^2=-1$, and let $f$ be a
fiber of $\xi$ that passes through the point $P_1$. Then the curve
${c}$ is a smooth rational in $|2z+2f|$.

There exists a commutative diagram
$$
\xymatrix{
S\ar@{->}[rrd]_{\pi}\ar@{->}[rrr]^{\upsilon}&&&
Bl_{P_1}\FF_1\ar@{->}[dl]_{\psi}\ar@{->}[dr]^{\phi}&\\
&&s=\FF_1\ar@{->}[d]_{\xi}&&\s=\PP^1\times\PP^1\ar@{->}[d]^{\xi^\prime}\\
&&\mathbb{P}^1\ar@{=}[rr]&&\mathbb{P}^{1},}
$$
where $\psi$ is a blow-up of the point $P_1$, $\phi$ is a
contraction of the proper transform of the fiber $f$,
$\upsilon$ is a birational morphism, and $\xi^\prime$ is a natural
projection. Put $\pi^\prime=\phi\circ\upsilon$. Let us show
that $\pi^{\prime}\colon S\to \s$ is the desired
replacement of the birational morphism $\pi\colon S\to s$.
These birational transformations
did not change the generic fiber
of the projection $\xi$. Thus, $\sigma$ comes equipped with
a fibration $\xi':\sigma\ra\PP^1$
In particular, the curve $\zeta$ in $\sigma$ corresponding to $z\subset s$
is a fiber of $\xi'$ and it has zero self-intersection. Thus,
$\s\cong\mathbb{P}^1\times\mathbb{P}^1$.
Becuase $\phi$ contracts a $-1$-curve ($\tilde f$, the proper
transform of $f$) that intersects both
$z$ and the exceptional curve $A$ of $\psi$, it follows
that $\phi(A)$ has zero self-intersection and intersects
$\zeta$ at one point. At the same time $\phi(A)$ intersects
the transformed boundary of $s$ (which equals $\pi'(C)$
at two points. Thus $\pi'(C)$ is a curve of bi-degree
$(2,1)$. Thus, $(S,C)$ is the blow-up of $(\PP^1\times\PP^1,\pi'(C))$
at $m\ge1$ points. Further, as already checked earlier, no two of
these points may lie on a single fiber of $\xi'$.
Thus, $(S,C)$ is $\mathrm{(I.9B.m)}$.
\end{proof}

\section{Strongly asymptotically log del Pezzo surfaces}
\label{section:r-2-3-4}

The following theorem gives complete classification in the case of
a reducible boundary curve.
We assume without further mention that in each
case listed below the curves composing the boundary intersect simply
and normally. A point in the smooth locus of such a boundary
means a point that is not an intersection point
of any two components of the boundary.

\begin{theorem}
\label{theorem:main-2-3-4} Let $S$ be a smooth surface, let
$C_1,\ldots,C_r$ be irreducible smooth curves on $S$ such that
$\sum_{i=1}^{r}C_i$  is a divisor with simple normal crossings.
Suppose that $r\ge 2$. Then $(S,\sum_{i=1}^{r}C_i)$ is a
strongly asymptotically log del Pezzo surface if and only if it
is one of the following pairs:
\begin{itemize}
\item [$\mathrm{(II.1A})$]  $|C_1\cap C_2|=2$, $S\cong\mathbb{P}^2$, and $C_1$ is a smooth conic, and $C_2$ is a line,%

\item [$\mathrm{(II.1B})$]  $|C_1\cap C_2|=1$, $S\cong\mathbb{P}^2$, and $C_1$ and $C_2$ are two distinct lines,%

\item [$\mathrm{(II.2A.n})$]  $C_1\cap C_2=\emptyset$, $S\cong\mathbb{F}_n$ for any $n\ge 0$, $C_1=Z_n$ and $C_2\in |Z_n+nF|$,%

\item [$\mathrm{(II.2B.n})$]  $|C_1\cap C_2|=1$, $S\cong\mathbb{F}_n$ for any $n\ge 0$, $C_1=Z_n$ and $C_2\in|Z_n+(n+1)F|$,%

\item [$\mathrm{(II.2C.n})$]  $|C_1\cap C_2|=1$, $S\cong\mathbb{F}_n$ for any $n\ge 0$, $C_1=Z_n$ and $C_2=F$,%

\item [$\mathrm{(II.3})$]  $|C_1\cap C_2|=1$,
$S\cong\mathbb{F}_1$, $C_1, C_2\in|Z_1+F|$,

\item [$\mathrm{(II.4A})$]  $|C_1\cap C_2|=2$,
$S\cong\mathbb{P}^1\times\mathbb{P}^1$, $C_1,C_2$
are distinct bi-degree $(1,1)$ curves,

\item [$\mathrm{(II.4B})$]  $|C_1\cap C_2|=2$,
$S\cong\mathbb{P}^1\times\mathbb{P}^1$, the curve $C_1$ is a
smooth rational curve of bi-degree $(2,1)$, and $C_2$ is a smooth
rational curve of bi-degree $(0,1)$,

\item [$\mathrm{(II.5A.m})$]  $|C_1\cap C_2|=2$, $(S,C)$ is
a blow-up of $\mathrm{(II.1A})$ at $1\le m\le 5$ points in the
smooth locus of the boundary curve such that the surface $S$ is a
del Pezzo surface and $C_1^2,C_2^2\ge 0$, i.e.,
$C_1+C_2\sim -K_{S}$, and there exists a birational morphism
$\pi\colon S\to\mathbb{P}^2$ such that $\pi(C_1)$ is a smooth
conic, and $\pi(C_2)$ is a line such that
$|\pi(C_1)\cap\pi(C_2)|=2$, and $\pi$ is a blow-up of
$1\le m\le 5$ distinct points on
$\pi(C_1)$ and $\pi(C_2)$ but away
from $\pi(C_1)\cap\pi(C_2)$
with no two of them on $\pi(C_1)$,
and no five of them on $\pi(C_2)$,

\item [$\mathrm{(II.5B.m})$]  $|C_1\cap C_2|=1$,  $(S,C)$ is
a blow-up of $\mathrm{(II.1B})$ at $m\ge 1$ points in the
smooth locus of the boundary curve,
i.e., there exists a birational
morphism $\pi\colon S\to\mathbb{P}^2$ such that $\pi(C_1)$ and
$\pi(C_2)$ are distinct lines, and $\pi$ is a blow-up of
$m\ge 1$ distinct points on
$\pi(C_1)$ and $\pi(C_2)$  but away from
$\pi(C_1)\cap \pi(C_2)$,

\item [$\mathrm{(II.6A.n.m})$]  $C_1\cap C_2=\emptyset$,
$(S,C)$ is a blow-up of $\mathrm{(II.2A.n})$ at $m\ge 1$
points on the boundary curve such that there exists a
birational morphism $\pi\colon S\to\mathbb{F}_{n}$ for some
$n\ge 0$ such that $\pi(C_1)=Z_n$, $\pi(C_2)\in
|Z_n+nF|$, and $\pi$ is a blow-up of $m$ distinct
points on $\pi(C_1)$ and $\pi(C_2)$
with at most one point on a single curve in the linear system $|F|$,

\item [$\mathrm{(II.6B.n.m})$]  $|C_1\cap C_2|=1$,  $(S,C)$
is a blow-up of $\mathrm{(II.2B.n})$ at $m\ge 1$ points in
the smooth locus of the boundary curve
such that
there exists a
birational morphism $\pi\colon S\to\mathbb{F}_{n}$ for some
$n\ge 0$ such that $\pi(C_1)=Z_n$, $\pi(C_2)\in
|Z_n+(n+1)F|$, and $\pi$ is a blow-up of $m\ge 1$
distinct points on $\pi(C_1)$ and
$\pi(C_2)$
with at most one point on a single curve in the linear system $|F|$,
and no point being $\pi(C_1)\cap \pi(C_2)$

\item [$\mathrm{(II.6C.n.m})$]  $|C_1\cap C_2|=1$,
 $(S,C)$ is a blow-up of $\mathrm{(II.2C.n})$ at $m\ge 1$ points in
the smooth locus of the boundary curve, i.e., there exists a
birational morphism $\pi\colon S\to\mathbb{F}_{n}$ for some
$n\ge 0$ such that $\pi(C_1)=Z_n$, $\pi(C_2)=F$,
and $\pi$ is a blow-up of $m\ge1$ distinct points
on $\pi(C_1)$ and $\pi(C_2)$ but away from
$\pi(C_1)\cap \pi(C_2)$,

\item [$\mathrm{(II.7.m})$]  $|C_1\cap C_2|=1$, $(S,C)$ is a
blow-up of $\mathrm{(II.3})$ at $m\ge 1$ points in the
smooth locus of the boundary curve
such that
there exists a birational
morphism $\pi\colon S\to\mathbb{F}_{1}$ such that
$\pi(C_1),\pi(C_2)\in|Z_1+F|$, and $\pi$ is
a blow-up of $m\ge1$ distinct points on
$\pi(C_1)$ and $\pi(C_2)$
with at most one point on a single curve in the linear system $|F|$,
and no point being $\pi(C_1)\cap \pi(C_2)$

\item [$\mathrm{(II.8.m})$]  $|C_1\cap C_2|=2$, $(S,C)$ is a
blow-up of $\mathrm{(II.4B})$ at $1\le m\le 4$ points in the
smooth locus of the boundary curve such that $S$ is a del Pezzo
surface and $C_1^2,C_2^2\ge 0$, i.e., $C_1+C_2\sim
-K_{S}$, and there exists a birational morphism $\pi\colon
S\to\mathbb{P}^1\times\mathbb{P}^1$ such that $\pi(C_1)$ is a
smooth rational curve of bi-degree $(2,1)$, $\pi(C_2)$ is a
smooth rational curve of bi-degree $(0,1)$, and $\pi$ is
a blow-up of $1\le m\le 4$ distinct points on $\pi(C_1)$
with no point being $\pi(C_1)\cap \pi(C_2)$,

\item [$\mathrm{(III.1})$]  $S\cong\mathbb{P}^2$, the curves
$C_1, C_2, C_3$ are lines, 

\item [$\mathrm{(III.2})$]  $S\cong\mathbb{P}^1\times\PP^1$,
$C_1, C_2, C_3$ are of bi-degree $(1,1), (0,1),$ and $(1,0),$ respectively,

\item [$\mathrm{(III.3.n})$]  $S\cong\mathbb{F}_n$ for any
$n\ge 0,  \, C_1=Z_n, C_2=F,$ and $C_3\in |Z_n+nF|$,

\item [$\mathrm{(III.4.m})$]  $(S,C)$ is a blow-up of
$\mathrm{(III.1})$ at $1\le m\le 3$ points in the smooth locus of
the boundary curve such that $S$ is a del Pezzo surface,
$C_1^2,C_2^2,C_3^2\ge 0$, i.e.,
$C_1+C_2+C_3\sim -K_{S}$, and there exists a birational morphism
$\pi\colon S\to\mathbb{P}^2$ such that the curves
$\pi(C_1),\pi(C_2), \pi(C_3)$ are lines that have no common intersection,
and $\pi$ is a blow-up of $1\le m\le 3$ distinct points on these lines
with at most one point on each line and no point on
an intersection of two lines,

\item [$\mathrm{(III.5.n.m})$]  $(S,C)$ is a blow-up of
$\mathrm{(III.3.n})$ at $m\ge 1$ points in the smooth locus
of the boundary curve
such that
there exists a birational morphism
$\pi\colon S\to\mathbb{F}_{n}$ for some $n\ge 0$ such that
$\pi(C_1)=Z_n$, $\pi(C_2)=F$, and $\pi(C_3)\in|Z_n+nF|$, and
$\pi$ is a blow-up of $m$ distinct points
on $\pi(C_1)$ and $\pi(C_3)$
with at most one point on a single curve in the linear system $|F|$,
and no point being $\pi(C_1)\cap \pi(C_2)$ or
$\pi(C_2)\cap\pi(C_3)$,

\item [$\mathrm{(IV})$]
$S\cong\mathbb{P}^1\times\mathbb{P}^1$, the curves $C_1$ and $C_2$
are distinct curves of bi-degree $(1,0)$, the curves $C_3$ and
$C_4$ are distinct curves of bi-degree $(0,1)$.
\end{itemize}
\end{theorem}

The rest of the section is devoted to the proof of
Theorem~\ref{theorem:main-2-3-4}.

\subsection{Basic properties of asymptotically log del Pezzo pairs}
\label{subsection:ALDP}

In this subsection,
before embarking on the proof of Theorem~\ref{theorem:main-1},
we collect several properties of
asymptotically log del Pezzo pairs that are not necessarily
strongly asymptotically log del Pezzo. These properties
are later used in the proof of that theorem, but they
should also be useful in a future classification of
the former class of pairs.

Thus in the rest of this subsection we
assume $(S,C), \be\in(0,1]^r$,
and $r\ge2$ are as in Definition \ref{definition:log-Fano}.

\begin{lemma}
\label{lemma:log-del-Pezzo-rational-or-elliptic} All curves
$C_1,\ldots,C_r$ are smooth rational curves.
\end{lemma}

\begin{proof}
Suppose that there exists a non-rational curve among the curves
$C_1,\ldots,C_r$. Without loss of generality, we may assume
that this curve is $C_1$. Since $-(K_S+\sum_{i=1}^{r}C_{i})$ is
nef, it follows from the adjunction theorem that
$$
2g(C_1)-2+\sum_{i=2}^{r}C_1. C_i=(K_S+\sum_{i=1}^{r}C_{i}). C_1\le 0,%
$$
which implies that $g(C_1)=1$ and $C_1\cap C_i=\emptyset$ for
every $i\ne 1$. Hence, we see that $C_1$ is an elliptic curve.
Arguing as in the proof of
Lemma~\ref{lemma:log-del-Pezzo-rational-or-elliptic-1}, we see
that there exists an effective divisor $R\sim C_1+K_{S}$.
Thus, $-R-C_1-\sum_{i=2}^{r}C_i\sim -K_S-C$ is nef,
implying $R\sim 0$ and $r=1$, which contradicts our
assumption that $r\ge2$.
\end{proof}

\begin{lemma}
\label{lemma:log-del-Pezzo-blow-down} Suppose that there exists a
smooth irreducible rational curve $E$ on the surface $S$ such that
$E^{2}=-1$ and $E\ne C_i$ for every $i$. Then either $E$ is
disjoint from $\sum_{i=1}^rC_i$, or it intersects exactly one
irreducible component of $\sum_{i=1}^rC_i$. Moreover, in the
latter case $E$ intersects that irreducible component transversally
at exactly one point.
\end{lemma}

\begin{proof}
Since $-K_{S}. E=1$ by adjunction then
$$
0<-(K_{S}+\sum_{i=1}^{r}(1-\beta_i)C_i).
E=1-\sum_{i=1}^{r}C_{i}. E+\sum_{i=1}^{r}\beta_i C_i.
E<2-\sum_{i=1}^{r}C_{i}. E,
$$
for small $|\be|$.
This implies
that $\sum_{i=1}^{r}C_{i}. E<2$. Hence, either
$\sum_{i=1}^{r}C_{i}. E=0$ or $\sum_{i=1}^{r}C_{i}. E=1$,
because $E\ne C_i$ for every $i$. In the former case $E\cap
C_i=\emptyset$ for every $i$. In the latter case there is an $i$ such
that $E. C_i=1$ and $E\cap C_j=\emptyset$ for every $j\ne i$.
\end{proof}

Similarly to
Definition~\ref{definiton:log-del-Pezzo-minimally-good-1}, let us
call the pair $(S,\sum_{i=1}^{r}C_{i})$ \emph{minimal} if there
exist no smooth irreducible rational curve on the surface $S$ such
that $E^{2}=-1$, $E\ne C_i$ for every $i$, and there is a $j$ such
that and $E\cap C_j\ne\emptyset$. Then we have the following
generalization of
Lemma~\ref{lemma:log-del-Pezzo-blow-down-minimally-1}.

\begin{lemma}
\label{lemma:log-del-Pezzo-blow-down-minimally} Suppose that there
exists a smooth irreducible rational curve $E$ on the surface $S$
such that $E^{2}=-1$ and $E\ne C_i$ for every $i$. Then there
exists a birational morphism $\pi\colon S\to{s}$ such that
${s}$ is a smooth surface, $\pi(E)$ is a point, the morphism
$\pi$ induces an isomorphism $S\setminus E\cong
{s}\setminus\pi(E)$, the divisor $\sum_{i=1}^r\pi(C_i)$ is a
divisor with simple normal crossings,  and
$\pi(C_1),\ldots,\pi(C_r)$ are smooth rational curves
whose dual graph is the same as the dual graph of the curves
$C_1,\ldots,C_r$. Moreover, the pair $({s},
\sum_{i=1}^{r}\pi(C_i))$ is asymptotically log del Pezzo
and strongly asymptotically log del Pezzo if $(S,C)$ is.
Furthermore, if the pair $(S,\sum_{i=1}^{r}C_{i})$ is minimal,
then the pair $({s}, \sum_{i=1}^{r}\pi(C_i))$ is minimal.

\end{lemma}

\begin{proof}
By the Castelnuovo's contractibility criterion, there exists a
birational morphism $\pi\colon S\to{s}$ such that ${s}$ is
a smooth surface, $\pi(E)$ is a point, the morphism $\pi$ induces
an isomorphism $S\setminus E\cong {s}\setminus\pi(E)$.
Moreover, the divisor $\sum_{i=1}^r\pi(C_i)$ is a divisor with a
simple normal crossing, the curves
$\pi(C_1),\ldots,\pi(C_r)$ are smooth rational curves
whose dual graph is the same as the dual graph of the curves
$C_1,\ldots,C_r$. Indeed, the latter is obvious if the curve $E$
is disjoint from $\sum_{i=1}^{r}C_{i}$. If $E$ is not disjoint
from $\sum_{i=1}^{r}C_{i}$, then it intersects exactly one
irreducible component of $\sum_{i=1}^rC_i$ (and intersects this
component transversally and at exactly one point) by
Lemma~\ref{lemma:log-del-Pezzo-blow-down}. The latter implies that
the divisor $\sum_{i=1}^r\pi(C_i)$ is a divisor with a simple
normal crossing, the curves $\pi(C_1),\ldots,\pi(C_r)$
are smooth rational curves whose dual graph is the same as the
dual graph of the curves $C_1,\ldots,C_r$. Now we can complete the
proof arguing as in the proof of
Lemma~\ref{lemma:log-del-Pezzo-blow-down-minimally-1} (ii).
\end{proof}

The next lemma describes the combinatorial structure of $C$.

\begin{lemma}
\label{lemma:log-del-Pezzo-rational-cycles-length-2} (i) Either
$|C_{i}\cap C_{j}|\le 1$ for $i\ne j$, or $r=2$,
$|C_1\cap C_2|=2$ and $C_1+C_2\sim -K_{S}$.
\hfill\break
(ii) If $r\ge
3$, then either the dual graph of the curves $C_1,\ldots,C_r$
forms a tree, or $\sum_{i=1}^{r}C_i\sim -K_{S}$ and the dual graph
of the curves $C_1,\ldots,C_r$ forms a cycle.
\hfill\break
(iii) If the
dual graph of the curves $C_1,\ldots,C_r$ forms a tree, then
it is a disjoint union of chains.
\end{lemma}

\begin{proof}
(i) Suppose that $|C_1\cap C_2|\ge 2$. We claim that $r=2$,
$C_1+C_2\sim -K_{S}$, and $|C_1\cap C_2|=2$.
By Serre duality,
$$
h^{2}(\mathcal{O}_{S}(K_{S}+C_1+C_{2}))=h^{0}(\mathcal{O}_{S}(-C_1-C_{2}))=0.
$$
Put $k=|C_1\cap C_2|$. Then
$$
\baeq
(K_{S}+C_1+C_{2}). & (C_1+C_{2})
=(C_1+C_2)^2+K_{S}. C_1+K_{S}. C_2
\cr
&=C_1^2+C_2^2+2C_1.
C_2+2g(C_1)-2-C_1^2+2g(C_2)-2-C_2^2=2k-4,
\eaeq
$$
since $C_1$ and $C_2$ are rational curves by
Lemma~\ref{lemma:log-del-Pezzo-rational-or-elliptic}. Since $S$ is
rational, it follows from the Riemann--Roch theorem that
$h^{0}(\mathcal{O}_{S}\big(K_{S}+C_1+C_{2}\big))\ge
1+(2k-4)/2=k-1\ge 1$.
The rest of the proof is now identical to that of
Lemma \ref{lemma:log-del-Pezzo-rational-or-elliptic}.

(ii) By (i), $|C_{i}\cap C_{j}|\le 1$ for every $i\ne j$.
Suppose that for some $k\le r$ the dual graph of the curves
$C_1,C_2,\ldots,C_k$ forms a cycle such that
$
C_k. C_1=C_1. C_2=\ldots=C_{k-1}. C_k=1,
$
and $C_i. C_j=0$ in all other cases when $1\le i\ne
j\le k$. We claim that $r=k$. Indeed, as before
$h^{2}(\mathcal{O}_{S}(K_{S}+\sum_{i=1}^{k}C_i))=h^{0}(\mathcal{O}_{S}(-\sum_{i=1}^{k}C_i))=0$.
Since
$$
\baeq
(K_{S}+\sum_{i=1}^{k}C_i).(\sum_{i=1}^{k}C_i)
&=
2\!\sum_{1\le i<j\le k} C_i. C_j
+
\sum_{i=1}^{k}K_{S}. C_j
+\sum_{i=1}^{k}C_i^2
\cr
&=
2\!\sum_{1\le i<j\le k} C_i. C_j
+\sum_{i=1}^{k}(2g(C_i)-2)
=2k-2k=0.
\eaeq
$$
Thus, as in (i),
$h^{0}(\mathcal{O}_{S}(K_{S}+\sum_{i=1}^{k}C_i))\ne 0$ by the
Riemann--–Roch theorem and there exists an effective
divisor $R$ such that $R\sim K_{S}+\sum_{i=1}^{k}C_i$,
hence the divisor
$
-K_{S}-\sum_{i=1}^{r}C_i\sim -R-\sum_{i=k}^{r}C_i
$
is nef, so $R=0$ and $r=k$.

(iii) Suppose that the dual graph of the curves
$C_1,\ldots,C_r$ forms a tree that is not a disjoint union of
chains. Then $r\ge 4$, and there exists a curve among
$C_{1},\ldots,C_{r}$ that intersects at least three other
different curves among $C_{1},\ldots,C_{r}$, say
$C_1. C_2=1$, $C_1. C_3=1$, and $C_1. C_4=1$.
Then
$$\baeq
0
&>
(K_S+\sum_{i=1}^{r}(1-\beta_i)C_i). C_1
=
K_{S}. C_1+(1-\beta_1)C_1^2+\sum_{i=2}^{4}(1-\beta_i)C_i. C_1+\sum_{i=5}^{r}(1-\beta_i)C_i. C_1
\cr
&\ge
K_{S}. C_1
+(1-\beta_1)C_1^2
+\sum_{i=2}^{4}(1-\beta_i)C_i. C_1
\cr
&= -2+C_1^2+(1-\beta_1)C_1^2+\sum_{i=2}^{4}(1-\beta_i)C_i. C_1
=1-\beta_1C_1^2-\beta_2-\beta_3-\beta_4>0,%
\eaeq
$$
for $|\beta|\ll 1$, a contradiction.
\end{proof}

The next lemma shows that only curves $C_i$ that are
at the `tail' of a chain can have negative self-intersection.

\begin{lemma}
\label{lemma:strongly-non-strongly} Suppose that
$(S,\sum_{i=1}^{r}C_i)$ is strongly asymptotically log del Pezzo.
Then $C_i^2\ge 0$ for every $C_i$ such that $C_i$ intersects
at least two curves among $C_1,\ldots,C_r$ different from itself.
Similarly, $C_i^2\ge 0$ if there exists a curve among
$C_1,\ldots,C_r$ different from $C_i$ that intersects $C_i$ by
more than one point.
\end{lemma}

\begin{proof}
Suppose that $C_1$, say, intersects at least two curves among
$C_2,\ldots,C_r$, say $C_1. C_2=C_1. C_3=1$. Suppose that $C_1^2<0$. Then it
follows from adjunction that
$$
\baeq
(K_S+\sum_{i=1}^{r}(1-\beta_i)C_i). C_1
&\ge
K_{S}.
C_1+(1-\beta_1)C_1^2+(1-\beta_2)C_2. C_1+(1-\beta_3)C_3.
C_1
\cr
&=-\beta_1C_1^2-\beta_2-\beta_3
%\ge \beta_1-\beta_2-\beta_3,%
\eaeq
$$
thus $\beta_1<\beta_2+\beta_3$.
The latter
contradicts our assumption that the divisor
$-(K_S+\sum_{i=1}^{r}(1-\beta_i)C_i)$ is ample for every
$(\beta_1,\ldots,\beta_r)\in(0,1]^r$ with
$|(\beta_1,\ldots,\beta_r)|<\epsilon$.

To complete the proof, we may assume that $C_1$ intersects some
curve among $C_2,\ldots,C_r$ by more than $2$ points. Then $r=2$
and $C_1. C_2=2$ by
Lemma~\ref{lemma:log-del-Pezzo-rational-cycles-length-2}. If
$C_1^2<0$, then
$$
0>(K_S+\sum_{i=1}^{r}(1-\beta_i)C_i). C_1=K_{S}. C_1+(1-\beta_1)C_1^2+(1-\beta_2)C_2. C_1=-\beta_1C_1^2-2\beta_2
$$
which once again does not hold
for all small $\be$.
\end{proof}

\begin{remark}
\label{remark:log-del-Pezzo-Shokurov}
We mention that
the number of connected components of the curve
$\sum_{i=1}^{r}C_{i}$ is at most $2$  \cite[Theorem~6.9]{Sho93}
(see also \cite[Proposition~2.1]{Fujino2000} for a generalization
in all dimensions). We will not use this result in the proof of
Theorem~\ref{theorem:main-2-3-4}. In fact, if
$(S,\sum_{i=1}^{r}C_{i})$ is strongly asymptotically log del
Pezzo, this also follows from Theorem~\ref{theorem:main-2-3-4}.
\end{remark}

The next example shows that the previous lemma
does not hold in the non-strongly asymptotically
log del Pezzo regime (nor in the ``diagonal
regime", i.e., where $\be=\be_1(1,\ldots,1)$),
where `interior' boundary
components could have negative self-intersection.

\begin{example}
\label{example:log-del-Pezzo-good-versus-very-good}
{\rm
 Let $S\cong\mathbb{F}_n$ for some $n>0$. Let $C_1$ and $C_2$
be two distinct fibers of the natural projection
$\mathbb{F}_n\to\mathbb{P}^1$, and let $C_3=Z_n$.
Then the pair $(S,\sum_{i=1}^{3}C_{i})$ is asymptotically log
Fano, but it is not strongly asymptotically log Fano. Indeed,
by \eqref{ampleFn} we see that the divisor
$-K_S-\sum_{i=1}^{3}(1-\be_i)C_{i})$ is ample if and only if
$\beta_1+\beta_2>n\beta_3$.
}
\end{example}

\subsection{Classification}

Note that in \S\ref{subsection:ALDP} we only assumed that
$(S,\sum_{i=1}^{r}C_{i})$ is asymptotically log del Pezzo.
In this subsection, we
assume that $(S,\sum_{i=1}^{r}C_{i})$ is strongly asymptotically
log del Pezzo (Definition \ref{definition:log-Fano}).
Namely, there exists a positive
$\epsilon\in(0,1]$ such that the divisor
\beq\label{KbetaSNCAssump}
-K_S-\sum_{i=1}^{r}(1-\beta_i)C_i
\eeq
is ample for every $\be=(\beta_1,\ldots,\beta_r)\in(0,1]^r$ with
$|\be|\le\epsilon$.

\subsubsection{Boundary with arithmetic genus one}

\begin{lemma}
\label{lemma:log-del-Pezzo-rational-ring-r-2} Suppose that
$\sum_{i=1}^{r}C_i\sim -K_{S}$.  Then $-K_{S}$ is ample,
$C_i\cong\mathbb{P}^1$ $\forall i$, and $C_i^2\ge 0$
$\forall i$. Furthermore, if $r=2$, then $|C_{1}\cap C_{2}|=2$. If
$r\ge 3$, then  $|C_{i}\cap C_{j}|\le 1$ for every
$i\ne j$, and the dual graph of the curves $C_1,\ldots,C_r$
forms a cycle.
\end{lemma}

\begin{proof}
The ampleness of $-K_{S}$ is obvious, because
$(S,\sum_{i=1}^{r}C_{i})$ is asymptotically log del Pezzo
$$
-(K_S+\sum_{i=1}^{r}(1-\beta)C_i)\sim_{\mathbb{R}} -\beta K_{S}
$$
for every real $\beta\in(0,1]$. By
Lemmas~\ref{lemma:log-del-Pezzo-rational-or-elliptic},
$C_i\cong\mathbb{P}^1$ $\forall i$. By
Lemma~\ref{lemma:strongly-non-strongly}, $C_i^2\ge 0$
$\forall i$. If $r=2$, then $|C_{1}\cap C_{2}|=2$ by
Lemma~\ref{lemma:log-del-Pezzo-rational-cycles-length-2}.
Similarly, if $r\ge 3$, then it follows from
Lemma~\ref{lemma:log-del-Pezzo-rational-cycles-length-2} that
$|C_{i}\cap C_{j}|\le 1$ for every $i\ne j$, and the dual
graph of the curves $C_1,\ldots,C_r$ forms a cycle.
\end{proof}

In analogy with
Corollary~\ref{corollary:log-del-Pezzo-rational-or-elliptic-1}, we
get:

\begin{corollary}
\label{corollary:log-del-Pezzo-rational-ring} Suppose
$\sum_{i=1}^{r}C_i\sim -K_{S}$. Then $(S,\sum_{i=1}^{r}C_{i})$ is
one of $\mathrm{(II.1A)}$, $\mathrm{(II.4A)}$, $\mathrm{(II.5A.m)}$,
$\mathrm{(II.8.m)}$, $\mathrm{(III.1)}$, $\mathrm{(III.2)}$, $\mathrm{(III.3.m)}$,
or $\mathrm{(IV)}$.
\end{corollary}

\subsubsection{Boundary with arithmetic genus zero}

To complete the proof of the classification part of
Theorem~\ref{theorem:main-2-3-4}, we may
assume that $\sum_{i=1}^{r}C_i\not\sim -K_{S}$. By
Lemma~\ref{lemma:log-del-Pezzo-rational-cycles-length-2}
$|C_{i}\cap C_{j}|\le 1$ for every $i\ne j$, and the dual
graph of the curves $C_{1},\ldots,C_{r}$ is a union of
disjoint chains. By Lemma~\ref{lemma:strongly-non-strongly},
$C_{k}^2\ge 0$ for every curve $C_k$ among $C_1,\ldots,C_r$
that intersects at least $2$ other different curves among the
curves $C_1,\ldots,C_r$.
The next lemma gives a complete classification in this situation
under the further assumption that the Picard group is small.

\begin{lemma}
\label{lemma:log-del-Pezzo-Picard-rank-two} Suppose that
$\mathrm{rk}(\mathrm{Pic}(S))\le 2$
and $C\not\sim-K_{S}$. Then when $(S,C)$
is minimal it is one of
$\mathrm{(II.1B)}$, $\mathrm{(II.2A.n)}$, $\mathrm{(II.2B.n)}$,
$\mathrm{(II.2C.n)}$, $\mathrm{(II.3)}$,
or $\mathrm{(III.3.n)}$, and otherwise it is
$\mathrm{(II.5B.1)}$.
\end{lemma}

\begin{proof}
Since $\mathrm{rk}(\mathrm{Pic}(S))\le 2$, either
$S\cong\mathbb{P}^2$ or $S\cong\mathbb{F}_n$ for some
$n\ge 0$.
If the latter case $(S,C)$ must be
$\mathrm{(II.1B)}$, as $C\not\sim-K_S$.
Assume from now on that $S=\FF_n$.

Recall that $|C_{i}\cap C_{j}|\le 1$ for every $i\ne j$, and
$C_1,\ldots,C_r$ are smooth rational curves whose dual graph is a
union of disjoint chains. If $n=0$ this determines $(S,C)$, i.e.,
the boundary $C$ must be either two disjoint fibers (II.2A.0), two
intersecting fibers (II.2C.0), a fiber and a bi-degree (1,1) curve
(II.2B.0), or three non-disjoint fibers (III.3.0).

To complete the proof
let us first consider the case $n\ge 2$.
First,
\beq\label{FirstIneqFnClassifEq}
0\le
-\Big(K_S+\sum_{i=1}^r(1-\beta_i)C_i\Big). Z_n
= 2-n-\sum_{i=1}^r(1-\beta_i)C_i. Z_n,
\eeq
so one of the curves, say $C_1$, equals $Z_n$.
If every curve $C_2,\ldots,C_r$ lies in $|F|$, then
$$
0<-\Big(K_S+\sum_{i=1}^r(1-\beta_i)C_i\Big). Z_n=-n\beta_1+2-\sum_{i=2}^r(1-\beta_i),%
$$
thus $r=2$ and $(S,C)$ is (II.2C.n).
Assume that $C_2\not\sim F$ and write $C_i\sim a_iZ_n+b_iF$.
Then since $[F]$ is nef
$$
0<-\Big(K_S+\sum_{i=1}^r(1-\beta_i)C_i\Big).
F=1+\beta_1-a_2(1-\beta_2)-\sum_{i=3}^r(1-\beta_i)C_i.
F\ge 1+\beta_1-a_2(1-\beta_2),
$$
and since $a_2>0$ this implies that $a_2=1$.
Then
$$
0<-\Big(K_S+\sum_{i=1}^r(1-\beta_i)C_i\Big).
F=\beta_1+\beta_2-\sum_{i=3}^r(1-\beta_i)C_i. F,
$$
i.e., $C_i. F=0$ for every $i\ge 3$.
Therefore, we see that $C_i\in|F|$ for every $i\ge 3$. Then
$$
-\Big(K_S+\sum_{i=1}^r(1-\beta_i)C_i\Big)\sim
(\beta_1+\beta_2)Z_n+(n+2-b_2-\sum_{i=3}^{r}(1-\beta_i))F,
$$
where $b_2\ge n$ by \eqref{irreducibleFn}. Then by \eqref{ampleFn}
$$
n+2-b_2-\sum_{i=3}^{r}(1-\beta_i)>n(\beta_1+\beta_2).
$$
If $r=2$ this implies that $b_2\le n+1$ so $b_2\in\{n,n+1\}$,
i.e., $(S,C)$ is (II.2A.n) or (II.2B.n) and if
$r=3$ then $b_2\le n$ so $b_2=n$ so
$(S,C)$ is (III.3.n).

Finally, assume $n=1$. Then \eqref{FirstIneqFnClassifEq}
implies that either $C_1=Z_1$ or $C_1\sim Z_1+F$. In the
former case the same arguments of the previous paragraph
apply to yield $(S,C)$ is either (II.2C.1), (II.2A.1), (II.2B.1),
or (III.3.1). In the latter case, if $C_2\sim F$ then
$r=2$ and $(S,C)$ is (II.5B.1) and is not minimal since
$Z_1$ intersects $C_2$ transversally at one point.
Due to \eqref{FirstIneqFnClassifEq}
the only other remaining possibility is $C_2\sim Z_1+F$
and then $(S,C)$ is (II.3). %
The proof is now complete since all the cases listed
in the statement are indeed strongly
asymptotically log del Pezzo by \S\ref{VerifSubsec}.
\end{proof}

The following is an analogue of
Lemma~\ref{lemma:log-del-Pezzo-blow-down-minimally}
for the case when a $-1$-curve contained in
the boundary is contracted. We omit the proof as
it is analogous to the proof of that lemma.

Recall that $|C_{i}\cap C_{j}|\le 1$ for every $i\ne j$, and
the dual graph of the curves $C_{1},\ldots,C_{r}$ is a union
of disjoint chains.

\begin{lemma}
\label{lemma:log-del-Pezzo-blow-down-minimally-2} Suppose that
$C_1^{2}=-1$. Then there exists a birational morphism $\pi\colon
S\to{s}$ such that ${s}$ is a smooth surface, $\pi(C_1)$
is a point, the morphism $\pi$ induces an isomorphism $S\setminus
C_k\cong {s}\setminus\pi(C_1)$, the divisor
$\sum_{i=2}^r\pi(C_i)$ is a divisor with simple normal crossings,
$|\pi(C_{i})\cap \pi(C_{j})|\le 1$ for every $i\ne j$, and
$\pi(C_2),\ldots,\pi(C_r)$ are smooth rational curves whose dual
graph is a union of disjoint chains. Moreover, the pair $({s},
\sum_{i=2}^{r}\pi(C_i))$ is strongly asymptotically log del Pezzo.
Furthermore, if the pair $(S,\sum_{i=1}^{r}C_{i})$ is minimal,
then the pair $({s}, \sum_{i=2}^{r}\pi(C_i))$ is minimal as
well.
\end{lemma}

\begin{proof}
Recall that $|C_{i}\cap C_{j}|\le 1$ for every $i\ne j$, and
$C_1,\ldots,C_r$ are smooth rational curves whose dual graph
is a union of disjoint chains. Moreover, it follows from
Lemma~\ref{lemma:strongly-non-strongly} that
$C_{1}^2\ge 0$ if $C_{1}$ intersects at least $2$ curves
among the curves $C_2,\ldots,C_r$. Arguing as in the proof of
Lemma~\ref{lemma:log-del-Pezzo-blow-down-minimally}, we obtain all
required assertions.
\end{proof}

Now we are ready to prove an analogue of
Lemma~\ref{lemma:log-del-Pezzo-minimal-good-small-Picard-r-1} that
plays a crucial role in the proof of
Theorem~\ref{theorem:main-2-3-4}.

\begin{lemma}
\label{lemma:minimalsmallPic}
 Suppose
that $(S,\sum_{i=1}^{r}C_{i})$ is minimal. Then
$\mathrm{rk}(\mathrm{Pic}(S))\le 2$.
\end{lemma}

\begin{proof}

If  $C\sim -K_S$ then $(S,C)$ is one
of the pairs listed in Corollary \ref{corollary:log-del-Pezzo-rational-ring}.
Of those,
$\mathrm{(II.1A)}$, $\mathrm{(II.4A)}$,
$\mathrm{(III.1)}$, $\mathrm{(III.2)}$, and $\mathrm{(IV)}$
are minimal and they all satisfy $\mathrm{rk}(\mathrm{Pic}(S))\le 2$.
So we assume from now on that $C\not\sim -K_S$.

Suppose that $\mathrm{rk}(\mathrm{Pic}(S))\ge 3$. Let us
derive a contradiction. 
By \eqref{MinusOneCurveExistence}
there exists a smooth rational
curve $E$ on the surface such that $E^2=-1$. Either $E\ne
C_i$ for every $i$, or there is $k$ such that $E=C_k$. 
By Lemmas~\ref{lemma:log-del-Pezzo-blow-down-minimally} and
\ref{lemma:log-del-Pezzo-blow-down-minimally-2}
and induction on $\mathrm{rk}(\mathrm{Pic}(S))$ we can assume
that $\mathrm{rk}(\mathrm{Pic}(S))=3$.

If $E\ne C_i$ for every $i$, then we can proceed exactly as in the
proof of
Lemma~\ref{lemma:log-del-Pezzo-minimal-good-small-Picard-r-1} to
obtain a contradiction. Thus, assume that 
$E=C_1$.
By Lemma~\ref{lemma:strongly-non-strongly}, $C_1$ intersects at
most one curve among the curves $C_1,\ldots,C_r$. 
Suppose that $C_1\cap C_i=\emptyset$ for
every $i\ge 3$ (if any).

Since the pair $(S,C)$ is minimal and
strongly asymptotically log del Pezzo, there exists a birational
morphism $\pi\colon S\to{s}$ as in
Lemma~\ref{lemma:log-del-Pezzo-blow-down-minimally-2}.
 Then
$({s}, \sum_{i=2}^{r}\pi(C_i))$ is minimal and
$\mathrm{rk}(\mathrm{Pic}({s}))=2$, and, in particular, 
${s}\cong\mathbb{F}_{n}$ for some $n\ge 0$.

Put ${c}_i=\pi(C_i)$ for every $i\ge 2$. Let $\xi\colon
S\to\mathbb{P}^1$ be the natural projection (it is uniquely
determined  if $n\ne 0$). Then either
$$
\pi(C_1)\not\in\bigcup_{i=2}^{r}{c}_i,
$$
(if $C_1\cap C_2=\emptyset$)
or $\pi(C_1)\in{c}_2$ and $\pi(C_1)\not\in{c}_i$ for every
$i\ge 3$ (if any) (if $C_1\cap C_2\ne\emptyset$).

We can apply Lemmas~\ref{lemma:log-del-Pezzo-Picard-rank-two-1}
and \ref{lemma:log-del-Pezzo-Picard-rank-two} to get an explicit
description of the pair $({s},\sum_{i=2}^{r}{c}_{i})$.
The cases $\mathrm{(I.1B)}$,
$\mathrm{(I.1C)}$, $\mathrm{(I.6B.1)}$, $\mathrm{(I.6C.1)}$,
$\mathrm{(II.1B)}$, and $\mathrm{(II.5B.1)}$
are excluded because either
the rank of their Picard group is one or else they 
are not minimal.
Thus, if $r\ge2$, $({s},\sum_{i=2}^{r}{c}_{i})$ 
is one of 
$\mathrm{(II.2A.n)}$, $\mathrm{(II.2B.n)}$, $\mathrm{(II.2C.n)}$,
$\mathrm{(II.3)}$, or $\mathrm{(III.2.n)}$. 
Similarly, if $r=1$, $({s},\sum_{i=2}^{r}{c}_{i})$
is one of $\mathrm{(I.2.n)}$,
$\mathrm{(I.3A)}$, $\mathrm{(I.3B)}$, $\mathrm{(I.4B)}$, or
$\mathrm{(I.4C)}$. In particular, $r$ is at most four.

Let ${f}$ be a fiber of the morphism $\xi$ that passes through
the point $\pi(C_1)$, and let $F$ be its proper transform on $S$.
Then $F$ is a smooth irreducible rational curve such that
$F^{2}=-1$. Moreover, we have $F\cap C_1\ne\emptyset$ by
construction. Since $(S,\sum_{i=1}^{r}C_{i})$ is minimal, the
curve $F$ must be one of the curves $C_2,\ldots,C_r$. Then $F=C_2$,
$C_1\cap C_2\ne\emptyset$, $\pi(C_1)\in c_2$, and
$\pi(C_1)\not\in{c}_i$ for every $i\ge 3$ (if any).
Moreover, it follows from
Lemma~\ref{lemma:strongly-non-strongly} that
$C_2$ does not intersect any curve among $C_3,\ldots,C_r$ (if
any), since $F^2=-1$. Thus, ${c}_2$ does not intersect any
curve among ${c}_3,\ldots,{c}_r$ (if any).

Suppose that $r=2$. Then $({s},{c}_2)$ is
one of $\mathrm{(I.2.n)}$,
$\mathrm{(I.3A)}$, $\mathrm{(I.3B)}$, $\mathrm{(I.4B)}$, or
$\mathrm{(I.4C)}$. The latter is possible only in the case
$\mathrm{(I.2.0)}$ since ${c}_2\in|f|$ is a fiber of the the
morphism $\xi$. Then ${s}\cong\mathbb{P}^1\times\mathbb{P}^1$.
The latter implies that $(S,C_1+C_2)$ is not minimal. Indeed, the
surface $S$ is a del Pezzo surface with $K_{S}^2=7$. It contains
three $(-1)$-curves. Two of them are the curves $C_1$ and $C_2$.
The third one intersects $C_2$, contradicting minimality. 

Suppose that $r=3$. As noted earlier then $({s},{c}_2+{c}_3)$ is
one of $\mathrm{(II.2A.n)}$, $\mathrm{(II.2B.n)}$, $\mathrm{(II.2C.n)}$,
or $\mathrm{(II.3)}$. Since
${c}_2\cap{c}_3=\emptyset$ it must be (II.2A.n). 
Since ${c}_2$ is a fiber of the  morphism $\xi$
and $c_2\cap c_3=\emptyset$ it follows that $c_3$
is also a fiber of $\xi$. Thus, $n=0$, i.e.,
${s}\cong\mathbb{P}^1\times\mathbb{P}^1$. 
The latter implies that $(S,C_1+C_2)$ is not asymptotically
log del Pezzo. Indeed, the surface $S$ is del Pezzo surface with
$K_{S}^2=7$. It contains three $(-1)$-curves. Two of them are the
curves $C_1$ and $C_2$. The third one intersects $C_2$ and $C_3$,
contradicting Lemma~\ref{lemma:log-del-Pezzo-blow-down}.

Suppose that $r=4$. Then $({s},{c}_2+{c}_3+{c}_4)$ 
must be $\mathrm{(III.2.n)}$. But this is precluded
by the fact that $c_2$ does not intersect ${c}_3$ or
${c}_4$.

In conclusion then $\mathrm{rk}(\mathrm{Pic}(S))\le 2$.
\end{proof}

We now complete the proof of the classification part of  
Theorem \ref{theorem:main-2-3-4}.
If
$\mathrm{rk}(\mathrm{Pic}(S))\le 2$, then 
$(S,C)$ is listed in
Corollary \ref{corollary:log-del-Pezzo-rational-ring}
if $C\sim-K_S$ and by 
Lemma~\ref{lemma:log-del-Pezzo-Picard-rank-two}
if $C\not\sim-K_S$.
On the other hand, if
$\mathrm{rk}(\mathrm{Pic}(S))> 2$, the pair
$(S,\sum_{i=1}^{r}C_{i})$ is not minimal by 
Lemma \ref{lemma:minimalsmallPic}.
But then, Lemmas 
\ref{lemma:log-del-Pezzo-blow-down-minimally} and
\ref{lemma:log-del-Pezzo-blow-down-minimally-2} imply
$(S,C)$ is a blow-up along the boundary of one of the minimal pairs that
we already listed. It remains to check
the genericity
conditions on the location of the blow-up points as 
stated in Theorem \ref{theorem:main-2-3-4}.
This is carried out in \S\ref{VerifSubsec}
where we simultaneously also verify
that all pairs listed in the theorem are indeed
strongly asymptotically del Pezzo.

\begin{remark}
The results of this section already give some hints as
to the difficulties in classifying all
asymptotically log del Pezzo surfaces.
In particular, $r$ can then be infinite,
and $-1$-curves can appear as `interior' curves of the boundary
(see Example \ref{example:log-del-Pezzo-good-versus-very-good})
even though the number of connected components of the support
of $C$ is still two by Remark \ref{remark:log-del-Pezzo-Shokurov}.
However, the classification of the `diagonal' regime
(where $-K_S-\sum(1-\be_i)C_i$ is ample for all
sufficiently small $\be$ of the form $\be=\be_1(1,\ldots,1)$)
should be more tractable. In a related vein, results of di Cerbo--di Cerbo
\cite{DiCerbo2013} give bounds
on the largest possible value $\be_1$ may take 
for pairs of class $(\daleth)$ in this last regime
depending only on $(K_S+C)^2$. We further note that
di Cerbo \cite{DiCerbo} considered the `diagonal' regime
in the setting of negative curvature,
and obtained necessary and sufficient intersection-theoretic
restrictions on the pair for $K_S+(1-\be)\sum C_i$ to be ample.
Of course, in the negative setting a complete classification
is lacking even in the smooth setting with no boundary.
\end{remark}

\section{Positivity properties of the
logarithmic anticanonical bundle}
\label{section:posititivy}

Let $S$ be a smooth surface, let $C_1,\ldots,C_r$ be smooth
irreducible curves on the surface $S$ such that
$\sum_{i=1}^{r}C_{i}$ is a divisor with simple normal crossings,
and let $\beta=(\beta_1,\beta_2,\ldots,\beta_r)\in (0,1]^r$, where
$r\ge 1$. Suppose that $(S,\sum_{i=1}^{r}C_{i})$ is strongly
asymptotically log del Pezzo. Then  we have the following mutually
excluding possibilities:
\begin{itemize}
\item[$\mathrm{(\aleph)}$] $-(K_{S}+\sum_{i=1}^{r}C_{i})\sim 0$,
$S$ is del Pezzo surface, and $\sum_{i=1}^{r}C_{i}\sim -K_{S}$,
and $C_i^2\ge 0$ $\forall i$,

\item[$\mathrm{(\beth)}$] $(-K_{S}-\sum_{i=1}^{r}C_{i})^2=0$, all
curves $C_1,\ldots,C_r$ are rational, and the dual graph of the
curves $C_1,\ldots,C_r$ is a disjoint union of chains,

\item[$\mathrm{(\gimel)}$] the divisor
$-(K_{S}+\sum_{i=1}^{r}C_{i})$ is big and nef, the divisor
$-(K_{S}+\sum_{i=1}^{r}C_{i})$ is not ample, all curves
$C_1,\ldots,C_r$ are rational and the dual graph of the curves
$C_1,\ldots,C_r$ is a disjoint union of chains,

\item[$\mathrm{(\daleth)}$] $C\cong\mathbb{P}^1$, the divisor
$-(K_{S}+\sum_{i=1}^{r}C_{i})$ is ample, all curves
$C_1,\ldots,C_r$ are rational and the dual graph of the curves
$C_1,\ldots,C_r$ is a disjoint union of chains.
\end{itemize}
This follows from
Lemmas~\ref{lemma:lemma:log-del-Pezzo-rational-curves-1-1},
\ref{lemma:log-del-Pezzo-rational-or-elliptic-1},
\ref{lemma:strongly-non-strongly} and
\ref{lemma:log-del-Pezzo-rational-cycles-length-2}. Note that the
classification in Theorems~\ref{theorem:main-1} and
\ref{theorem:main-2-3-4} generalizes Maeda's classical result
\cite{Maeda} that corresponds to class $\mathrm{(\daleth)}$.
Maeda further provided a full classification in class
$\mathrm{(\daleth)}$. In \S\ref{PositivitySubsec} we generalize this by giving
a complete classification in each of the remaining classs.
In \S\ref{PositivityTwoSubsec} we go further in the class
$\mathrm{(\beth)}$ by proving that
the linear system
$|-K_{S}-\sum_{i=1}^{r}C_{i}|$ gives a morphism
$S\to\mathbb{P}^1$ whose general fiber is $\mathbb{P}^1$.

\begin{remark}{\rm
There is perhaps no real need to distinguish between classes
$(\gimel)$ and $(\daleth)$, but we do that mainly for
a historical reason. Indeed,
class $\mathrm{(\daleth)}$
is not new. These pairs
were completely classified in Maeda's work who coined
the term `log del Pezzo surface' for this class of pairs \cite{Maeda}.
In higher dimensions it might prove more natural to
identify only $\dim X+1$ classes according to the Kodaira dimension
of $-K_X-D$.
}
\end{remark}

\subsection{Positivity classification}
\label{PositivitySubsec}

We now prove Theorem \ref{theorem:4-cases}, relying
on Theorems \ref{theorem:main-1} and \ref{theorem:main-2-3-4}.

Class $(\aleph)$ follows from Corollaries
\ref{corollary:log-del-Pezzo-rational-or-elliptic-1}
(and the remark preceeding it) and \ref{corollary:log-del-Pezzo-rational-ring}.
Class $(\daleth)$ follows from \eqref{ampleFn}.

Next, if $(S,\sum_{i=1}^{r}C_{i})$ is not minimal (see
Definition~\ref{definiton:log-del-Pezzo-minimally-good-1}), then
it follows from the proof of Theorems~\ref{theorem:main-1} and
\ref{theorem:main-2-3-4} that there exists a non-biregular
birational map $\pi\colon S\to{s}$ such that the pair
$({s}, \sum_{i=1}^{r}\pi(C_{i}))$ is minimal and
$$
-K_{S}+\sum_{i=1}^{r}C_{i}\sim
\pi^{*}(-K_{{s}}-\sum_{i=1}^{r}\pi(C_{i})).
$$
Indeed, by our construction, each $-1$-curve that is contracted
intersects the boundary transversally exactly at one point.
This also shows that
$-K_{S}-\sum_{i=1}^{r}C_{i}$ can not be ample if
$(S,\sum_{i=1}^{r}C_{i})$ is not minimal, because
$-K_{S}-\sum_{i=1}^{r}C_{i}$ intersects all $\pi$-exceptional
curves trivially.
In sum, if $(S,C)$ is of
class $\mathrm{(\aleph)}$, respectively $\mathrm{(\beth)}$,
then so is $(s,c)$, and if $(S,C)$
is of class $\mathrm{(\gimel)}$ or $\mathrm{(\daleth)}$
$(s,c)$ is of class $\mathrm{(\gimel)}$.
This completes the verification of class $\mathrm{(\gimel)}$
since each of these pairs are blow-ups of a pair
of class $(\daleth)$, while the pairs
$\mathrm{(I.3A})$, $\mathrm{(I.4B})$,
$\mathrm{(II.2A.n})$, $\mathrm{(II.2B.n})$, $\mathrm{(II.3})$,
and $\mathrm{(III.3.n})$ all satisfy $(K_S+C)^2=0$.
Class $(\beth)$ then contains, by exclusion, all the remaining
pairs listed in Theorems \ref{theorem:main-1} and \ref{theorem:main-2-3-4}.

\subsection{Verification of the list}
\label{VerifSubsec}

Using the positivity classification of the original lists of Theorems
\ref{theorem:main-1} and \ref{theorem:main-2-3-4},
we now verify that indeed each of the pairs listed there
is strongly asymptotically log del Pezzo.
This is the last step remaining to complete the proof of the main classification result, Theorem
\ref{theorem:main}.

The Maeda case $\mathrm{(\daleth)}$ is immediate by convexity as then $-K_S-C$
itself is ample, and so is the case $\mathrm{(\aleph)}$.
So suppose $(S,C)$ is a pair of class
$\mathrm{(\beth)}$ or $\mathrm{(\gimel)}$
listed in Theorem \ref{theorem:4-cases}.
Then there exists a blow-down map $\pi:S\ra s$
such that the pair $(s,c)$ is minimal where $c=\pi(C)$.
Then
\beq\label{KSPullBackEq}
-K_{S}
-
(1-\beta)C\equiv\pi^{*}\Big(-\big(K_{s}
+
(1-\beta){c}\big)\Big)-\sum_{i=1}^{m}\beta E_i.%
\eeq
Here $E_i=\pi^{-1}(P_i)$, with $P_1,\ldots,P_m$ denoting
the blow-up points.
The slight subtlety is that while the second
term on the right is `small' in terms of its
contribution to intersection numbers and
the first term is ample, the latter also
depends on $\be$ and so a priori it is
not clear which term will dominate.
In fact, the following example illustrates
a situation 
where such a problem arises.

\begin{example}
\label{example:log-del-Pezzo-good-versus-very-good-3}
{\rm
Consider the surface $\FF_n$ and let
$R$ be some smooth curve in $|Z_n+nF|$.
Then $(\mathbb{F}_n,
Z_n+F+R)$ is strongly asymptotically log Fano. Let $\pi\colon
S\to\mathbb{F}_n$ be a blow-up of $m$ distinct points
in the smooth locus of $Z_n+F+R$ such that no two
of the points lie on one curve in the linear
system $|F|$. Let $C_1$, $C_2$, $C_3$ be the proper transforms of
the curves $Z_n$, $F$, $R$, respectively. Then $(S, C_1+C_2+C_3)$ is
asymptotically log Fano. On the other hand, $(S, C_1+C_2+C_3)$ is
strongly asymptotically log Fano if and only if none of
the blow-up points lies in $F$ (cf. $\mathrm{(III.5.n.m})$ in
Theorem~\ref{theorem:main-2-3-4}).
}\end{example}

We now go through the lists of classes $\beth$ and
$\gimel$ and verify the pairs
are strongly asymptotically log del Pezzo.
We assume
without mention that $\be$ is taken in each equation to
be sufficiently small, depending only on $(S,C)$.

By the Nakai--Moishezon criterion \eqref{NMEq}, we
have to check that $(K_{S}+(1-\beta)C)^2>0$ and
$-(K_{S}+(1-\beta)C). Z>0$ for every irreducible curve
$Z\subset S$, with $\be$ independent of $Z$
(and we use, e.g., the notation
$O_{n,m}(\be_1)$
to denote a quantity bounded by $C\be_1$
with $C$ depending only on $n,m,$ and $(S,C)$).

To do this, let us fix some
irreducible curve $Z$ on the surface $S$.
We may assume that $Z$ is not $\pi$-exceptional since
by \eqref{KSPullBackEq} and the fact that all the blow-up points
are distinct (so none of the exceptional divisors intersect)
$-(K_{S}+(1-\beta)C)Z>0$ if $Z$ is $\pi$-exceptional.
Put $z:=\pi(Z)$, and suppose that
$\mathrm{mult}_{P_{1}}({z})\le
\cdots\le
\mathrm{mult}_{P_{m}}({z})$.

\medskip
\noindent
{\bf Class $\mathrm{(\gimel)}$.}
Suppose that we are either in case $\mathrm{(I.6B.m)}$ or
$\mathrm{(I.6C.m)}$. Then $S\cong\mathbb{P}^2$, and
${c}$ is a conic in the case $\mathrm{(I.6B.m)}$ or a line in
the case $\mathrm{(I.6C.m)}$. Let $\delta$ be the degree of the
curve ${c}$ in $S\cong\mathbb{P}^{2}$, i.e., either $\delta=2$
in the case $\mathrm{(I.6B.m)}$ or $\delta=1$ in the case
$\mathrm{(I.6C.m)}$. Let $l$ be a line in $S$. Then
$
-(K_{S}+(1-\beta)C)\equiv\pi^{*}((3-(1-\beta)\delta){l})-\sum_{i=1}^{m}\beta E_i,%
$
which implies that
$(K_{S}+(1-\beta)C)^2=(3-\delta+\delta\beta)^2-m\beta^2>0$
since $\delta\in\{1,2\}$. Let $d$ be
the degree of the curve ${z}$ in $S\cong\mathbb{P}^2$. Then
$
-(K_{S}+(1-\beta)C). Z=(3-(1-\beta)\delta)d-\sum_{i=1}^{m}\beta \mathrm{mult}_{P_{i}}({z})\ge d-m\beta\mathrm{mult}_{P_{m}}({z})>0%
$,
concluding these cases.

Now consider the case $\mathrm{(I.7.n.m)}$.
Then $s\cong\mathbb{F}_{n}$, and ${c}=Z_n$.
Let ${f}$ be a fiber of the natural projection
$s\to\mathbb{P}^{1}$, i.e., ${f}$ is a smooth irreducible
rational curve such that ${f}.{c}=1$ and
${f}^2=0$. Then $-K_{s}\sim 2{c}+(2+n){f}$, so
$
-(K_{S}+(1-\beta)C)\equiv\pi^{*}((1+\beta){c}+(2+n){f})-\sum_{i=1}^{m}\beta E_i%
$
and
$(K_{S}+(1-\beta)C)^2=4+n+4\beta-n\beta^2-m\beta^2>0$.
 Note that ${z}\sim
a_1{c}+a_2{f}$ for some non-negative integers $a_1,a_2$ such
that either $(a_1,a_2)=(1,0)$ (if $n\ge 1$, then $Z=C$ in this
case), or $(a_1,a_2)=(0,1)$ (this means that ${z}$ is a fiber of
the projection $s\to\mathbb{P}^1$), or $a_2\ge na_1$ (see
\cite[Corollary~2.18]{Har77}). Then
$
-(K_{S}+(1-\beta)C). Z=a_2+2a_1+\beta(a_2-na_1)-\sum_{i=1}^{m}\beta\mathrm{mult}_{P_{i}}({z})> 0
$.
Indeed, since the divisor ${c}+(n+1){f}$ is
very ample \cite[Theorem~2.17]{Har77}, we get
a uniform bound on the multiplicity:
$$
\mathrm{mult}_{P_{m}}({z})\le({c}+(n+1){f}){z}=a_2+a_1.
$$
This concludes this case.

The case $\mathrm{(I.8B.m})$ is treated similarly.

Suppose now  $(S,C)$ is $\mathrm{(I.9C.m)}$.
Then $s\cong\mathbb{P}^1\times\mathbb{P}^1$.
Let ${f}_1,{f}_2$ be fibers of the two natural projections
$s\to\mathbb{P}^1$. Then $c\sim f_1+f_2$ and
$-K_{s}\sim 2{f}_1+2{f}_2$, and ${z}\sim
a_1{f}_1+a_2{f}_2$ for some non-negative integers $a_1$
and $a_2$ such that $(a_1,a_2)\ne (0,0)$.
Then
$
-K_{S}-(1-\beta)C =\pi^{*}((1+\beta){f}_1+(1+\beta){f}_2)-\sum_{i=1}^{m}\beta E_i,%
$
and $(K_{S}+(1-\beta)C)^2=2(1+\beta)^2-m\beta^2>0$.
Since $f_1+f_2$ is very ample
$\mathrm{mult}_{P_{m}}({z})\le
a_1+a_2$.
Thus,
$
-(K_{S}+(1-\beta)C). Z=(a_1+a_2)(1+\beta)-\sum_{i=1}^{m}\beta\mathrm{mult}_{P_{i}}({z})\ge (a_1+a_2)(1+\beta)-m\beta\mathrm{mult}_{P_{m}}({z})%
>0$.

Finally, one readily checks that
the case $\mathrm{(II.5B.m})$ reduces to
$\mathrm{(II.1B})$ which in turns is essentially identical
to $\mathrm{(I.1B})$. Similarly, $\mathrm{(II.6C.n.m})$
reduces to $\mathrm{(II.2C.n})$ which is in the class
$\mathrm{(\daleth)}$.

\medskip
\noindent
{\bf Class $\mathrm{(\beth)}$.}
The cases $\mathrm{(I.3A})$ and $\mathrm{(I.4B})$ are immediate.
Suppose we are in the case
$\mathrm{(I.9B.m)}$.
Then
$
-K_{S}-(1-\beta)C\equiv\pi^{*}(2\beta{f}_1+(1+\beta){f}_2)-\sum_{i=1}^{m}\beta E_i,
$
so
$(K_{S}+(1-\beta)C)^2=4\beta(1+\beta)-m\beta^2>0$, and
\begin{equation}
\label{equation:log-del-Pezzo}
-(K_{S}+(1-\beta)C). Z=2\beta a_2+(1+\beta)a_1-\sum_{i=1}^{m}\beta\mathrm{mult}_{P_{i}}({z})>0
\end{equation}
since if $z\in|f_2|$ (i.e., $(a_1,a_2)=(0,1)$) then
$z$ passes through at most one of the blow-up points
and in this case $z$  (a fiber) is also necessarily smooth,
so $\mathrm{mult}_{P_{i}}({z})\in\{0,1\}$.

Among $\mathrm{(II.2A.n})$ and $\mathrm{(II.2B.n})$
it suffices to check the latter. In fact,
since
$\mathrm{(II.6A.n.m})$ and $\mathrm{(II.6B.n.m})$
are their blow-ups we only need to consider
$\mathrm{(II.6B.n.m})$ (allowing $m$ to possibly
equal $0$). In this case,
$
-K_{S}-(1-\beta)C
=
\pi^{*}\big((\be_1+\beta_2)Z_n+(1+(n+1)\be_2)F\big)
-\sum_{i=1}^{k}\beta_1 E_i
-\sum_{i=k+1}^{m}\beta_2 E_i,
$
assuming that exactly the first $k$ points are
blown-up along $\pi(C_1)=Z_n$.
The square of this class is then $2\be_1+2\be_2+
O_{n,m}(\be_1\be_2+\be_1^2+\be_2^2)>0$, and
its intersection with $Z$ (such that $z\sim a_1Z_n+a_2F$)
equals $-na_1(\be_1+\be_2)+a_1(1+(n+1)\be_2)+a_2(\be_1+\be_2)
-\be_1\sum_{i=1}^k\mathrm{mult}_{P_{i}}({z})
-\be_2\sum_{i=k+1}^m\mathrm{mult}_{P_{i}}({z})
=
a_1(1+b_2-n\be_1)+a_2(\be_1+\be_2)
-\be_1\sum_{i=1}^k\mathrm{mult}_{P_{i}}({z})
-\be_2\sum_{i=k+1}^m\mathrm{mult}_{P_{i}}({z})
$.
This is positive if $a_1>0$ since, as before,
the multiplicities are uniformly bounded
independently of $z$.
If $a_1=0$ then $a_2>0$, so $a_2=1$
as $z$ is irreducible, thus a fiber.
Then the intersection number
is positive (bounded below by $\min\{\be_1,\be_2\}$)
provided the fiber does not intersect
more than one of the $P_i$.

The case $\mathrm{(II.7.m})$  (that implies the case $\mathrm{(II.3})$)
is proven using very similar computations.

Finally we consider $\mathrm{(III.5.n.m})$
(that takes care of the case $\mathrm{(III.3.n})$).
Then
$
-K_{S}-(1-\beta)C
=
\pi^{*}\big((\be_1+\beta_3)Z_n+(1+\be_2+n\be_3)F\big)
-\sum_{i=1}^{k}\beta_1 E_i
-\sum_{i=k+1}^{m}\beta_3 E_i
$
This squares to
$-n(\be_1+\be_3)^2+2(\be_1+\be_3)(1+\be_2+n\be_3)
-k\be_1^2-(m-k-1)\be_3^2>0$ (here we see why
blow-ups along $\pi(C_2)$ are prohibited).
The verification of the intersection with $Z$ is
as in the previous case.

The proof of Theorem \ref{theorem:main} is now complete.

\subsection{Nef and non-big adjoint anticanonucal bundle}
\label{PositivityTwoSubsec}

In the case $\mathrm{(\beth)}$, the linear system
$|-(K_{S}+\sum_{i=1}^{r}C_{i})|$ gives a morphism
$S\to\mathbb{P}^1$ whose general fiber is $\mathbb{P}^1$. This can
be shown by using our classification in
Theorems~\ref{theorem:main-1} and \ref{theorem:main-2-3-4}
or alternatively (and in any dimension) from Kawamata--Shokurov's
results as demonstrated in Theorem \ref{theorem:KS}. But we
prefer to give a self-contained classification-free proof
of Proposition \ref{proposition:conic-bundle} that
does not rely on these deep works.

In the remaining part of this subsection we prove
Proposition~\ref{proposition:conic-bundle}. Suppose that
$(K_{S}+\sum_{i=1}^{r}C_{i})^2=0$. By
Lemma~\ref{lemma:log-del-Pezzo-rational-cycles-length-2}, the dual
graph of the curves $C_{1},\ldots,C_r$ is a disjoint union of
chains. Let $l$ be the number of connected components of the curve
$\sum_{i=1}^{r}C_{i}$ (by
Remark~\ref{remark:log-del-Pezzo-Shokurov} one has $l\le
2$ but we will not use it here).

\begin{lemma}
\label{lemma:conic-bundles-1}  One has
$h^{0}(\mathcal{O}_{S}(-(K_{S}+\sum_{i=1}^{r}C_{i})))=1+l$.
\end{lemma}

\begin{proof}
Since the dual graph of the curves $C_{1},\ldots,C_r$ is a
disjoint union of chains, one can easily check that
$$
(K_{S}+\sum_{i=1}^{r}C_{i}). (\sum_{i=1}^{r}C_{i})=-2l.
$$
This allows us to compute
$h^{0}(\mathcal{O}_{S}(-(K_{S}+\sum_{i=1}^{r}C_{i})))$. Indeed, we
have
$$
h^{2}(\mathcal{O}_{S}(-(K_{S}+\sum_{i=1}^{r}C_{i})))=h^{1}(\mathcal{O}_{S}(-(K_{S}+\sum_{i=1}^{r}C_{i})))=0
$$
by
\eqref{KbetaSNCAssump} and
the Kawamata--Viehweg Vanishing Theorem.
Therefore, it follows from the Riemann--Roch
Theorem that
$$\baeq
h^{0}\Big(\mathcal{O}_{S}\big(-(K_{S}+\sum_{i=1}^{r}C_{i})\big)\Big)
&=1+\frac{(K_{S}+\sum_{i=1}^{r}C_{i}).(2K_{S}+\sum_{i=1}^{r}C_{i})}{2}
\cr
&=1-(K_{S}+\sum_{i=1}^{r}C_{i}).(\sum_{i=1}^{r}C_{i})=1+l,%
\eaeq
$$
because $(-K_{S}-\sum_{i=1}^{r}C_{i})^2=0$ by assumption.
\end{proof}

Thus, we see that $|-(K_{S}+C)|$ is at least a pencil. Moreover,
if $l=1$, then it is a pencil, since $S$ is rational. Note that we
can use \cite[Theorem~6.9]{Sho93} to show that $l\le 2$. But we do
not need this. In fact, one can show that $l\le 2$ using
Lemma~\ref{lemma:conic-bundles-1} (cf. the proof of
\cite[Theorem~6.9]{Sho93}).

\begin{lemma}
\label{lemma:conic-bundles-2}  The linear system
$|-K_{S}-\sum_{i=1}^{r}C_{i}|$ is a base point free.
\end{lemma}

\begin{proof}
Let us first show that $|-(K_{S}+C)|$ is free from fixed
components (see \cite[Theorem~III.1]{Harb}). Suppose this is not
the case. Let $B$ be the fixed part of the linear system
$|-(K_{S}+C)|$, and let $M$ be its mobile part. Then $M$ is nef.
In particular, we have
$h^{1}(\mathcal{O}_{S}(M))=h^{2}(\mathcal{O}_{S}(M))=0$ by the
Kawamata-–Viehweg vanishing theorem. Then it follows from the
Riemann--Roch theorem that
$$
2l=h^{0}(\mathcal{O}_{S}(M+B)=h^{0}(\mathcal{O}_{S}(M)=1+\frac{M. (M-K_{S})}{2},%
$$
which implies that $M^2-M. K_{S}=4l-2$. On the other hand, we
have
\begin{multline*}
0=(K_{S}+\sum_{i=1}^{r}C_{i})^2=(B+M)^2=B^2+2B. M+M^2=\\
=B. (B+M)+B. M+M^2=-(K_{S}+\sum_{i=1}^{r}C_{i}).
B+B. M+M^2\ge 0,\\
\end{multline*}
since both $-(K_{S}+\sum_{i=1}^{r}C_{i})$ and $M$ are nef.
Hence, we have $M^2=0$ and $B. M=0$, which implies that
$B^2=0$, since $(B+M)^2=0$.

We claim that $B$ is nef. Indeed, put $B=\sum_{i=1}^{k}a_{i}B_i$,
where $B_i$ is an irreducible curve, and $a_i$ is a positive
integers. Then
$$
0=(B+M).\Big(\sum_{i=1}^{k}a_{i}B_i\Big)\ge \sum_{i=1}^{k}a_{i}(B+M). B_i,%
$$
which implies that $(B+M). B_i=0$ for every possible $i$.
Similarly, we see that $M. B_i=0$ for every possible $i$,
which implies that $B. B_i=0$ for every possible $i$. Hence,
the divisor $B$ is nef.

Since $B$ is nef, we have
$h^{1}(\mathcal{O}_{S}(B))=h^{2}(\mathcal{O}_{S}(B))=0$ by the
Kawamata–Viehweg vanishing theorem. Applying  the Riemann--Roch
theorem to the divisor $B$, we see that
$$
h^{0}(\mathcal{O}_{S}(B))=1+\frac{B. (B-K_{S})}{2}=1-\frac{B. K_{S})}{2}\ge 0,%
$$
since $-K_{S}\sim \sum_{i=1}^{r}C_{i}+B+M$ and $B$ is nef. But
$h^{0}(\mathcal{O}_{S}(B)=1$, because $B$ is the fixed part of the
linear system $|-(K_{S}+\sum_{i=1}^{r}C_{i})|$. The latter implies
that $-K_{S}. B=0$. Since $B\ne 0$ by assumption, the
Riemann--Roch theorem implies that
\begin{multline*}
h^{2}(\mathcal{O}_{S}(-B))=h^{0}(\mathcal{O}_{S}(-B))+h^{2}(\mathcal{O}_{S}(-B))=\\
=1+h^{1}(\mathcal{O}_{S}(-B))+\frac{B.
(B-K_{S})}{2}=1+h^{1}(\mathcal{O}_{S}(-B))\ge 1,\\
\end{multline*}
which implies that $h^{2}(\mathcal{O}_{S}(-B))\ne 0$. By Serre
duality, we have
$h^{2}(\mathcal{O}_{S}(-B))=h^{0}(\mathcal{O}_{S}(B+K_{S}))$. But
$$
B+K_{S}\sim -\sum_{i=1}^{r}C_{i}-M,
$$
which implies that $h^{0}(\mathcal{O}_{S}(B+K_{S}))=0$, which is a
contradiction. Thus, the linear system
$|-(K_{S}+\sum_{i=1}^{r}C_{i})|$ is free from fixed curves.

Since $|-(K_{S}+\sum_{i=1}^{r}C_{i})|$ is free from fixed curves
and $(-K_{S}-\sum_{i=1}^{r}C_{i})^2=0$, the linear system
$|-(K_{S}+\sum_{i=1}^{r}C_{i})|$ does not have base points at all.
\end{proof}

Since $(-K_{S}-\sum_{i=1}^{r}C_{i})^2=0$, the linear system
$|-(K_{S}+\sum_{i=1}^{r}C_{i})|$ is composed from a base point
free pencil. By Bertini theorem, there exists a smooth irreducible
curve $F$ such that $F^2=0$, the linear system $|F|$ is a base
point free pencil, and
$$
-K_{S}-\sum_{i=1}^{r}C_{i}\sim kF
$$
for some positive integer $k$. Since $-K_{S}$ is big, we have
$-K_{S}. F>0$. Hence, we have $-K_{S}. F=2$ and
$F\cong\mathbb{P}^1$ by adjunction formula. Then it follows from
the Riemann--Roch theorem that $h^0(\mathcal{O}_{S}(kF))=k+1$,
which implies that $k=l$.

We may assume that $F$ is a general curve in $|F|$. The pencil
$|F|$ gives a morphism $\xi\colon S\to\mathbb{P}^1$ whose general
fiber is $F\cong\mathbb{P}^1$, i.e., the morphism $\xi$ is a
\emph{conic bundle}. Since $-K_{S}. F=2$ and
$-(K_{S}+\sum_{i=1}^{r}C_{i}). F=0$, we have $F.
(\sum_{i=1}^{r}C_{i})=2$.

For every irreducible curve $Z$ on the surface that is contained
in the fibers of $\xi$, we have
$$
0<-(K_{S}+\sum_{i=1}^r(1-\beta_i)C_i). Z\sim_{\mathbb{R}} F. Z+\sum_{i=1}^r\beta_iC_i. Z=\sum_{i=1}^r\beta_iC_i. Z,%
$$
for all $0<|\beta|\ll 1$, because $(S,\sum_{i=1}^{r}C_{i})$ is
strongly asymptotically log del Pezzo (note that this step does
not work if $(S,\sum_{i=1}^{r}C_{i})$ just asymptotically log del
Pezzo). This implies that $\sum_{i=1}^{r}C_{i}. Z>0$. Keeping
in mind that $\sum_{i=1}^{r}C_{i}. F=2$, we see that either
$\sum_{i=1}^{r}C_{i}. Z=1$ or $\sum_{i=1}^{r}C_{i}. Z=2$.
In the latter case, we must have $Z\sim F$. This implies that
$\xi$ is so-called \emph{standard conic bundle}, i.e., every
singular fiber of $\xi$ consists of a union of two smooth rational
curves that intersect each other transversally at one point.

\section{Reductivity of the automorphism group of a pair}
\label{ReductiveSec}

Denote by $\aut(X)$ the Lie algebra of holomorphic vector fields
on $(X,J)$, i.e., all vector fields $V\in\Gamma(X,TX)$
satisfying $L_VJ=0$. We emphasize that these are real vector fields.
The projection of $V$ onto $T^{1,0}X$, denoted
$V^{1,0}=(V-\i JV)/2$ will
be referred to as a holomorphic (1,0)-vector field, and
it is sometimes convenient to work with $\aut(X)$
recast in this complex notation.
Let  $\aut(X,D)\subset\aut(X)$
denote the subspace of fields tangent to $D$. It
is a Lie subalgebra.

\begin{prop}
\label{reductiveprop}
Let $(X,D,\o)$ be a KEE manifold. Then
$\aut(X,D)$ is the complexification of the Lie algebra
of (Hamiltonian) Killing vector fields of $(X,\o)$.
\end{prop}

\bpf
Suppose $V=\nabla u\in\aut(X,D)$ is a gradient holomorphic vector field
(here $u$ is a real-valued function).
We claim that $JV$ is a Killing field with respect to $g$ (the metric
associated to $\o$).
Indeed, this is equivalent to $Z\mapsto \nabla_Z(JV)=J\nabla_ZV$ being a skew-symmetric
endomorphism of $TX$ \cite[Proposition 27]{Petersen2nd}. But
$\nabla (J\nabla u)
=
\nabla
(\i\nabla^{1,0}u-\i\nabla^{0,1}u)
=
\i(
\nabla^{0,1}\nabla^{1,0}u-\nabla^{1,0}\nabla^{0,1}u)
=-\i\ddbar u$, since
$\nabla^{1,0}V^{1,0}=\nabla^{1,0}\nabla^{1,0}u=0$.

Next, we claim that any element $X$ of $\aut(X,D)$ is necessarily
a linear combination of a gradient vector field
and $J$ applied to such a field. In fact, consider the
$(0,1)$-form $g(V^{1,0},\,.\,)$ given in local coordinates
by $g_{i\b j}V^i\overline{dz^j}$. Since $\nabla^{1,0}V=0$,
this form is closed. Thus, by
Lemma~\ref{boneLemma} below,
it equals a $\dbar$-exact form, say $\dbar u/2$ with
$u$ complex-valued. It follows that $V^{1,0}=\nabla^{1,0}u/2$, and
$V^{0,1}=\nabla^{0,1}\b u/2$, so
\beq
\label{XDecompEq}
V=\nabla\Re\, u+J\nabla\Im\, u.
\eeq

The same argument also shows that any Killing field $V$ is
necessarily a Hamiltonian vector field. In fact,
an isometry homotopic to the identity preserves any $\o$-harmonic form
by Hodge theory \cite[p. 82]{GH} since it preserves its class.
Thus, $L_V\o=0$, or $\iota_V\o=0$. Since $b^1(V)=0$
then $\iota_V\o=du$ and $V=-J\nabla u$.
Further, since $L_Vg=0$, $L_V\o=0$, and $\o(\,.\,,\,.\,)=g(J\,.\,,\,.\,)$,
then also $L_VJ=0$. Next, note that any automorphism
of $(X,\o)$ must preserve its singular set, i.e., $D$.
Thus, combining all the above, $\aut(X,D)$ must contain
the Lie algebra of Killing fields of $(X,\o)$.

Thus, we would be done if we knew that each summand
in the decomposition \eqref{XDecompEq} of a holomorphic vector field
$V$ were itself a holomorphic vector field (then $V$
would be equal to a Killing field and $J$ times such a field, by
the previous paragraphs).
To show that, recall that as shown in \cite[\S6]{JMR},
if $\o$ is a KEE metric
of positive Ricci curvature $\mu$, and $\phi$ is
a complex-valued eigenfunction of $-\Delta_\o$ with eigenvalue $\mu$
then $\nabla^{1,0}\phi$ is a holomorphic (1,0)-vector field tangent to
$D$. We claim that the converse is true as well.
Assuming this claim, there is an isomorphism
between $\Lambda_\mu(-\Delta_\o)$ (the aforementioned eigenspace)
and $\aut(X,D)$ given by $u+\i v\mapsto\nabla u+J\nabla v$,
where $u,v\in C^\infty(X)$ are real-valued functions.
By the remark at the beginning of this paragraph then,
the proposition follows: indeed, for a complex-valued function
in $\Lambda_\mu(-\Delta_\o)$ it is immediate that both
its real and imaginary parts are contained in $\Lambda_\mu(-\Delta_\o)$
since $-\Delta_\o$ is a linear operator.

Thus suppose that $u\in C^\infty(X\setminus D)\cap C^0(X)$ is such that
$\nabla u\in\aut(X,D)$, so that $\nabla^{1,0}\nabla^{1,0} u=0$
(the gradients here and below are with respect to the edge metric $\o$
and the underlying complex structure).
Thus, using the Weitzenb\"ock formula \cite[\S6]{JMR} and the KEE assumption,
\begin{equation}
\label{WeitzCxEq}
\baeq
\Delta_\o|\nabla^{1,0} u |_\o^2
&= 2\Ric(\nabla^{1,0} u ,\nabla^{0,1} u ) + 2|\nabla^{1,0}\nabla^{1,0}u |^2
+ 2(\Delta_\o u)^2 + 4\o(\nabla^{1,0} u,\nabla^{0,1} \Delta_\o u)
\cr
&= 2\mu|\nabla^{1,0} u |_\o^2
+ 2(\Delta_\o u)^2 +4\o(\nabla^{1,0} u,\nabla^{0,1} \Delta_\o u)
.
\eaeq
\end{equation}
To conclude then, it would suffice to integrate \eqref{WeitzCxEq} and prove that
\beq
\label{StokesVerifEq}
\int_X
\Delta_\o|\nabla^{1,0} u |_\o^2\on=0,
\q \int_X
|\nabla^{1,0} u |_\o^2\on
=-
\int_X
u\Delta_\o u\on,
\eeq
and
\beq
\label{StokesVerifSecondEq}
-\int_X\o(\nabla^{1,0} u,\nabla^{0,1} \D_\o u)\on
=\int_X(\D_\o u)^2\on.
\eeq
Indeed, these identities then imply
\beq\label{ImplyEq}
\mu\int_X
|\nabla^{1,0} u |_\o^2\on
=
\int_X (\D_\o u) ^2.
\eeq
They also imply that
$$
\baeq
\int_X
|\nabla^{1,0} u |_\o^2\on
&=
-\int_X
u\Delta_\o u\on
\le
||u||_{L^2(X,\on)}
||\D_\o u||_{L^2(X,\on)}
\cr
&\le
\mu^{-1/2}||\D_\o u||_{L^2(X,\on)}
||\nabla^{1,0} u||_{L^2(X,\on)}
\eaeq
$$
where we used the fact that since $\o$ is KEE, the first positive eigenvalue of
$-\D_\o$ equals $\mu$ \cite[Lemma 6.1]{JMR}.
Therefore,
$\mu\int_X
|\nabla^{1,0} u |_\o^2\on
\le
\int_X (\D_\o u) ^2,
$
with equality if and only if $u$ is an eigenfunction of $-\D_\o$
with eigenvalue $\la_1=\mu$. Thus, by \eqref{ImplyEq}, $u$
is such an eigenfunction, concluding the proof of the proposition.

We now turn to proving \eqref{StokesVerifEq}--\eqref{StokesVerifSecondEq}.
First, we claim that
$u$ in fact has a polyhomogeneous expansion of the form
\beq\label{uPhgExpEq}
u \sim
a_0(y) + (a_{10}(y)\cos\th+a_{11}(y)\sin\th)r^{\frac{1}{\be}} + a_2(y) r^2 + O(r^{2+\eta}),
\eeq
for some $\eta > 0$. Here $y$ is a local coordinate on $D$
and $re^{\i \th}=z_1^\be$ with $D=\{z_1=0\}$ locally.
For the proof, observe that
since $\nabla^{1,0}\nabla^{1,0}u=0$,
$u$ lies in the kernel
of the self-adjoint fourth-order Lichnerowicz operator
$D_\o:=L_\o^\star\circ L_\o:C^\infty(X\setminus D)\cap C^0(X)
\ra C^\infty(X\setminus D)\cap C^0(X)$; here
$L_\o:u\mapsto \nabla^{1,0}\nabla^{1,0}u$ and $L_\o^\star$
is the formal $L^2$ adjoint computed with respect to $\o$.
Second, $D_\o$ is a linear degenerate elliptic operator of edge
type, in the sense of Mazzeo \cite{Mazzeo}, whose principal symbol is
$\Delta_\o^2$; more precisely \cite[(2.1)]{Calabi1985},
\beq\label{LichneEq}
D_\o=\Delta_\o^2+(\Rico,\i\ddbar(\,.\,))_\o
+(\del s_\o,\del u)_\o=\D_\o^2+\mu\D_\o,
\eeq
by the KEE assumption.
The expansion \eqref{uPhgExpEq}
then follows from the polyhomogeneous expansion for
the KEE metric $\o$
\cite[Theorem 1]{JMR}
and the polyhomogeneous structure of inverses of
elliptic edge operators associated to polyhomogeneous \K edge metrics
\cite[Theorem 6.1]{Mazzeo},\cite[Proposition 3.8]{JMR}, and the
fact that $u$ is bounded (if $u$ were not bounded then
its expansion would contain a $\log r$ term, but then the corresponding
vector field would not be bounded).

Finally, given  \eqref{uPhgExpEq}, the verification of
\eqref{StokesVerifEq}--\eqref{StokesVerifSecondEq} follows in the same way as in
the proof of \cite[Lemma 6.1]{JMR}.
\epf

\begin{remark}
{\rm
A shorter proof would be to avoid the Weitzenb\"ock formula
and and use \eqref{LichneEq} directly. It then follows that
$v=\D_\o u+\mu u$ is in the kernel of $\D_\o$. By the asymptotic
expansion for bounded solutions of $D_\o u=0$ we see that
$v$ is bounded (indeed, the term of order $O(r^{1/\be})$
in \eqref{uPhgExpEq} is in the kernel of $\D_\o$), and hence a constant.
Then it follows that by changing $u$ by a constant it must
be a eigenvalue of $\D_\o$ with eigenvalue $-\mu$. We preferred
the current proof since the Weitzenb\"ock formula
was used in  \cite{JMR} to obtain one direction
of the isomorphism proved here,
and it seemed natural to emphasize what is needed to make
that proof work in the other direction.
}
\end{remark}

\begin{lemma}
\label{boneLemma}
Let $(X,D,\o)$ be a \K edge manifold, and suppose that
$c_1(X)-\sum_i(1-\be_i)[D_i]=\mu[\o]$ with $\mu>0$.
Then $b_1(X)=0$.
\end{lemma}

\bpf
This is a direct corollary of the Kawamata--Viehweg Vanishing
Theorem which states that $H^i(X,\calO_X(K_X+N))=0$ for all $i>0$
whenever $N$ is numerically equivalent to a sum $B+\Delta$ of
a big and nef $\QQ$-divisor $B$, and a $\QQ$-divisor with snc
support $\Delta$ \cite[Vol. II, \S9.1.C]{Lazarsfeld}.
Thus, we may choose $\be\in\QQ^N\cap (0,1)^N$ such that
$B:=-K_X-\sum_i(1-\be_i)D_i$ is ample, and set $N=\sum_i(1-\be_i)D_i$.
Finally, by Hodge theory $b_1(X)=2h^{1,0}(X)=2\dim H^1(X,\calO_X)=0$
\cite[p. 105]{GH}.
\epf

\begin{remark}{\rm
In fact, it follows from \cite[Corollary 1]{Zhang}
that $X$ is simply connected in a much more general
setting. This generalizes the
classical result of Kobayashi in the Fano case
\cite{Kobayashi}.
}
\end{remark}

Let $\Aut_0(X)$ denote the connected Lie group
associated to $\aut(X)$. Similarly, denote by $\Aut_0(X,D)\subset\Aut_0(X)$
the Lie subgroup associated to $\aut(X,D)$. This is the identity component
of the automorphism group of the pair.
Putting the above results together we obtain a version of Matsushima's
Theorem \cite{Matsushima} for pairs.

\begin{proof}[Proof of Theorem \ref{reductiveThm}]
Suppose that
$c_1(X)-(1-\be)[D]=\mu[\o]$ with $\mu\in\RR$.
In case $\mu>0$ the statement is a corollary of Propostion
\ref{reductiveprop} since as noted in its proof every
Killing vector field of $(X,g)$ is Hamiltonian.

Suppose now that $\mu\le0$. Let $\psi\in\Aut_0(X,D)$.
Since $\psi\in\Aut(X)$, $\psi^\star c_1(X)=c_1(X)$. Since
$\psi$ fixes $D$, $\psi^\star[D]=[D]$. Thus, $\psi^\star[\o]=[\o]$.
Therefore, if $\o$ is KEE then $\psi^\star\o$ is a cohomologous
KEE form. But, when $\mu\le0$ the KEE form is unique in its
cohomology class \cite[Theorem 2]{JMR}. Thus $\psi$ is
the identity map, and $\Aut_0(X,D)=\{\id\}$.
\epf

\begin{remark}{\rm
Using the arguments above one can prove a corresponding
generalization to the edge setting of Calabi's theorem on the structure of
the automorphism group of an extremal metric \cite{Calabi1985}.
For brevity, we do not go into the details here.
}
\end{remark}

\section{Tian invariants of asymptotic pairs}
\label{section:alpha}

Throughout the article we use the standard language of the
singularities of pairs \cite{Ko97,CSD}. By strictly
log canonical (lc) singularities we mean log canonical singularities
that are not Kawamata log terminal \cite[Definition~3.5]{Ko97}.
We also distinguish
between an $\a$-invariant as
in Definition \ref{definition:global-threshold} below,
by which we refer to a global log
canonical threshold, and a Tian invariant, by which we
refer to the analogous invariant defined analytically in
terms of metrics \cite{T87}. These two invariants coincide under
certain regularity assumptions
\cite{CSD,Berm}. The algebraic definition
makes sense in more general (singular and/or degenerate) settings, while
the analytic definition is useful for proving
existence of KEE metrics by Theorem \ref{theorem:Mazzeo-Rubinstein-Jeffres}.

\subsection{A general bound on 
global
log canonical thresholds
of pairs
}
\label{lctSubsec}

Given a proper birational morphism $\pi:Y\ra X$, we define
the exceptional set of $\pi$ to be the smallest subset $\exc(\pi)\subset Y$,
such that $\pi:Y\sm\exc(\pi)\ra X\sm \pi(\exc(\pi))$ is an isomorphism.

A log resolution of $(X,\Delta)$ is a proper birational morphism
$\pi:Y\ra X$ such that $\pi^{-1}(\Delta)\cup\{\exc(\pi)\}$
is divisor with snc support.
Log resolutions exist for all the pairs we will consider in this article,
by Hironaka's theorem.

Assume that \mbox{$K_{X}+\D$} is a~$\mathbb{Q}$-Cartier
divisor.
Given a log resolution of $(X,\Delta)$, write
$$
\pi^\star (K_X+\Delta)=K_Y+\tilde \Delta +\sum e_i E_i,
$$
where
$\tilde \D$ denotes the proper transform of $\Delta$, and
where $\exc(\pi)=\cup E_i$, and $E_i$ are irreducible codimension one subvarieties.
Also, assume $\Delta=\sum \delta_i\Delta_i$, with $\Delta_i$ irreducible codimension one subvarieties,
so $\tilde \Delta=\sum \delta_i\tilde \Delta_i$.
Singularities of pairs can be measured as follows.

\begin{definition}
Let $Z\subset X$ be a subvariety.
A pair $(X,\Delta)$ has at most log canonical (lc)
singularities
along $Z$ if $e_i,\delta_j\le 1$ for every $i$ such that
$E_i\cap Z\not=\emptyset$ and every $j$ such that $\Delta_j\cap Z\not=\emptyset$.

\end{definition}

\begin{definition}
Let $Z\subset X$ be a subvariety.
The log canonical threshold of the pair $(X,\Delta)$ along $Z$ is
$$
\lct_Z(X,\Delta):=\sup\{\lambda\,:\, (X,\lambda \Delta) \h{\ is log-canonical along $Z$}\}.
$$
Set $\lct(X,\Delta):=\lct_X(X,\Delta)$.
\end{definition}

%%%%%%%%%%%%%%%%%%%%%%%%%%

Let $X$ be a~variety, let $B$ and $D$ be effective
Cartier $\mathbb{Q}$-divisors on the~variety $X$~such that
the~singularities of the~log pair $(X,B)$~are
log
terminal, and \mbox{$K_{X}+B+D$} is a~$\mathbb{Q}$-Cartier
divisor. Recall that the log canonical threshold of the~boundary
$D$  is the number
$$
\mathrm{lct}\big(X,B;D\big)=\mathrm{sup}\left\{\lambda\in\mathbb{Q}
\,:\, \mathrm{the~pair}\  \big(X, B+\lambda D\big)\ \mathrm{is\ log\ canonical}\right\}.%
$$
Let $H$ be an ample $\mathbb{Q}$-divisor on $X$, and let $[H]$ be
the class of the divisor $H$ in
$\mathrm{Pic}(X)\otimes\mathbb{Q}$.

\begin{definition}
\label{definition:global-threshold} The
global log canonical threshold
of the
log pair $(X,B)$ with respect to $[H]$ is the number
$$
\alpha(X,B,[H]):=
\mathrm{inf}\big\{ \mathrm{lct}\big(X,B;D\big)\,:\, D\
\h{is effective $\QQ$-divisor such that}\ D\sim_{\mathbb{Q}}\ H\big\}.%
$$
\end{definition}

For simplicity, we put $\alpha(X,[H])=\alpha(X,B,[H])$ if
there is no boundary, i.e.,
$B=0$. Similarly, we put $\alpha(X, B)=\alpha(X,B,[H])$
if $H\sim_\QQ-(K_{X}+B)$.

Finally, we put $\alpha(X)=\alpha(X,[H])$
if $B=0$ and $H=-K_{X}$. Note that it follows from
Definition~\ref{definition:global-threshold} that
\beq
\label{AlphaSecondDefEq}
\alpha\big(X,B,[H]\big)=\mathrm{sup}\left\{\ c \ \left| \
\aligned
&\ \h{for every $\QQ$-divisor}\ D\ \mathrm{such\ that}\ D\sim_{\mathbb{Q}} H\\
&\ \mathrm{the\ log\ pair}\ (X, B+cD)\ \mathrm{is\ log\ canonical}\\
\endaligned\right.\right\},
\eeq
and $\alpha(X,B,[\mu H])=\a(X,B,[H])/\mu$ for
every positive rational number $\mu$.

By a result of Demailly \cite[Appendix]{CSD}
(with complements by Berman \cite{Berm} in the log setting)
$\alpha(X,\sum(1-\be_i)D_i,[H])$ coincides with Tian's invariant
for the \K class $[H]$ \cite{T87}
when $X$ is smooth, $\sum D_i$ has simple normal crossings
and when the background measure has edge singularities
of angle $2\pi\be_i$ along $D_i$. In other words
\beq\label{SecondDefAlpha}
\alpha(X,\sum(1-\be_i)D_i,[H])
=
\sup\Big\{\; a \,:\,
\sup_{\vp\in\PSH(X,\o_0)}\int_X
e^{-a(\vp-\sup\vp)}\o^n<\infty
\Big\},
\eeq
where $\o$ is a \K edge metric with angle
$2\pi\be_i$ along $D_i$ and $\o_0$ is a smooth
\K metric with $[\o_0]=[\o]=[H]$.
In the notation of \cite[\S6.3]{JMR} $\mu=1$,
so in this normalization the criterion for existence of KEE
is precisely the one stated in Theorem \ref{theorem:Mazzeo-Rubinstein-Jeffres}.

The next lemma gives an explicit bound for $\a$-invariants on
curves. We will
make use of it in Proposition \ref{theorem:log-del-Pezzo-alpha-1}
to obtain explicit bounds for $\a$-invariants on log del Pezzo surfaces.
It also serves to illustrate the definitions above.

\begin{lemma}
\label{lemma:lct-curve} Let $C$ be a smooth curve,
$P_i\in C$ distinct points, and $a_i\ge0$.
Suppose that
$(C,\sum_{i=1}^{k}a_{i}P_{i})$ is
log terminal, i.e.,
$a_i<1$ for all $i$. Let $H$ be an ample $\mathbb{R}$-divisor on
$C$, and let $d\in\mathbb{R}_{>0}$ be its degree. Then
$$
\alpha\Big(C,\sum_{i=1}^{k}a_{i}P_{i},[H]\Big)\ge\frac{1-\max\{a_1,\ldots,a_{k}\}}{d}.
$$
Furthermore, equality holds when $C=\mathbb{P}^1$.
\end{lemma}

\begin{proof}
If $D\sim_\QQ H$ then $D=\sum b_iQ_i$ with $b_i\ge0$ and $\sum b_i=d$.
Then $(C,\sum_{i=1}^{n}a_{i}P_{i}+\lambda D)$
is log canonical precisely when $a_i+\lambda b_i\le 1$,
i.e., $\lambda\le (1-a_i)/b_i$ for all such admissible $b_i$
(here we are assuming that $P_i=Q_i$ otherwise the bounds
are even weaker).
In particular, if $\lambda\le (1-\max_i\, a_i)/d$ then the pair
is always log canonical. This proves the inequality.
The result follows
since we may choose a divisor $D=dP_j$
with $j$ such that $\max_i\,a_i=a_j$;
on $\PP^1$
$\QQ$-rational equivalence is determined solely by degree
so $D\sim_\QQ H$, thus in this case
$\alpha\Big(C,\sum_{i=1}^{k}a_{i}P_{i},[H]\Big)\le
(1-\max_i\, a_i)/d$.
\end{proof}

The following gives a general bound on global log canonical
thresholds of pairs.

\begin{prop}
\label{theorem:handy-proposition} Suppose that $B=(1-\beta)S$,
where $\beta\in(0,1)$ and $S$ is an irreducible nef Cartier
divisor on $X$.
Let $H$ be an ample $\mathbb{Q}$-divisor on $X$.
Put
$$
\gamma=\mathrm{sup}\big\{c\in\mathbb{Q}\ \big\vert\ H-cS\ \mathrm{is\ pseudoeffective}\big\}.%
$$
Then $\alpha(X,(1-\beta)S,[H])\ge\mathrm{min}(\beta/\gamma,
\alpha(X,[H]), \alpha(S,[H]\vert_{S}))$.
\end{prop}

\begin{proof}
Put $\lambda=\mathrm{min}(\beta/\gamma, \alpha(X,[H]),
\alpha(S,[H]\vert_{S}))$. We may assume that $\lambda>0$. Suppose
that $\mathrm{lct}(X,(1-\beta)S,[H])<\lambda$. Then there exists
an effective $\mathbb{Q}$-divisor $\Delta$ on $X$ such that
$\Delta\sim_{\mathbb{Q}} H$ and the log pair $(X, (1-\beta)S+\mu
\Delta)$ is not log canonical at some point $P\in X$ for some
positive rational number $\mu<\lambda$.

If $P\not\in S$, then  the log pair $(X, \mu \Delta)$ is not log
canonical at the point $P\in X$, contradicting
$\mu<\lambda\le\alpha(X,[H])$ and $\Delta\sim_{\mathbb{Q}}
H$. Thus,  $P\in S$.

Put $(1-\beta)S+\mu \Delta=aS+R$ for some positive rational number
$a\ge 1-\beta$ and some effective $\mathbb{Q}$-divisor $R$
such that $S\not\subset\mathrm{Supp}(R)$. Since
$$
H\sim_{\mathbb{Q}}\Delta=\frac{a-1+\beta}{\mu}S+\frac{1}{\mu}R,
$$
we see that $(a-1+\beta)/\mu\le\gamma$. Because
$\mu<\lambda\le\beta/\gamma$ then $a\le 1$. Since $a\le 1$, the
log pair $(X, S+R)$ is not log canonical at the point $P\in S$.
Thus, it follows from adjunction theorem
\cite[Theorem~5.50]{KoMo98} that the log pair $(S, R\vert_{S})$ is
not log canonical as well.

Note that $R\sim_{\mathbb{Q}} \mu H-(a-1+\beta)S$. Thus, if
$S\vert_{S}$ is $\QQ$-linearly equivalent to some effective
divisor $T_{S}$ on $S$, then
$$
R\vert_{S}+(a-1+\beta)T_{S}\sim_{\mathbb{Q}} \mu H
$$
while $(S, R\vert_{S}+(a-1+\beta)T_{S})$ is not log canonical,
which contradicts $\mu<\alpha(S,[H]\vert_{S}))$. Unfortunately, we
do not know that $S\vert_{S}$ is $\QQ$-linearly equivalent to some
effective divisor on $S$, because we only know that $S\vert_{S}$
is nef. Nevertheless, we still can obtain a contradiction in a
similar way by adding to $S\vert_{S}$ a small piece of an ample
divisor $H\vert_{S}$. Note that
$$
R\sim_{\mathbb{Q}} \mu H-(a-1+\beta)S=\lambda H-\Big((\lambda-\mu)H+(a-1+\beta)S\Big),%
$$
where $(\lambda-\mu)H+(a-1+\beta)S$ is an ample
$\mathbb{Q}$-divisor, since $S$ is nef and $H$ is ample. Thus,
there exists an effective $\mathbb{Q}$-divisor $G$ on the variety
$X$ such that $G\sim_{\mathbb{Q}} (\lambda-\mu)H+(a-1+\beta)S$ and
$S\not\subset\mathrm{Supp}(G)$. Then
$(R+G)\vert_{S}\sim_{\mathbb{Q}} [\lambda H]\vert_{S}$ and the log
pair $(S, (R+G)\vert_{S})$ is not log canonical, since $(S,
R\vert_{S})$ is not log canonical and $G\vert_{S}$ is an effective
$\mathbb{Q}$-divisor on $S$. On the other hand, the log pair $(S,
(R+G)\vert_{S})$ must be log canonical, because
$\lambda\le\alpha(S,[H]\vert_{S}))$ and
$\lambda^{-1}(R+G)\vert_{S}\sim_{\mathbb{Q}} [H]\vert_{S}$.
\end{proof}

The previous result specializes to a result of Berman \cite{Berm}
when $X$ is further assumed to be smooth, when the boundary $S$
is assumed to be smooth and ample, and further
when $S$ and $H$ are proportional in the sense that $S\sim_\QQ c H$,
(i.e., in his setting $S$ is a section of $H$ and $\gamma=c$).
Upon completion of this article we learned that Odaka--Sun
also gave an algebraic proof of Berman's result
in the special case $[H]=[S]=-K_X$ \cite[Corollary 5.5]{OdakaS}. 
We decided
to keep Proposition \ref{theorem:handy-proposition} due
to its general form and possible application to polarizations
different from $-K_X$.
Thus, we obtain as a corollary the following result originally proved
by Berman using analytic methods.

\begin{corollary}
\label{corollary:Berman-inequality} Suppose that $X$ is smooth,
$B=(1-\beta)S$, where $\beta\in(0,1]$ and $S$ is an
irreducible smooth ample Cartier divisor on $X$. Then
$$
\alpha\big(X,(1-\beta)S,[\beta S]\big)\ge
\mathrm{min}\bigg\{1, \frac{\alpha(X,[S])}{\beta}, \frac{\alpha(S,[S]\vert_{S})}{\beta}\bigg\}.%
$$
\end{corollary}

\subsection{Limiting behavior of $\alpha$-invariants}
\label{LimitAlphaSubsec}

In this subsection we prove Theorem \ref{theorem:log-del-Pezzo-alpha}.
The proof is divided into Propositions
\ref{theorem:log-del-Pezzo-alpha-2},
\ref{theorem:log-del-Pezzo-alpha-3},
and \ref{theorem:log-del-Pezzo-alpha-1}.

More generally, we conjectured in \eqref{AlphaConjEq} that in higher dimensions
$$
\lim_{\beta\to
0^+}\alpha(X,(1-\beta)D)=
\left\{\aligned%
&1\qquad \mathrm{if}\ K_{X}+D\sim 0,\\
&\mathrm{min}\{1,\alpha(X,[-K_{X}-D]),\alpha(D)\}
\qquad \!\mathrm{if}\ 0\not\sim-K_{X}-D\ \mathrm{is\ not\ big},\\
&0\qquad  \mathrm{if}\ -K_{X}-D\ \mathrm{is\ big}.\\
\endaligned
\right.
$$
provided that $D$ is irreducible and smooth. Note that the
divisors $-K_{X}-D$ and $-K_{S}=(-K_{X}-D)\vert_{D}$ may not be
ample. This violates our definitions of $\alpha(X,[-K_{X}-D])$ and
$\alpha(D)=\alpha(D,[(-K_{X}-D)\vert_{D}])$. However, it follows
from \cite[Theorem 3.3]{KoMo98}  that $-K_{X}-D$ and $-K_{D}$ are semiample, so we
can define $\alpha(X,[-K_{X}-D])$ and $\alpha(D)$ in the same way
as in the case when $-K_{X}-D$ and $-K_{D}$ are ample
(and it could happen that $\a=\infty$, for instance). We prove
this conjecture in the cases when $K_{X}+D\sim 0$ or $-K_{X}-D$
is big. While Proposition \ref{theorem:log-del-Pezzo-alpha-3} is
two-dimensional, we believe its proof should find a suitable
generalization to higher dimensions.

\begin{remark}
{\rm
The situation in the simple normal crossings case is more complicated
and there is in general no unique limit for Tian's invariant as $|\be|$
tends to zero. To illustrate this we consider the following
toric example.
Let $L_{1}$, $L_{2}$, $L_{3}$ be
distinct lines on $\mathbb{P}^2$. Then
$$
\alpha(\mathbb{P}^2,\sum_{i=1}^3(1-\beta_{i})L_{i})=\frac{\mathrm{max}(\beta_1,\beta_2,\beta_3)}{\beta_1+\beta_2+\beta_{3}}
$$
for any $(\beta_1,\beta_2,\beta_3)\in(0,1]^3$. Furthermore, let
$G$ be a finite group in $\mathrm{Aut}(\mathbb{P}^2)$ such that
$L_{1}+L_2+L_3$ is $G$-invariant and $G$ does not fix any point in
$\mathbb{P}^2$ (there are infinitely many such groups). Then
$\alpha_{G}(\mathbb{P}^2,(1-\beta)\sum_{i=1}^3L_{i})=1$ for any
$\beta\in(0,1]$. The proof, modelled
on the arguments in \cite[Lemma~5.1.]{CSD},
is left to the reader.
}
\end{remark}

\subsubsection{Class $(\gimel)$ and $(\daleth)$}

The next result holds for asymptotically log Fano varieties in any dimension
and for $C$ with any $r\ge1$.
\begin{prop}
\label{theorem:log-del-Pezzo-alpha-2} Suppose that
$-K_{S}-C$ is big. Then $\lim_{\beta\to 0^+}\alpha(S,(1-\beta)C)=0$.
\end{prop}

\begin{proof}
Since $-K_{S}-C$ is big, there exists positive integer $N$ such
that
$$
-N(K_{S}+C)\sim C+\Delta
$$
for some effective divisor $\Delta$
(this follows from the characterization of bigness
\cite[Corollary 2.2.7]{Lazarsfeld}
since given an ample class $H$ and an effective class $C$
one may find an integer $M$ such that $MH\sim C+E$
for some effective divisor $E$).
Put $\epsilon=\frac{1}{N}$.
Then $-K_{S}-C\sim_{\mathbb{Q}}\epsilon C+\epsilon\Delta$. Take
any sufficiently small real $\beta>0$ such that
$-(K_{S}+(1-\beta)C)$ is ample. Put $D=(\beta+\epsilon)
C+\epsilon\Delta$. Then
$$
D\sim_{\mathbb{R}} (\beta+\epsilon) C+\epsilon\Delta\sim_{\mathbb{R}} -(K_{S}+(1-\beta) C).%
$$
On the other side, we have $\mathrm{lct}(S,(1-\beta)C;D)\le
\frac{\beta}{\epsilon+\beta}$, which implies that
$\alpha(S,(1-\beta)C)\le\frac{\beta}{\epsilon+\beta}$. This
shows that
$$
\lim_{\beta\to 0^+}\alpha(S,(1-\beta)C)=0,
$$
since $\epsilon$ depends only on the pair $(S,C)$ and not on $\beta$.
\end{proof}

\subsubsection{Class $(\beth)$}

In this subsection we prove Theorem \ref{theorem:log-del-Pezzo-alpha}
in the case when $(S,C)$ is of class $(\beth)$.

Before embarking on the proof let us say few words about
the idea of the proof. If $-K_{S}-C$ is not big, then
by Proposition \ref{proposition:conic-bundle}
$C\cong\mathbb{P}^1$ and $|-K_{S}-C|$ is free from base points and
gives a morphism $S\to\mathbb{P}^1$ whose general fiber is
$\mathbb{P}^1$ (a conic bundle). Moreover, the general curve in
$|-K_{S}-C|$ is a fiber of this conic bundle. On the other hand,
when $\beta$ is  small, the class
$$
-K_{S}-(1-\beta)C\sim_{\mathbb{R}} -(K_{S}+C)+\beta C
$$
is close to $-K_{S}-C$. Thus, when looking for divisors
$$
\Delta\sim_{\mathbb{R}}-(K_{S}+(1-\beta)C)
$$
having small $\mathrm{lct}(S,(1-\beta)C;\Delta)$ with $0<\beta\ll
1$, there are not many options. Namely, we can take $\Delta$ to
be $\beta C+F$ where $F$ is a fiber of the conic bundle. All other
choices of $\Delta$ gives us either better or similar
singularities. The reason is a \emph{continuity} of
$\alpha(S,(1-\beta)C)$ in $\beta$. When $\beta$ is very small, we
have
$$
\alpha(S,(1-\beta)C)\approx\alpha(S,C).
$$
Note that $\alpha(S,C)$ is not well defined according to our
definition of the $\alpha$-invariant, because $-K_{S}-C$ is not
ample. Nevertheless, we can  still define $\alpha(S,C)$ in a
similar way, since $-(K_{S}+C)$ is semi-ample. On the other hand,
if $\beta=0$, we have no freedom in choosing $\Delta$ at all!
Indeed, if $\beta=0$, then
$$
\Delta\sim_{\mathbb{R}}-K_{S}-C,
$$
which implies that every irreducible component of $\Delta$ must be
a fiber of the conic bundle $S\to\mathbb{P}^1$. In this case, the
worst $\Delta$ (the one with smallest $\mathrm{lct}(S,C;\Delta)$)
must be a fiber of the conic bundle. Furthermore, among these
fibers there are exactly two that are worse than others, i.e., the
two fibers that do not intersect $C$ transversally. So in a sense
we have a choice choice of exactly two divisors for $\Delta$,
which both gives us $\mathrm{lct}(S,C;\Delta)\approx\frac{1}{2}$.

\begin{prop}
\label{theorem:log-del-Pezzo-alpha-3} Suppose that $-K_{S}-C$ is
not big. Then $\lim_{\beta\to
0^+}\alpha(S,(1-\beta)C)=\frac{1}{2}$.
\end{prop}

\begin{proof}
By Lemma~\ref{lemma:log-del-Pezzo-rational-or-elliptic-1}, we have
$C\cong\mathbb{P}^1$. By
Proposition~\ref{proposition:conic-bundle} the linear system
$|-K_{S}-\sum_{i=1}^{r}C_{i}|$ is free from base points and
gives a morphism $\xi\colon S\to\mathbb{P}^1$ such that its
general fiber is $\mathbb{P}^1$, and every reducible fiber
consists of exactly two components.

Let $F$ be a general fiber of $\xi$. Then $-K_{S}-C\sim F$,
since $|-K_{S}-\sum_{i=1}^{r}C_{i}|$ is a pencil by
Lemmas~\ref{lemma:conic-bundles-1} and
\ref{lemma:conic-bundles-2}. Then $F.C=2$, since
$-K_{S}-C^2=0$.

The morphism $\xi$ induces a double cover $C\to\mathbb{P}^1$.
Since $C$ is a smooth rational curve, this double cover has
exactly two ramification points. Let $O$ be one of these two
ramification points, and let $F_O$ be a fiber of $\xi$ that passes
through it. Recall that
$$
-(K_{S}+(1-\beta)C)\sim_{\mathbb{R}} F_O+\beta C
$$
by construction. On the other hand, we have
$$
\mathrm{lct}(S,(1-\beta)C;F_{O}+\beta C)=\left\{\aligned%
&\frac{1+\beta}{2+\beta}\ \mathrm{if}\ F_O\ \mathrm{is\ singular},\\
&\frac{1+2\beta}{2+2\beta}\ \mathrm{if}\ F_O\ \mathrm{is\ smooth}.\\
\endaligned
\right.
$$
To see this it suffices to blow-up once when $F_O$ is singular
and twice when it is smooth.
Hence, $\alpha(S,(1-\beta)C)\le
(1+\beta)/(2+\beta)$. To complete the proof it is thus enough to
show that for every positive real $\epsilon>0$ there exists real
$\delta=\delta(\eps,C)>0$ such that both \eqref{mainassumption} and
\beq
\label{AlphaEpsEq}
\alpha(S,(1-\beta)C)\ge\frac{1}{2}-\epsilon
\eeq
for every real $\beta\in(0,\delta)$. In fact, we claim
that $\delta=\min\{1/2,\epsilon/|C^2|,\be_\max\}$ will do,
where \eqref{mainassumption} holds for $\be\in(0,\be_\max)$.

To that end we work with the definition \eqref{SecondDefAlpha}
of the global log canonical threshold of the pair $(S, (1-\beta)C)$.
We use repeatedly the following application of adjunction:
if $K\subset S$ is a smooth irreducible curve
and $M$ an effective $\RR$-divisor
on $S$ and if $(S,K+ M)$ is not lc at a point $Q$ on $K$
then $(K,M|_K)$ is not lc at $Q$, or equivalently
$\mult_Q K.M>1$ \cite[Excercise~6.31]{CoKoSm}.

Throughout the proof we let $D$
be an effective $\mathbb{R}$-divisor satisfying
$$
D\sim_{\mathbb{R}} F+\beta C
$$
If the pair $(S, (1-\beta)C+\lambda D)$ is not lc
at some point $P\in C$ and $C\not\subset\mathrm{Supp}(D)$
then
$$
2+\beta C^2=C. \big(F+\beta
C\big)=C. D\ge\mathrm{mult}_{P}\big(C.
D\big)>{\lambda}^{-1},
$$
thus $\lambda>\frac1{2+\be C^2}$.
If $(S, (1-\beta)C+\lambda D)$ is not lc
at some point $P\in C$ and $C\subset\mathrm{Supp}(D)$
then write $D=\mu C+\Omega$,
where $\mu$ is a positive rational number, and $\Omega$
is an effective $\mathbb{R}$-divisor on the surface $S$ whose
support does not contain the curve $C$. Then
$$
2\beta=(F+\beta C). F=D. F=(\mu C+\Omega). F=2\mu+\Omega. F\ge 2\mu,%
$$
so $\mu\le\beta$. On the other hand,
$(S, (1-\beta+\lambda\mu)C+\lambda\Omega)$ is not lc at
$P$. Since $1-\beta+\lambda\mu\le 1$, also $(S,
C+\lambda\Omega)$ is not lc at $P$.
Thus,
$$
2+(\beta-\mu) C^2=C.
\big(F+(\beta-\mu) C\big)=C.
\Omega\ge\mathrm{mult}_{P}\big(C.
\Omega\big)>{\lambda}^{-1},
$$
so again $\lambda>\frac 1{2+\be C^2}$.

Next, suppose that
$(S, (1-\beta)C+\lambda D)$ is not lc
at some point $P\not\in C$. Then $(S, \lambda D)$ is not
lc at $P$. Let $F_P$ be the fiber of $\xi$ that passes
through $P$. Then we must consider three cases: $F_P$ is
smooth, $F_P$ is singular and $P\ne\mathrm{Sing}(F_P)$,  $F_P$ is
singular and $P=\mathrm{Sing}(F_P)$.

First, suppose $F_P$ is smooth and put $D=\tau F_P+\Delta$,
where $0<\tau\in\QQ$, and $\Delta$ is an
effective $\mathbb{R}$-divisor with $F_P\not\subset\mathrm{Supp}(\D)$.
Then
$$
4\beta+\beta^2C^2=D^2=(F_P+\beta C). D=(\tau F_P+\Delta). D=2\beta\tau+\Delta. D\ge 2\beta\tau,%
$$
so $\tau\le 2+\frac{\beta}{2}C^2$.
If $\lambda\tau>1$ then $\lambda>\frac 1{2+\be C^2/2}$.
Suppose that $\lambda\tau\le 1$.
Thus, the pair $(S, F_P+\lambda
\Delta)$ is not lc at $P$.
Then
$$
2\beta=
F_P.\big(F_P+\beta C-\tau F_P\big)=F_P.\Delta\ge\mathrm{mult}_{P}\big(F_P. \Delta\big)>{\lambda}^{-1},
$$
so $\lambda>\frac1{2\be}$.

Next, suppose $F_P$ is singular. Then $F_P=F_1+F_2$, where
$F_1$ and $F_2$ are smooth rational curves on $S$ such that
$
F_1. F_2=F_1. C=1=F_2. C=1,
$
and $F_1^2=F_2^2=-1$. Put $D=\tau_1 F_1+\tau_2 F_2+\Theta$, where
$0<\tau_1,\tau_2\in\QQ$, and $\Theta$
is an effective $\mathbb{R}$-divisor
with
$F_1,F_2\not\subset\mathrm{Supp}(\Theta)$.
Then
\beq\label{betatauEq}
\beta=(F+\beta C). F_1=(\tau_1 F_1+\tau_2 F_2+\Theta). F_1=-\tau_1+\tau_2+\Theta. F_1\ge -\tau_1+\tau_2,%
\eeq
and similarly $\tau_1-\tau_2\le\beta$.
On the other hand, using that $D$ is ample we have
$$
4\beta+\beta^2C^2=(F+\beta C). D=(\tau_1 F_1+\tau_2 F_2+\Theta). D=\beta(\tau_1+\tau_2)+\Theta. D\ge \beta(\tau_1+\tau_2),%
$$
so $\tau_1+\tau_2\le 4+\beta
C^2$, and combined with \eqref{betatauEq}
then
$\tau_2\le 2+\frac{\beta}{2} C^2+\frac{\beta}{2}$
and similarly for $\tau_1$.
If $\lambda\tau_i>1$ for some $i$ then $\lambda>\frac 1{2+\be C^2/2+\be/2}$.
Suppose that $\lambda\tau_1,\lambda\tau_2\le 1$.
There are two cases
to consider: $P=F_1\cap F_2$ and $P\ne F_1\cap
F_2$.
Suppose first that $P=F_1\cap F_2$. Then the pairs $(S,
F_1+\lambda\tau_2F_2+\lambda \Delta)$ and  $(S, \lambda\tau_1
F_1+F_2+\lambda\Theta)$ are not lc at $P$.
Then
$$
\beta+\tau_1=F_1.(F+\beta C-\tau_1 F_1)=F_1.(\tau_2F_2+\Theta)>\tau_2+{\lambda}^{-1} 
$$
so
$\la> \frac1{\beta+\tau_1-\tau_2}$
and similarly
$\la> \frac1{\beta+\tau_2-\tau_1}$,
so using \eqref{betatauEq} $\la>\frac1{2\be}$.
Next, suppose $P\ne F_1\cap F_2$, say $P\not\in F_1$, $P\in F_2$.
Hence, the log
pair $(S, \lambda\tau_2 F_2+\lambda\Theta)$ is not lc
at $P$. Then
$$
\beta+\tau_2-\tau_1=F_2.\big(F+\beta C-\tau_1 F_1-\tau_2F_2\big)=F_2.\Theta>{\lambda}^{-1},
$$
so
$\la> \frac1{\beta+\tau_2-\tau_1}$,
so again using \eqref{betatauEq} $\la>\frac1{2\be}$

In conclusion, we see that if
$(S, (1-\beta)C+\lambda D)$ is not lc then $\la>\frac12-\eps$
whenever $\be<\min\{1/2,\epsilon/|C^2|,\be_\max\}$.
Thus, \eqref{AlphaEpsEq} follows from \eqref{SecondDefAlpha},
concluding the proof.
\end{proof}

\subsubsection{Class $(\aleph)$}

By Corollary \ref{corollary:Berman-inequality},
$
\alpha(X,(1-\beta)S)
\ge
\mathrm{min}
\{
1,\be^{-1}\alpha(X),\be^{-1}\alpha(S,[S]\vert_{S})
\}%
$ when $(S,C)$ is of class $(\aleph)$. Moreover, from the
definition and the fact that $D\sim-K_X$ this invariant is bounded
above by 1. Combining this, Lemma~\ref{lemma:lct-curve}, and
Theorem \ref{theorem:Mazzeo-Rubinstein-Jeffres} yields:

\begin{prop}
\label{theorem:log-del-Pezzo-alpha-1} Let $(X,D)$ be an
asymptotically log Fano pair with $D\in|-K_X|$ a smooth
irreducible divisor and $X$ Fano. Then $\lim_{\beta\to
0^+}\alpha(S,(1-\beta)C)=1$. Moreover, \hfill\break (i) if $\dim
X=2$, then
$\alpha(X,(1-\beta)D)\in[\mathrm{min}\big\{1,\frac{1}{9\beta}\big\},1],
$ and $(X,D)$ admits KEE metrics for all $\beta\in(0,1/6)$.
\hfill\break (ii) if $\dim X=3$, then $ \alpha(X,(1-\beta)D)
\in[\mathrm{min}\big\{1,\frac{1}{64\beta}\big\},1], $ and $(X,D)$
admits KEE metrics for all $\beta\in(0,1/48)$.
\end{prop}

\begin{proof}
As noted above, it suffices to estimate
$
\mathrm{min}\{\alpha(X), \alpha(D,[D]\vert_{D})\}.%
$

(i) 
First, by using Lemma~\ref{lemma:lct-curve}
and the fact that $K_{X}^2\le 9$ for every 
smooth del Pezzo surface (by their classification)
one has ${\alpha(D,[D]\vert_{D})}\ge{1}/{9}$.

It remains to show that ${\alpha(X)}{}\ge{1}/{9}$. 
This follows from the complete list of lcts
of del Pezzo surfaces \cite[Theorem~1.7]{Ch07b} but
we now explain a direct derivation that 
can also be adapted to prove (ii).
Let $\Delta$
be an effective
$\mathbb{Q}$-divisor on $X$ such that
$\D\sim_{\mathbb{Q}}-K_{X}$ and the log pair $(X,\lambda\Delta)$ is
not log canonical at some point $P\in X$ for some positive
rational $\lambda$.

If $-K_{X}$ is very ample, let $H_P$ be a general curve in
$|-K_{X}|$ that passes through $P$. By the very ampleness 
we may assume that $H_P$
is not contained in the support of the divisor $D$.
Thus, as in the proof of Proposition \ref{theorem:log-del-Pezzo-alpha-3}, 
$$
K_{X}^2=\D.
H_P\ge\mathrm{mult}_P(\D)\mathrm{mult}_P(H_P)\ge
\mathrm{mult}_P(\D)>{\lambda}^{-1},
$$
so $\la>1/9$.

If $-K_{X}$ is not very ample but still base-point free
$|-K_{X}|$ gives a surjective finite morphism $X\to
V$, where $V$ is a surface, which implies that we can
still proceed as in the very ample case.
If $-K_X$ is not base-point free by the classification
of del Pezzo surfaces $K_X^2=1$ and the linear system $|-2K_{X}|$
is base-point free and gives a surjective finite morphism $X\to
V^\prime$, where $V^\prime$ is a surface. Let
$H_P^\prime$ denote a general curve in $|-2K_{X}|$ that passes
though $P$. Then
$$
2K_{X}^2=\D.
H_P\ge\mathrm{mult}_P(\D)\mathrm{mult}_P(H_P)\ge
\mathrm{mult}_P(\D)>{\lambda}^{-1},
$$
so $\la>1/2K_X^2=1/2$. The result now follows from \eqref{AlphaSecondDefEq}.

(ii)
Let $\D$ be as in (i).
Suppose first that $|-K_{X}|$ is base-point free. We claim that
$\mathrm{mult}_{P}(\Delta)\le64$ for every point $P\in X$ and for
every divisor $\Delta$ on $X$ such that
$\Delta\sim_{\mathbb{Q}}-K_{X}$. Indeed, since $|-K_{X}|$ is 
base-point free, the linear system $|-K_{X}|$ gives a finite surjective
morphism $X\to U$, where $U$ is a threefold. Thus,
there exists $S_P$ and $S_{P}^\prime$ in $|-K_{X}|$ such that
$P\in\mathrm{Supp}(S_{P}. S_P^\prime)$ and no component of
the $1$-cycle $S_{P}. S_P^\prime$ is contained in the
support of $D$. Then
$$
-K_{X}^3=\D. S_P.
S_{P}^\prime\ge\mathrm{mult}_P(\D)\mathrm{mult}_P(S_P)\mathrm{mult}_P(S_P^\prime)\ge
\mathrm{mult}_P(\D)>{\lambda}^{-1}.
$$
Thus, $\lambda>-1/K_{X}^3=1/64$ \cite[Corollary~7.1.2]{IsPr99},
implying $\alpha(X)\ge1/64$.
Similarly, we can prove that $\mathrm{mult}_{P}(\Omega)\le64$ for
every point $P\in D$ and for every divisor $\Omega$ on $D$ such
that $\Omega\sim_{\mathbb{Q}}-K_{X}\vert_{D}$. 
Thus, $\alpha(D,[D]\vert_{D})\ge{1}/{64}$.

Next, suppose that
$|-K_{X}|$ has base-points. This is a very special situation.
Indeed, it follows from \cite[Theorem~2.4.5]{IsPr99} that either
$-K_{X}^3=4$ and $X$ is a blow up of a smooth hypersurface in
$\mathbb{P}(1,1,1,2,3)$ of degree $6$ along a smooth elliptic
curve that is a complete intersection of two surfaces in
$|-\frac{1}{2}K_{X}|$, or $-K_{X}^3=6$ and
$X\cong\mathbb{P}^1\times S_1$, where $S_1$ is a smooth del Pezzo
surface with $K_{S_1}^2=1$. In both cases $|-2K_{X}|$ is 
base-point free. Thus, the same arguments as in the 
base-point free case show that $\mathrm{mult}_{P}(\Delta)<-4K_{X}^3$
for every point $P\in X$ and for every divisor $\Delta$ on $X$
such that $\Delta\sim_{\mathbb{Q}}-K_{X}$, and that
$\mathrm{mult}_{P}(\Omega)<-4K_{X}^3$ for every point $P\in D$ and
for every divisor $\Omega$ on $D$ such that
$\Omega\sim_{\mathbb{Q}}-K_{X}\vert_{D}$. Keeping in mind that
$-K_{X}^3\le 6$, we see that
$\alpha(X)\ge{1}/{24}$ and
$\alpha(D,[D]\vert_{D})\ge{1}/{24}$.
\end{proof}

Note that the lower bounds in
Proposition~\ref{theorem:log-del-Pezzo-alpha-1} can be improved
by a case-by-case analysis using results from \cite{Ch07b,CSD}. 
When $\dim X=2$ it is also possible to say
more about the existence of KEE metrics. In fact, in the cases
(I.1A) and (I.5A.m) with $m\ge3$ a KEE metric exists for all
$\be\in(0,1]$ since it exists for $\be=1$
\cite{Berm,LiSun,JMR,Ti90}. In the remaining two cases (I.5A.m),
$m\in\{1,2\}$, it is possible to compute $\a(S,(1-\be)C)$ to find
all $\beta\in(0,1]$ such that $\a(S,(1-\be)C)>\frac{2}{3}$.
Moreover, in the latter cases the value of the $\a$-invariant
depends on the choice of the anticanonical boundary curve itself.

\begin{remark}{\rm
Let $X$ be a smooth Fano variety of dimension $n$, and let $D$ be
a smooth divisor in $|-K_{X}|$. Put
$M=3^n(2^n-1)^n(n+1)^{n(n+2)(2^n-1)}$ and $N=2(n+1)(n+2)!$. Then
%$\alpha(X)\ge\frac{1}{N^{n-1}M}$ and
$$
\alpha(X,(1-\beta)D)\ge\min\{1,{\beta}^{-1}{N^{n-1}M}\},
$$
for every $\be\in(0,1]$.
Indeed, $(-K_{X})^n\le
3^n(2^n-1)^n(n+1)^{n(n+2)(2^n-1)}$ 
(see, e.g., \cite[Theorem 5.18]{Debarre}).
On the other hand, $|-NK_{X}|$ is base-point free by
\cite[Theorem 1]{Kollar1993}.
Note that
$-12n^nK_{X}$ is very ample by \cite[Corollary 12.11]{Demailly1993}.
Thus we can proceed as in
the proof of Proposition~\ref{theorem:log-del-Pezzo-alpha-1}.
}
\end{remark}

\section{Existence and non-existence of KEE metrics}

Our goal in this section is to make several first steps
towards the uniformization of asymptotically log del Pezzo surfaces
as stated in Conjecture \ref{UniformizationConj}.

\subsection{Automorphism groups}
\label{AutomorphismSec}

Theorem \ref{NoKEEThm} is a direct consequence of
Theorem \ref{reductiveThm} and the following result.

\begin{prop}
\label{NoKEEProp} The automorphisms groups of the following pairs
of class $(\gimel)$ or $(\daleth)$ are not reductive:
$\mathrm{(I.1C})$, $\mathrm{(I.2.n})$ with any $n\ge 0$,
$\mathrm{(I.6C.m})$ with any $m\ge 1$, $\mathrm{(I.7.n.m})$ with
any $n\ge 0$ and $m\ge 1$, $\mathrm{(I.6B.1})$,
$\mathrm{(I.8B.1})$ and $\mathrm{(I.9C.1})$. On the other hand,
$\Aut(S,C)$ is reductive when $(S,C)$ is one of the following:
$\mathrm{(I.1A})$, $\mathrm{(I.4A})$, $\mathrm{(I.3B})$,
$\mathrm{(I.4C})$, $\mathrm{(I.5.m})$ with $m\ge1$, $\mathrm{(I.1B})$,
$\mathrm{(I.6B.m)}$, $\mathrm{(I.8B.m)}$, or $\mathrm{(I.9C.m)}$ with $m\ge2$.
\end{prop}

\bpf If $(S,C)$ is $\mathrm{(I.1A})$, $\mathrm{(I.4A})$, or
$\mathrm{(I.5.m})$, then $\Aut(S,C)$ is finite, since $C$ is a
$\Aut(S,C)$-invariant elliptic curve that is an anticanonical
ample divisor. If $(S,C)$ is $\mathrm{(I.1B})$ then
$\Aut(S,C)\cong\mathrm{PGL}_{2}(\mathbb{C})$. If $(S,C)$ is
$\mathrm{(I.3B})$ then
$\Aut(S,C)\cong\mathrm{GL}_{2}(\mathbb{C})$. If $(S,C)$ is 
$\mathrm{(I.4C})$ then
$\Aut_0(S,C)\cong\mathrm{PGL}_{2}(\mathbb{C})$.

For the case $\mathrm{(I.1C})$, or, in fact, in any dimension, the
pair $(\PP^n, H)$ with $H$ a hyperplane in $\mathbb{P}^n$,
satisfies
$$
\Aut(\PP^n, H)\cong\Aut(\PP^n,
p)\cong\Aut(\h{Bl}_p\PP^n)\cong\mathbb{G}_a^n\rtimes\mathrm{GL}_{n}(\mathbb{C}),
$$
for a point $p\in\mathbb{P}^n$, where $\h{Bl}_p\PP^n$ denotes the
blow-up of $\PP^n$ at $p$.
 The latter group is not reductive. Note
that this generalizes Troyanov's obstruction to the existence of a
constant curvature metric on the teardrop ($S^2$ with one cone
point). 

In the case $\mathrm{(I.2.0})$, we have
$$
\Aut(\mathbb{F}_n,
Z_n)\cong\mathrm{PGL}_{2}(\mathbb{C})\times\Aut(\mathbb{C}^1)\cong\mathrm{PGL}_{2}(\mathbb{C})\times(\mathbb{G}_a\rtimes\mathbb{G}_m),%
$$
which is not reductive. 
 In the case $\mathrm{(I.2.n})$
with $n\ge 1$, we have $\Aut(\mathbb{F}_n,
Z_n)\cong\Aut(\mathbb{F}_n)$, because the curve $Z_n$ must be
fixed by any automorphism of $\mathbb{F}_n$ (since $n>0$). On the
other hand, it follows from \cite[Theorem~4.10]{ID} that if $n>0$,
then
$$
\Aut(\mathbb{F}_n)\cong\mathbb{G}_{a}^{n+1}\rtimes (\mathrm{GL}_{2}(\mathbb{C})/\mu_n),%
$$
where $\mathrm{GL}_{2}(\mathbb{C})/\mu_n$ 
acts on
$\mathbb{G}_{a}^{n+1}$ by means of its natural linear
representation in the space of binary forms of degree $n$. The
latter group is not reductive.

If $(S,C)$ is in $\mathrm{(I.6C.1})$, then
$\Aut(S,C)\cong\mathbb{G}_a^2\rtimes(\mathbb{G}_a\rtimes\mathbb{G}_m^2)$.
If $(S,C)$ is in $\mathrm{(I.6C.2})$, then
$\Aut_0(S,C)\cong\mathbb{G}_a^2\rtimes\mathbb{G}_m^2$. If $(S,C)$
is in $\mathrm{(I.6C.m})$ with $m\ge 3$, then
$\Aut_0(S,C)\cong\mathbb{G}_a^2\rtimes\mathbb{G}_m$. All these
groups are not reductive.

Now let us consider the case $\mathrm{(I.7.n.m})$ for $m>0$. If
$(S,C)$ is in $\mathrm{(I.7.0.1})$, then
$\Aut(S,C)\cong(\mathbb{G}_a\rtimes\mathbb{G}_m)\times
(\mathbb{G}_a\rtimes\mathbb{G}_m)$.  If $(S,C)$ is in
$\mathrm{(I.7.0.2})$, then $\Aut_0(S,C)\cong \mathbb{G}_m\times
(\mathbb{G}_a\rtimes\mathbb{G}_m)$. If $(S,C)$ is in
$\mathrm{(I.7.0.m})$ with $m\ge 3$, then $\Aut_0(S,C)\cong
\mathbb{G}_a\rtimes\mathbb{G}_m$. If $(S,C)$ is in
$\mathrm{(I.7.1.1})$, then $\Aut(S,C)\cong
\mathbb{G}_a^2\rtimes(\mathbb{G}_a\rtimes\mathbb{G}_m^2)$. If
$(S,C)$ is in $\mathrm{(I.7.1.2})$, then $\Aut_0(S,C)\cong
\mathbb{G}_a^2\rtimes\mathbb{G}_m^2$.  If $(S,C)$ is in
$\mathrm{(I.7.1.m})$ with $m\ge 3$, then $\Aut_0(S,C)\cong
\mathbb{G}_a^2\rtimes\mathbb{G}_m$.

If $(S,C)$ is in $\mathrm{(I.7.n.1})$ with $n\ge 2$, then
$$
\Aut(S,C)\cong\mathbb{G}_{a}^{n+1}\rtimes ((\mathbb{G}_a\rtimes\mathbb{G}_m^2)/\mu_n),%
$$
where
$((\mathbb{G}_a\rtimes\mathbb{G}_m^2)/\mu_n\subset\mathrm{GL}_{2}(\mathbb{C})/\mu_n$
acts on $\mathbb{G}_{a}^{n+1}$ by means of its natural linear
representation in the space of binary forms of degree $n$.
Similarly, if $(S,C)$ is in $\mathrm{(I.7.n.2})$ with $n\ge 2$,
then
$\Aut_0(S,C)\cong\mathbb{G}_{a}^{n+1}\rtimes\mathbb{G}_m^2/\mu_n$.
Finally, if $(S,C)$ is in $\mathrm{(I.7.n.m})$ with $n\ge 2$ and
$m\ge 3$, then
$\Aut_0(S,C)\cong\mathbb{G}_{a}^{n+1}\rtimes\mathbb{G}_m/\mu_n$.
All these groups are not reductive.

Now let us consider the case $\mathrm{(I.6B.m})$ with $m\ge
1$. If $(S,C)$ is in $\mathrm{(I.6B.1})$, then
$\Aut(S,C)\cong\mathbb{G}_{a}\rtimes\mathbb{G}_m^2$, which is not
reductive group. If $m=2$, then $\Aut_0(S,C)\cong\mathbb{G}_m$,
which is reductive. If $m\ge 3$, then $\Aut(S,C)$ is finite.

Now let us consider the case $\mathrm{(I.8B.m})$ with $m\ge
1$. If $m=1$, then
$\Aut(S,C)\cong\mathbb{G}_{a}\rtimes\mathbb{G}_m^2$, which is not
reductive group. If $m=2$, then $\Aut_0(S,C)\cong\mathbb{G}_m^2$,
which is reductive. If $m\ge 3$, then
$\Aut_0(S,C)\cong\mathbb{G}_m$, which is reductive.

Finally let us consider the case $\mathrm{(I.9C.m})$ with
$m\ge 1$. If $(S,C)$ is in $\mathrm{(I.9C.1})$, then
$\Aut_0(S,C)\cong\mathbb{G}_{a}\rtimes\mathbb{G}_m$, which is not
reductive group. If $m=2$, then $\Aut_0(S,C)\cong\mathbb{G}_m$,
which is reductive. If $m\ge 3$, then $\Aut(S,C)$ is finite. \epf

The following result shows that all pairs of class $(\beth)$
have reductive automorphism groups. This gives
further evidence for Conjecture \ref{UniformizationConj}.

\begin{theorem}
\label{BethReductiveGroupsThm}
Let $(S,C)$ be a pair of class $(\beth)$ with $C$ smooth and irreducible. Then $\Aut(S)$ is reductive.
\end{theorem}

\bpf If $(S,C)$ is (I.3A), then
$\mathrm{Aut}_{0}(S,C)\cong\mathbb{G}_m$ (this is easy). If
$(S,C)$ is (I.4B), then $\mathrm{Aut}_{0}(S,C)$ is a subgroup in
$\mathrm{PGL}_2(\mathbb{C})$ that fixes two points (the
ramification points of the double cover projection
$C\to\mathbb{P}^1$), which implies that $\mathrm{Aut}_{0}(S,C)$ is
either trivial or $\mathbb{G}_m$. Thus, if $(S,C)$ is (I.3A) or
(I.4B) then $\Aut(S)$ is reductive. Note that this also follows
from Theorem \ref{reductiveThm} combined with Theorem
\ref{KEEBethThm}.

Suppose $(S,C)$ is (I.9B.m) with $m\ge 1$. Then
$\mathrm{Aut}_{0}(S,C)$ preserves the conic bundle given by
$|-K_{S}-C|$ (see Proposition~\ref{proposition:conic-bundle}).
This implies that $\mathrm{Aut}_{0}(S,C)$ is a subgroup of the
group $\mathrm{Aut}_{0}(S^\prime,C^\prime)$ where
$(S^\prime,C^\prime)$ is a minimal model (see the proof of
Theorem~\ref{theorem:main-1}) of $(S,C)$ (it is either (I.3A) or
(I.4B)). Thus, we see that $\mathrm{Aut}_{0}(S,C)$ is a subgroup
of $\mathbb{G}_m$, which is either trivial of $\mathbb{G}_m$. In
particular, we see that $\mathrm{Aut}_{0}(S,C)$ is reductive. \epf

\subsection{Existence of KEE metrics on some pairs of class $\beth$}
\label{BethExistenceSubsec}

The goal of this subsection is to prove Theorem \ref{KEEBethThm}
as a first step towards
confirming Conjecture \ref{UniformizationConj}.
This gives the first examples of pairs with KEE metrics
of positive Ricci curvature which are not of class $(\aleph)$.
In \S\ref{SymmetrySubsubsec} we define the $G$-invariant
Tian invariant, with $G$ a finite group of automorphisms.
In the remainder of this subsection we then  
compute the Tian invariants of three 
pairs of class $\beth$. For the first two
(I.3A), (I.4B) the surface is fixed ($\FF_1$ or $\PP^1\times\PP^1$),
while for the third (I.9B.5) we 
specialize
to the Clebsch cubic surface.
The proofs use the results
of Section \ref{section:alpha} and Shokurov's
connectedness principle. Note that 
Proposition~\ref{lemma:cubic-surface} generalizes to
the logarithmic setting
the result that $\a(S)=2/3$ when $S$ is a cubic surface
in $\PP^3$ with an Eckardt point \cite{Ch07b}.
This result also serves to show (Example \ref{example:adjunction-nef})
that the bound of Proposition \ref{theorem:handy-proposition} cannot
hold without the nefness assumption. 

\subsubsection{Symmetry considerations}
\label{SymmetrySubsubsec}
Suppose that $X$ is acted by a finite group $G$ of automorphisms,
the divisor $B$ is $G$-invariant, and the class $[H]$ is $G$-invariant.
Then one can consider a $G$-invariant analogue of the
global lct of the pair $(X,B)$ with respect to
$[H]$.

\begin{definition}
\label{definition:global-threshold-G}
Let $G\subset \Aut(X)$. The $G$-invariant
global lct of the pair $(X,B)$ with respect to
$[H]$ is the number
$$
\alpha_{G}(X,B,[H]):=
\mathrm{inf}\Big\{
\mathrm{lct}\big(X,B;D\big):
\h{$D$ is effective $G$-invariant $\QQ$-divisor such that
$D\sim_{\QQ} H$}\Big\},%
$$
\end{definition}

For simplicity, we put $\alpha_{G}(X,[H])=\alpha_{G}(X,B,[H])$
if $B=0$. Similarly, we put $\alpha_{G}(X,
B_X)=\alpha_{G}(X,B,[H])$ if $H=-(K_{X}+B)$. Finally, we
put $\alpha_{G}(X)=\alpha_{G}(X,[H])$ if $B=0$ and $H=-K_{X}$.

\subsubsection{$\PP^1\times\PP^1$}

Let $G$ be a subgroup in $\mathrm{Aut}(\mathbb{P}^1)$ that is
isomorphic to $\mathrm{D}_{10}$ (the dihedral group of order
$10$). Then the action of $G$ is given by an irreducible
unimodular two-dimensional representation of the binary dihedral
group $2.G$ (a central extension of $G$ by $\mathbb{Z}_2$). Let us
denote this representation by $\mathbb{V}_2$ (we can identify it
with $H^0(\mathcal{O}_{\mathbb{P}^1}(1))$).

Note that the group $2.G$ has eight distinct irreducible
representations: the trivial one (which we denote by
$\mathbb{I}$), the two-dimensional representation $\mathbb{V}_2$,
three more two-dimensional representations (which we denote by
$\mathbb{V}_2^{\prime}$, $\mathbb{V}_2^{\prime\prime}$ and
$\mathbb{V}_2^{\prime\prime\prime}$, and three non-trivial
one-dimensional representations (which we denote by
$\mathbb{V}_1$, $\mathbb{V}_1^\prime$ and
$\mathbb{V}_1^{\prime\prime}$). Then
$\mathrm{Sym}^{3}(\mathbb{V}_2)\cong\mathbb{V}_2\oplus\mathbb{V}_2^{\prime\prime}$.
Moreover, one has
$$
\mathrm{Sym}^{6}(\mathbb{V}_2)\cong\mathbb{V}_1\oplus\mathbb{V}_2^{\prime}\oplus\mathbb{V}_2^{\prime}\oplus\mathbb{V}_2^{\prime\prime\prime},
$$
and
$\mathrm{Sym}^{2}(\mathrm{Sym}^{3}(\mathbb{V}_2))\cong\mathrm{Sym}^{6}(\mathbb{V}_2)\oplus\mathbb{V}_1\oplus\mathbb{V}_2^{\prime\prime\prime}$.
This follows from elementary representation theory.

Let $\phi\colon\mathbb{P}^1\to\mathbb{P}^3$ be an embedding given
by the linear system $|\mathcal{O}_{\mathbb{P}^1}(3)|$. Then
$\phi$ is $G$-equivariant. Put $C=\phi(\mathbb{P}^1)$. Then $C$ is
a smooth rational cubic curve in $\mathbb{P}^3$. Since $C$ is
projectively normal, we have an exact sequence of
$2.G$-representations
$$
0\to H^{0}(\mathcal{O}_{\mathbb{P}^3}(2)\otimes\mathcal{I}_{C})\to H^{0}(\mathcal{O}_{\mathbb{P}^3}(2))\to H^{0}(\mathcal{O}_{C}\otimes \mathcal{O}_{\mathbb{P}^3}(2))\to 0,%
$$
where
$H^{0}(\mathcal{O}_{\mathbb{P}^3}(2))\cong\mathrm{Sym}^{6}(\mathbb{V}_2)$
and $H^{0}(\mathcal{O}_{C}\otimes
\mathcal{O}_{\mathbb{P}^3}(2))\cong\mathrm{Sym}^{6}(\mathbb{V}_2)$.
This gives
$$
H^{0}(\mathcal{O}_{\mathbb{P}^3}(2)\otimes\mathcal{I}_{C})\cong\mathbb{V}_1\oplus\mathbb{V}_2^{\prime\prime\prime},
$$
which implies, in particular, that there exists unique
$G$-invariant quadric surface in $\mathbb{P}^3$ that contains the
curve $C$. Let us denote this quadric surface by $S$.

Since $C$ is not contained in a hyperplane in $\mathbb{P}^3$, the
surface $S$ is reduced and irreducible, Moreover, the surface $S$
is smooth, since $\mathrm{Sym}^{3}(\mathbb{V}_2)$ does not contain
one-dimensional subrepresentations of the group $2.G$. Then
$S\cong\mathbb{P}^1\times\mathbb{P}^1$ and $C$ is a curve of
bi-degree $(2,1)$ on $S$ so $(S,C)$ is $\mathrm{(I.4B})$.

\begin{prop}
\label{lemma:quadric} One has $\alpha_{G}(S,(1-\beta)C)=1$ for
every $\be\in(0,1]$.
\end{prop}

\begin{proof}
Note that $|-K_{S}-C|$ is a free pencil on $S$ that gives a
projection $S\to\mathbb{P}^1$ (cf.
Lemma~\ref{lemma:conic-bundles-2}). Let $Z_1$ be a curve in
$|-K_{S}-C|$, and let $Z_2,\ldots,Z_r$ be all curves in
$|-K_{S}-C|$ that are images of $Z_1$ via $G$. Then
$$
\frac1r{\sum_{i=1}^{r}Z_i}+\beta C\sim_{\mathbb{Q}} -(K_{S}+(1-\beta)C),%
$$
and $Z_1+Z_2+\cdots+Z_r$ is $G$-invariant. On the other hand, we
have
$$
\mathrm{lct}\big(S,(1-\beta)C;r^{-1}\h{$\sum_{i=1}^{r}Z_i$}+\beta C\big)\le 1,%
$$
so $\alpha(S,(1-\beta)C)\le 1$.

Suppose that $\alpha_{G}(S,(1-\beta)L)<1$. Then there exists an
effective $G$-invariant $\mathbb{Q}$-divisor $\Delta$ such that
$$
\Delta\sim_{\mathbb{R}} -(K_{S}+(1-\beta)C)
$$ and the pair
$(S,(1-\beta)C+\mu\Delta)$ is not lc at some point $O\in S$ for
some positive rational $\mu<1$.
We claim that $(S,(1-\beta)C+\mu\Delta)$ is lc outside of the
point  $O$. Indeed, suppose that this is not the case. Then
$(S,(1-\beta)C+\mu\Delta)$ is not lc along a curve. The latter
follows from the connectedness principle \cite[Lemma~5.7]{Sho93}
since the divisor $-K_{S}-(1-\beta)C-\mu\Delta)$ is ample, because
$\mu<1$. Thus, we see that there exists a $G$-invariant (possibly
reducible) curve $Z\subset S$ such that
$$
\Delta=\epsilon Z+\Omega
$$
for some effective $\mathbb{R}$-divisors $\Omega$ whose support
does not contain the curve $Z$ and some positive rational
$\epsilon$ such that either $Z=C$ and $\mu\epsilon>\beta$ or $Z\ne
C$ and $\mu\epsilon>1$. This is, of course, impossible, because
$\Delta\sim_{\mathbb{R}} -K_{S}-C+\beta C$. Indeed, if $Z=C$, then
$(\mu\epsilon-\beta)C+\Omega\sim_{\mathbb{R}} -K_{S}-C$, which
implies that
$$
0<2(\mu\epsilon-\beta)=(\mu\epsilon-\beta)C. (-K_{S}-C)\le \Big((\mu\epsilon-\beta)C+\Omega\Big). (-K_{S}-C)=(-K_{S}-C)^2=0,%
$$
which is absurd. Thus, we have $Z\ne C$. Then
$$
Z.(-K_{S}-C)\le \mu\epsilon Z.(-K_{S}-C)\le(\mu\epsilon Z+\Omega).(-K_{S}-C)=(-K_{S}-C+\beta C). (-K_{S}-C)=2\beta,%
$$
which implies that $Z.(-K_{S}-C)=0$. Then $Z\in |n(-K_{S}-C)|$
for some $n\in\mathbb{N}$. On the other hand, the pencil
$|-K_{S}-C|$ does not contain $G$-invariant curves (if
$|-K_{S}-C|$ contains a $G$-invariant curve, then $|-K_{S}|$
contains a $G$-invariant curve, which is impossible, since there
exists unique $G$-invariant quadric surface in $\mathbb{P}^3$ that
contains the curve $C$. Therefore, we see that $n\ge 2$. Then
$(n\mu\epsilon-1)(-K_{S}-C)+\Omega\sim_{\mathbb{R}}\beta C$, which
implies that
$$
2<(2n\mu\epsilon-1)\le
(2n\mu\epsilon-1)+\Omega.(-K_{S})=((n\mu\epsilon-1)(-K_{S}-C)+\Omega).(-K_{S})=\beta
C. (-K_{S})=6\beta
$$
which is impossible for small $\beta$. The obtained contradiction
shows that $(S,(1-\beta)C+\mu\Delta)$ is lc outside of the point
$O$.

Since $(1-\beta)C+\mu\Delta$ is $G$-invariant and the pair
$(S,(1-\beta)C+\mu\Delta)$ is lc outside of the point $O$, the
point $O$ must be $G$-invariant. The latter is impossible, since
$\mathrm{Sym}^{3}(V)$ does not contain one-dimensional
sub-representations. Thus $\alpha_{G}(S,(1-\beta)C)=1$.
\end{proof}

\subsubsection{$\mathbb{F}_1$}

Let $G$ be a subgroup in $\mathrm{Aut}(\mathbb{P}^1)$ that is
isomorphic to $\mathrm{D}_{2n}$ (the dihedral group of order $2n$)
for $n\ge 2$ (if $n=2$, then we assume that
$G\cong\mathbb{Z}_2\times\mathbb{Z}_2$). Then the action of $G$ is
given by an irreducible unimodular two-dimensional representation
of the group $2.G$ (a central extension of $G$ by $\mathbb{Z}_2$).
Let us denote this representation by $V$. Them
$\mathrm{Sym}^{2}(V)$ is a representation of the group $G$.
Moreover, it splits as a union of an irreducible two-dimensional
representation of $G$ and a one-dimensional subrepresentation.

Let $\phi\colon\mathbb{P}^1\to\mathbb{P}^2$ be an embedding given
by the linear system $|\mathcal{O}_{\mathbb{P}^1}(2)|$. Then
$\phi$ is $G$-equivariant. Put $\bar{C}=\phi(\mathbb{P}^1)$. Then
$\bar{C}$ is a smooth conic in $\mathbb{P}^2$. Moreover, there
exists $G$-invariant point $P\in\mathbb{P}^2$. Since $V$ is
irreducible representation of the group $2.G$, we see $P\not\in
\bar{C}$.

Let $\pi\colon S\to\mathbb{P}^2$ be the blow up of the point $P$.
Then the action of $G$ lifts to $S$ and $S\cong\mathbb{F}_1$.
Denote by $C$ the proper transform of the curve $\bar{C}$ on the
surface $S$. Thus, $(S,C)$ is $\mathrm{(I.3A})$.
The proof of the following result is almost identical to
the proof of Proposition~\ref{lemma:quadric}.

\begin{prop}
\label{lemma:F1} One has $\alpha_{G}(S,(1-\beta)C)=1$ for every
$\be\in(0,1]$.
\end{prop}

\subsubsection{Cubic surfaces}

The Tian invariant of a smooth cubic surface with an Eckardt point
is $2/3$ \cite[Theorem 1.7]{Ch07b}.
The following is a natural generalization.

\begin{prop}
\label{lemma:cubic-surface} Let $S$ be a smooth cubic surface in
$\mathbb{P}^{3}$, and let $C$ be a line on $S$. Then the divisor
$-(K_{S}+(1-\beta)C)$ is ample for every real $\beta\in(0,1]$.
Suppose that $C$ contains an Eckardt point. Then
$$
\alpha(S,(1-\beta)C)=\frac{1+\beta}{2+\beta}
$$
for every real $\beta\in(0,1]$.
\end{prop}

\begin{proof}
Let $P$ be an Eckardt point on $C$, let $L_1$ and $L_2$ be two
lines in $S$ such that $L_1\cap L_2\cap C=P$. Then
$$
\mathrm{lct}\Big(S,(1-\beta)C;L_1+L_2+\beta
C\Big)=\frac{1+\beta}{2+\beta}
$$
and $L_1+L_2+\beta C\sim_{\mathbb{Q}} -(K_{S}+(1-\beta)C)$, which
implies that $\alpha(S,(1-\beta)C)\le (1+\beta)/(2+\beta)$.

Suppose that $\alpha(S,(1-\beta)C)<(1+\beta)/(2+\beta)$. Then
there exists an effective $\mathbb{Q}$-divisor $\Delta$ on the
surface $S$ such that $\Delta\sim_{\mathbb{Q}}
-(K_{S}+(1-\beta)C)$ and the pair $(S,(1-\beta)C+\mu\Delta)$
is not lc at some point $O\in S$ for some positive
rational number $\mu<(1+\beta)/(2+\beta)$. Let us derive a
contradiction (compare the proofs of \cite[Lemmas~3.4 and 3.6]{Ch07b}).

Since $(1-\beta)C+\Delta\sim_{\mathbb{Q}}-K_{S}$, it follows from
\cite[Lemma~5.36]{CoKoSm} that the pair $(S,(1-\beta)C+\Delta)$ is
lc outside of  finitely many points in $C$. Hence, the
pair $(S,(1-\beta)C+\mu\Delta)$ is lc outside of the a
finitely many points in $S$, since $\mu\le 1$. In fact, this
implies that the log pair $(S,(1-\beta)C+\mu\Delta)$ is log
canonical outside of the point $O$ by the connectedness principle
\cite[Lemma~5.7]{Sho93}, because the divisor
$-(K_{S}+1-\beta)C+\mu\Delta)$ is ample.

If $O\not\in C$, then the pair $(S,\mu\Delta)$ is not log
canonical at the point $O\in S$, which is impossible, since
$\alpha(S)=2/3$ \cite[Theorem 1.7]{Ch07b} and
$\frac{1+\beta}{2+\beta}<2/3$. Thus, $O\in C$.

There exists a birational morphism $\pi\colon S\to\mathbb{P}^2$
such that $\pi$ is an isomorphism in a neighborhood of the point
$O$, and $\pi(C)$ is a line in $\mathbb{P}^{2}$. Put
${c}=\pi(C)$ and $\bar{\Delta}=\pi(\Delta)$. Then the pair
$(\mathbb{P}^2,(1-\beta){c}+\mu\bar{\Delta})$ is not log
canonical at the point $\pi(O)$. Moreover,  the pair
$(\mathbb{P}^2,(1-\beta){c}+\mu\bar{\Delta})$ is lc
outside of finitely many points in $\mathbb{P}^2$. Then
$(\mathbb{P}^2,(1-\beta){c}+\mu\bar{\Delta})$ is lc
outside of the point $\pi(O)$ by the connectedness principle,
because the divisor
$-(K_{\mathbb{P}^2}+(1-\beta){c}+\mu\bar{\Delta})$ is ample.

Let $L$ be a general line in $\mathbb{P}^2$. Then the pair
$(\mathbb{P}^2,(1-\beta){c}+\mu\bar{\Delta}+\epsilon L)$ is
not lc along $L$ for every rational number
$\epsilon>1$. Choose $\epsilon>1$ such that $\epsilon<1+3\beta$.
Then
$$
3-(1-\beta)-\mu\big(2+\beta\big)-\epsilon>3-2(1-\beta)-\frac{1+\beta}{2+\beta}\big(2+\beta\big)-\epsilon=1+3\beta-\epsilon>0,
$$
which implies that the divisor
$-(K_{\mathbb{P}^2}+(1-\beta){c}+\mu\bar{\Delta}+\epsilon L)$
is ample. This contradicts the connectedness principle, because
the pair
$(\mathbb{P}^2,(1-\beta){c}+\mu\bar{\Delta}+\epsilon L)$ is
not lc at every point of the non-connected set
$\pi(O)\not\in L$, and it is lc outside of this set.
\end{proof}

Proposition~\ref{lemma:cubic-surface} shows that the nefness
conditions in Theorem~\ref{theorem:handy-proposition} can not be
omitted as the following example demonstrates.

\begin{example}
\label{example:adjunction-nef}
{\rm Let $S$ be a smooth cubic surface
in $\mathbb{P}^{3}$, and let $C$ be a line on $X$ such that there
exists an Eckardt point on $C$. Put $H=-(K_{S}+(1-\beta)C)$ for
any $\beta\in(0,1)$. Then $H$ is ample. Put
$$
\gamma=\mathrm{sup}\big\{c\in\mathbb{Q}\ \big\vert\ H-cC\ \mathrm{is\ pseudoeffective}\big\}.%
$$
Then $\gamma=\beta$. Moreover, it follows from
Definition~\ref{definition:global-threshold} that
$\alpha(S,[H])\ge\a(S,[H+(1-\be)C])=\alpha(S)=2/3$. % (see \cite{Ch07b}).
But it follows from Lemma~\ref{lemma:lct-curve} that
$$
\alpha(C,[H]\vert_{C}))=\frac{1}{H. C}=\frac{1}{2-\beta}.%
$$
On the other hand, it follows from Proposition~\ref{lemma:cubic-surface}
that $\alpha(S,(1-\beta)C)=\frac{1+\beta}{2+\beta}$ for any
$\beta\in(0,1)$. Thus, we see that
$$
\alpha(S,(1-\beta)C,[H])=\alpha(S,(1-\beta)C)=\frac{1+\beta}{2+\beta}\not\ge\frac{1}{2-\beta}=\mathrm{min}
\{\beta/\gamma, \alpha(S,[H]), \alpha(C,[H]\vert_{C})\}%
$$
for sufficiently small $\beta>0$. Note that $C$ is not nef, since
$C^2=-1$ on the surface $S$.
}
\end{example}

%%%%%%%%%%%%%%%%%%%%%%%%%%%%%%%%%%%%%%%%%%%%%

Next, we show that for the Clebsch diagonal cubic surface Tian's
invariant is in fact equal to 1 for any $\be\in(0,1]$. Recall that
the Clebsch diagonal cubic surface is a smooth cubic surface with
$\mathrm{Aut}(S)=\mathrm{S}_5$ (see \cite[\S~4]{Hi08}). Such
surface exists and it is unique (this follows from basic
representation and invariant theory of the group $\mathrm{S}_5$).

\begin{prop}
\label{lemma:Clebsch-cubic-surface} Let $S$ be the Clebsch
diagonal cubic surface, i.e., the unique smooth cubic surface in
$\mathbb{P}^3$ such that $\mathrm{Aut}(S)\cong\mathrm{S}_5$. Let
$G\cong\mathrm{D}_{10}$ be a subgroup in $\mathrm{Aut}(S)$
consisting of even permutations. Then there exists a $G$-invariant
line $C\subset S$ and $\alpha_{G}(S,(1-\beta)C)=1$ for every
$\be\in(0,1]$.
\end{prop}

\begin{proof}
The surface $S$ can be obtained as $\mathrm{A}_5$-equivariant blow
up of $\mathbb{P}^2$ at the unique $\mathrm{A}_5$-orbit of length
$6$ (see \cite[\S~4]{Hi08} for details). Then the stabilizer in
$\mathrm{A}_5$ of any exceptional curve of this blow up is a
finite group isomorphic to $G$. Keeping in mind that all finite
subgroups in $\mathrm{A}_5$ that are  isomorphic to $G$ are
conjugate, we see that there exists a $G$-invariant line $C\subset
S$.

By Proposition \ref{proposition:conic-bundle} the linear system
$|-K_{S}-C|$ is a free pencil of conics on $S$. By our assumptions
this pencil is $G$-invariant. Let $Z_1$ be any curve in
$|-K_{S}-C|$, and let $Z_2,\ldots,Z_r$ be all curves in
$|-K_{S}-C|$ that are images of $Z_1$ via $G$. Then the divisor
$Z_1+Z_2+\cdots+Z_r$ is $G$-invariant and
$$
\frac1r{\sum_{i=1}^{r}Z_i}+\beta C\sim_{\mathbb{Q}} 
-K_{S}-(1-\beta)C.%
$$ 
On the other hand, 
$$
\mathrm{lct}\big(S,(1-\beta)C;r^{-1}\h{${\sum_{i=1}^{r}Z_i}$}+\beta C\big)\le 1,%
$$
so $\alpha(S,(1-\beta)C)\le 1$.

Suppose that $\alpha_{G}(S,(1-\beta)C)<1$. Then there exists an
effective $G$-invariant $\mathbb{Q}$-divisor $\Delta$ such that
$$
\Delta\sim_{\mathbb{Q}} -K_{S}-(1-\beta)C
$$ and the pair
$(S,(1-\beta)C+\mu\Delta)$ is not lc at some point $O\in S$ for
some positive rational $\mu<1$. Since
$(1-\beta)C+\Delta\sim_{\mathbb{Q}}-K_{S}$, it follows from
\cite[Lemma~5.36]{CoKoSm} that the pair $(S,(1-\beta)C+\Delta)$ is
lc outside of  finitely many points in $S$. Since
$\mu<1$, the divisor $-K_{S}-(1-\beta)C-\mu\Delta)$ is ample, and
thus the connectedness principle \cite[Lemma~5.7]{Sho93} implies
that the pair $(S,(1-\beta)C+\mu\Delta)$ is lc outside of the
point $O\in S$. In particular, this point $O$ must be
$G$-invariant.

On the other hand, the vector space $H^0(\mathcal{O}_{S}(-K_{S}))$
is a four-dimensional ($\chi_S(-K_S)=h^0(S,\calO_S(-K_S))=1+K_S^2=4$
\cite[p. 471]{GH}
since $S$ is a six-point blow-up of $\PP^2$)
representation of the group $G$ that splits
as a sum of two irreducible two-dimensional representations.
Hence, there exists no $G$-invariant point in $S$, since
otherwise $H^0(\mathcal{O}_{S}(-K_{S}))$ would contain a
one-dimensional sub-representation of $G$.
Thus $\alpha_{G}(S,(1-\beta)C)=1$.
\end{proof}

{\sc University of Edinburgh, Edinburgh, Scotland}

{\tt i.cheltsov@ed.ac.uk}

\smallskip

{\sc University of Maryland, College Park}

{\tt yanir@umd.edu}


\begin{thebibliography}{99}


\bibitem{Berm}
R.\,Berman, \emph{A thermodynamical formalism for Monge--Amp\`ere
equations, \hfill\break
Moser--Trudinger inequalities and K\"ahler--Einstein
metrics}, preprint, arxiv:1011.3976.

\bibitem{BCHM} C.\,Birkar, P.\,Cascini, C.\,Hacon, J.\,McKernan,
\emph{Existence of minimal models for varieties of log general
type}, J. Amer. Math. Soc. \textbf{23} (2010), 405--468.

\bibitem{Brown}
M.\,Brown, \emph{Singularities of Cox Rings of Fano Varieties},
preprint, arxiv:1109.6368.

\bibitem{Calabi1985} E. Calabi, \emph{Extremal {K}\"ahler metrics}, {II}, in:
Differential geometry and complex analysis
(I. Chavel, H. M. Farkas, Eds.), Springer, 1985, pp. {95--114}.

\bibitem{Ch07b}
I.\,Cheltsov, \emph{Log canonical thresholds of del Pezzo
surfaces},
Geom. Funct. Anal. \textbf{18} (2008), 1118--1144.%

\bibitem{CSD}
I.\,Cheltsov, C.\,Shramov, \emph{Log canonical thresholds of
smooth Fano threefolds}, with an appendix by J.-P. Demailly,
Russian Math. Surv. \textbf{63} (2008), 859--958.

\bibitem{CDSun} X. Chen, S. Donaldson, S. Song,
K\"ahler--Einstein metrics on Fano manifolds, III,
preprint, arxiv:1302.0282v1.

\bibitem{CoKoSm}
A.\,Corti, J.\,Koll\'ar, K.\,Smith, \emph{Rational and nearly
rational varieties}, Cambridge University Press, 2004.

\bibitem{Debarre} O. Debarre, \emph{Higher-dimensional
algebraic geometry}, Springer, 2001. 

\bibitem{delPezzo}
P. del Pezzo,
\emph{Sulle superficie delle $n^{\h{\sml mo}}$ ordine immerse nello spazio di
$n$ dimensioni},
Rend. del circolo matematico di Palermo {\bf 1} (1887), 241--271.

\bibitem{Demailly1993} J.-P. Demailly, \emph{A numerical criterion for very ample line bundles}, J.
Differential Geom. {\bf 37} (1993), 323--374.

\bibitem{DiCerbo2013}
G. \& L. Di Cerbo, \emph{Positivity questions in K\"ahler--Einstein theory},
preprint, \hfill\break
arxiv:1210.0218.

\bibitem{DiCerbo}
L.\, di Cerbo, \emph{On K\"ahler--Einstein surfaces with edge
singularities}, preprint,\hfill\break 
arxiv:1205.3203.

\bibitem{ID}
I. Dolgachev, V. Iskovskikh, \emph{Finite subgroups of the plane
Cremona group}, in: Algebra, Arithmetic and Geometry: In honor of
Y. I. Manin, Vol. 1, 
Springer, 2009, pp. 443–-549.


\bibitem{D}
S.K. Donaldson,
\emph{K\"ahler metrics with cone singularities along a divisor}, in: Essays on Mathematics and its applications (P.M. Pardalos et al., Eds.), Springer, 2012, pp. 49-79.


\bibitem{Fujino2000}
O.\,Fujino, \emph{Abundance theorem for semi log canonical
threefolds}, Duke Math. J. \textbf{102} (2000),
513--532.

\bibitem{Fujita2012}
K.\,Fujita, \emph{Simple normal crossing Fano varieties and log
Fano manifolds}, preprint, arxiv:1206.1994.


\bibitem{GOST}
Y.\,Gongyo, Sh.\,Okawa, A.\,Sannai, Sh.\,Takagi,
\emph{Characterization of varieties of Fano type via singularities
of Cox rings}, preprint, arxiv:1201.1133.

\bibitem{GH} P. Griffiths, J. Harris, \emph{Principles of algebraic geometry}, Wiley Interscience, 1978.

\bibitem{Harb}
B.\,Harbourne, \emph{Anticanonical rational surfaces},
Trans. Amer. Math. Soc. \textbf{349} (1997), 1191--1208.

\bibitem{Har77}
R.\,Hartshorne, \emph{Algebraic geometry},
Springer, 1977.

\bibitem{Hitchin} N. Hitchin, \emph{On the curvature of rational surfaces},
Proc. Symp. Pure Math. 27, Amer. Math. Soc., 1975, pp. 65--80.

\bibitem{Hi08}
N.\,Hitchin, \emph{Spherical harmonics and the~icosahedron},
in: Groups and symmetries,
Amer. Math. Soc., 2009, pp. 215--231.

\bibitem{IsPr99}
V. Iskovskikh, Yu. Prokhorov, \emph{Fano varieties}, Encyclopaedia
of Mathematical Sciences \textbf{47} (1999) Springer, Berlin.

\bibitem{JMR}
T. Jeffres, R. Mazzeo, Y.A. Rubinstein, \emph{K\"ahler--Einstein
metrics with edge singularities},
with an appendix by C. Li and Y.A. Rubinstein,
 preprint, arxiv:1105.5216.

\bibitem{Kobayashi} S. Kobayashi, \emph{On compact K\"ahler manifolds
of positive Ricci tensor}, Ann. of Math. 74 (1961), 570--574.

\bibitem{Kollar1993} J. Koll\'ar, \emph{Effective base point
freeness}, Math. Ann. {\bf 296} (1993), 595-605.

\bibitem{Ko97}
J.\,Koll\'ar, \emph{Singularities of pairs},
Proc. Symposia Pure Math. \textbf{62} (1997), 221--287.%

\bibitem{Kollar2005}
J.\,Kollar, \emph{Einstein metrics on five-dimensional Seifert
bundles}, J. Geom. Anal. \textbf{15} (2005),
445--476.

\bibitem{Kollar2007}
J.\,Kollar, \emph{Einstein metrics on connected sums of $S^2\times
S^3$}, J. Differential Geom. \textbf{75} (2007), 259--272.



\bibitem{KoMo98}
J.\,Koll\'ar, S.\,Mori, \emph{Birational geometry of algebraic
varieties},
Cambridge University Press, 1998.

\bibitem{Lazarsfeld} R. Lazarsfeld, \emph{Positivity in algebraic geometry}, Vol. I, II
Springer, 2004.

\bibitem{LiSun} C. Li, S. Sun,
\emph{Conical K\"ahler--Einstein metric revisited}, preprint,
arxiv:1207.5011v2.


\bibitem{Maeda}
H.\,Maeda, \emph{Classification of logarithmic threefolds},
Compositio Math. \textbf{57} (1986), 81--125.


\bibitem{Matsushima} Y. Matsushima,
\emph{Sur la structure du groupe d'hom\'eomorphismes analytiques
d'une certaine vari\'et\'e k\"ahl\'erienne},
Nagoya Math. J. \textbf{11} (1957), 145--150.

\bibitem{Mazzeo} R. Mazzeo,
{\it Elliptic theory of differential edge operators, I},  Comm. PDE
{\bf 16} (1991), 1616--1664.

\bibitem{MR}
R.\,Mazzeo, Y.A. Rubinstein, \emph{The Ricci continuity method for
the complex Monge-Ampere equation, with applications to
Kahler-Einstein edge metrics}, C. R. Math. Acad. Sci. Paris
\textbf{350} (2012), 693--697.

\bibitem{OdakaS} Y. Odaka, S. Sun,
\emph{Testing log K-stability by blowing up formalism},
preprint, \hfill\break
arxiv:1112.1353v1.

\bibitem{Petersen2nd} P. Petersen, \emph{Riemannian geometry}, 2nd Ed., Springer, 2006.

\bibitem{PrSh}
Yu.\,Prokhorov, V.\,Shokurov, \emph{Towards the second main
theorem on complements}, J. Algebraic Geom. \textbf{18} (2009),
151--199.

\bibitem{Sho93}
V.\,Shokurov, \emph{Three-fold log flips},
Russian Academy of Sciences, Izvestiya Mathematics \textbf{40} (1993), 95--202. %

\bibitem{SongTian} J. Song, G. Tian,
\emph{The K\"ahler-–Ricci flow on surfaces
of positive Kodaira dimension}, Invent. Math. {\bf 170} (2007), 609--653.

\bibitem{T87}
G.\,Tian, \emph{On K\"ahler--Einstein metrics on certain K\"ahler
manifolds with $c_{1}(M)>0$},
Invent. Math. \textbf{89} (1987), 225--246.%

\bibitem{Ti90}
G.\,Tian, \emph{On Calabi's conjecture for complex surfaces with
positive first Chern class}, Invent. Math. \textbf{101} (1990), 101--172.%

\bibitem{Tian1994}
G.\,Tian, \emph{K\"ahler--Einstein metrics on algebraic
manifolds}, Lecture Notes in Math. \textbf{1646} (1994),
143--185.

\bibitem{Tian2013}
G. Tian, {\it K-stability and K\"ahler--Einstein metrics,} preprint, 
arxiv:1211.4669v2.

\bibitem{TianYau}
G. Tian, S.-T. Yau,
\emph{Complete K\"ahler manifolds with
zero Ricci curvature, I}, J. Amer. Math. Soc. {\bf 3} (1990), 579--609.

\bibitem{Tsuji} H. Tsuji, \emph{Stability of tangent bundles of
minimal algebraic varieties} {\bf 27} (1988), 429--442.

\bibitem{Zhang}
Q. Zhang, \emph{Rational connectedness of log $\mathbb{Q}$-Fano
varieties}, Crelle's J. %urnal fur die reine und angewandte Mathematik
\textbf{590} (2006), 131-–142.


\end{thebibliography}
\end{document}